\documentclass[aps, 12pt, amsfonts, amsmath, amssymb, nofootinbib]{revtex4}
\usepackage{amsthm}
\usepackage{todonotes}
\usepackage{tikz}
\usetikzlibrary{positioning,matrix,arrows,shapes.misc}
\usepackage{mathtools}
\usepackage{hyperref}		
\usepackage{stmaryrd} 

\begin{document}
\newcommand{\N}{\mathbb{N}}
\newcommand{\Z}{\mathbb{Z}}
\newcommand{\Q}{\mathbb{Q}} 
\newcommand{\C}{\mathbb{C}}
\newcommand{\R}{\mathbb{R}}
\newcommand{\fld}{\Bbbk}
\newcommand{\op}[1]{\text{#1}}
\newcommand{\ot}{\otimes}
\newcommand{\no}{\nonumber}
\newcommand{\KK}{\mathfrak{K}}
\newcommand{\feyn}{\textbf{F}}
\newcommand{\encirc}[1]{\begin{tikzpicture}[baseline=(C.base)] \node[draw,circle,inner sep=0.5pt](C) {\footnotesize #1}; \end{tikzpicture}}
\newcommand{\encircs}[1]{\begin{tikzpicture}[baseline=(C.base)] \node[draw,circle,inner sep=0.1pt](C) {\tiny #1}; \end{tikzpicture}}
\newcommand{\oP}{\mathcal{P}}	
\newcommand{\cCo}{\ensuremath{\mathsf{Cor}}}	
\newcommand{\set}[2]{\left\{ #1 \ | \ #2 \right\} }	
\newcommand{\cyc}[1]{\left(\hspace{-.6ex}\left( #1 \right)\hspace{-.6ex}\right)}	
\newcommand{\cycu}[1]{\left(\hspace{-.3ex}\left( #1 \right)\hspace{-.3ex}\right)}	
\newcommand{\cycm}[1]{\left(\hspace{-.7ex}\left( #1 \right)\hspace{-.7ex}\right)}	
\newcommand{\cycb}[1]{\left(\hspace{-.9ex}\left( #1 \right)\hspace{-.9ex}\right)}	
\newcommand{\ooo}[2]{\sideset{_{#1}}{_{#2}}{\mathop{\circ}}}	
\newcommand{\dg}[1]{{\left| #1 \right|}}	
\newcommand{\id}{1}	
\newcommand{\oQC}{\mathcal{QC}}	
\newcommand{\oQO}{\mathcal{QO}}	
\newcommand{\oQOC}{\mathcal{QOC}}	
\newcommand{\oOC}{\mathcal{OC}}	
\newcommand{\oAss}{\mathcal{A}\mathit{ss}}	
\newcommand{\oCom}{\mathcal{C}\mathit{om}}	
\newcommand{\oMod}[1]{\textbf{Mod}\left(#1\right)}	
\newcommand{\Span}{\mathrm{Span}}	
\newcommand{\gr}[1]{\mathsf{#1}}	
\newcommand{\oT}{\mathcal{T}}	
\newcommand{\oEnd}[1]{{\mathcal{E}_{#1}}}	
\newcommand{\oooo}[4]{\prescript{#3}{#1}{\circ}^{#4}_{#2}}	
\newcommand{\cc}{\mathbf{c}}	
\newcommand{\cd}{\mathbf{d}}	
\newcommand{\co}{\mathbf{o}}	
\newcommand{\cobar}{\mathbf{C}}	
\newcommand{\comp}{\circ}	
\newcommand{\uo}{{\underline{o}}}	
\newcommand{\uc}{{\underline{c}}}	
\newcommand{\ob}{{\overline{b}}} 
\newcommand{\Stab}{{\textrm{Stab}}} 
\newcommand{\ooot}[2]{\sideset{_{#1}}{_{#2}}{\mathop{\overline{\circ}}}}	
\newcommand{\oxit}[1]{\overline{\xi}_{#1}}	
\newcommand{\cSg}{\mathsf{\Sigma}}	
\newcommand{\std}{\mathbf{\rho}}	
\newcommand{\oxi}[1]{\xi_{#1}}	
\newcommand{\cOp}{\mathsf{Op}}	
\newcommand{\Sh}[1]{\mathrm{Sh}(#1)}	
\newcommand{\oEndt}[1]{{\overline{\mathcal{E}}_{#1}}}	
\newcommand{\perm}[2]{\left[\begin{matrix} #1 \\ #2 \end{matrix} \right]}


\newcommand{\PICGraph}{
\begin{tikzpicture}[scale=0.5]
\draw (0,0) .. controls +(-2,1.5) and +(-2,-1.5) .. (0,0) .. controls +(2,1) .. (4,0);
\draw (0,0) .. controls +(2,-1) .. (4,0) -- (5,1);
\draw (4,0) -- (5,-1);
\draw (2,5) -- (2,4) -- (0,0);
\draw (2,4) -- (4,0);
\end{tikzpicture}\qquad\qquad\qquad
\begin{tikzpicture}[scale=0.5]
\draw (0,0) .. controls +(-2,1.5) and +(-2,-1.5) .. (0,0) node[below]{$\scriptscriptstyle{G_1}$} .. controls +(2,1) .. (4,0) node[below]{$\scriptscriptstyle{G_2}$};
\draw (0,0) .. controls +(2,-1) .. (4,0) -- (5,1) node[right]{$\scriptscriptstyle{l_2}$};
\draw (4,0) -- (5,-1) node[right]{$\scriptscriptstyle{l_3}$};
\draw (2,5) node[above]{$\scriptscriptstyle{l_1}$} -- (2,4) node[right]{$\scriptscriptstyle{G_3}$} -- (0,0);
\draw (2,4) -- (4,0);
\node at (-1,0.25) [above] {$\scriptscriptstyle{h_1}$}; 
\node at (-1,-0.25) [below] {$\scriptscriptstyle{h_2}$};
\node at (1,-0.25) [below] {$\scriptscriptstyle{h_3}$};
\node at (1,0.35) [above] {$\scriptscriptstyle{h_4}$};
\node at (0.6,1) [above left] {$\scriptscriptstyle{h_5}$};
\draw (2,0.625) -- (2,0.875);
\draw (2,-0.625) -- (2,-0.875);
\draw (-1.625,0) -- (-1.375,0);
\draw (0.9,2.1) -- (1.2,1.9);
\draw (2.9,1.9) -- (3.2,2.1);
\end{tikzpicture}
}

\newcommand{\PICKorbiA}{
\begin{tikzpicture}
\draw (0,0) circle (1);
\filldraw (0:1) circle (2pt) node[right]{$1$};
\filldraw (45:1) circle (2pt) node[above right]{$2$};
\filldraw (90:1) circle (2pt) node[above]{$3$};
\draw[line width=1pt, dotted] (160:1.2) arc (160:200:1.2);
\filldraw (-45:1) circle (2pt) node[below right]{$n$};
\end{tikzpicture}
}

\newcommand{\PICKorbiB}{
\begin{tikzpicture}
\draw (0,0) circle (1);
\filldraw (0:1) circle (2pt) node[right]{$1$};
\draw[line width=1pt, dotted] (25:1.2) arc (25:65:1.2);
\filldraw (90:1) circle (2pt) node[above]{};
\filldraw (135:1) circle (2pt) node[above]{};
\draw[line width=1pt, dotted] (160:1.2) arc (160:200:1.2);
\filldraw (225:1) circle (2pt) node[above]{};
\filldraw (270:1) circle (2pt) node[above]{};
\draw[line width=1pt, dotted] (-82.5:1.2) arc (-82.5:-52.5:1.2);
\filldraw (-45:1) circle (2pt) node[below right]{$n$};
\draw[dashed] (112.5:1.2) .. controls (-.13,0) .. (247.5:1.2);
\node at (-.55,0) {$C_1$};
\node at (0.4,0) {$C_2$};
\end{tikzpicture}
}

\newcommand{\PICA}{
\begin{tikzpicture}[baseline=-\the\dimexpr\fontdimen22\textfont2\relax, scale=0.5, yshift=-2.5cm]
\draw (0,0.5) node{$\scriptstyle{1}$} ellipse (0.25 and 0.5);
\draw (0,0) .. controls +(2,0) .. (3,1) .. controls +(1,.75) .. (6,2);
\draw (6,2.5) node{$\scriptstyle{2}$} ellipse (0.25 and 0.5);
\draw (0,5) .. controls +(2,0) .. (3,4) .. controls +(1,-.75) .. (6,3);
\draw (0,4.5) node{$\scriptstyle{3}$} ellipse (0.25 and 0.5);
\draw (0,1) .. controls +(2,0) and +(2,0) .. (0,4);
\draw (-0.217+2.5,2.5) arc (210:330:0.5);
\draw (-.13+2.5,-.115+2.5) arc (150:30:0.4);
\end{tikzpicture}
}


\newcommand{\PICB}{
\begin{tikzpicture}[baseline=-\the\dimexpr\fontdimen22\textfont2\relax, scale=0.5, yshift=-2.5cm]
\draw (0,0.5) ellipse (0.25 and 0.5);
\draw (0,0) .. controls +(2,0) .. (3,1) .. controls +(1,.75) .. (6,2);
\draw (6,2.5) ellipse (0.25 and 0.5);
\draw (0,5) .. controls +(2,0) .. (3,4) .. controls +(1,-.75) .. (6,3);
\draw (0,4.5) ellipse (0.25 and 0.5);
\draw (0,1) .. controls +(2,0) and +(2,0) .. (0,4);
\draw (-0.217+2.5,2.5) arc (210:330:0.5);
\draw (-.13+2.5,-.115+2.5) arc (150:30:0.4);

\draw[dashed] (6,2)--(7,2);
\draw[dashed] (6,3)--(7,3);

\draw (12.5,2.5) .. controls +(0,1.5) and +(1,1) .. (8.5,3.5) .. controls +(-1,-0.5) .. (7,3);
\draw (12.5,2.5) .. controls +(0,-1.5) and +(1,-1) .. (8.5,1.5) .. controls +(-1,0.5) .. (7,2);
\draw (7,2.5) ellipse (0.25 and 0.5);
\draw (-0.217+9,2.5) arc (210:330:0.5);
\draw (-.13+9,-.115+2.5) arc (150:30:0.4);
\draw (-0.217+11,2.5) arc (210:330:0.5);
\draw (-.13+11,-.115+2.5) arc (150:30:0.4);
\end{tikzpicture}
}


\newcommand{\PICC}{
\begin{tikzpicture}[baseline=-\the\dimexpr\fontdimen22\textfont2\relax, scale=0.5, yshift=-2.5cm]
\draw (0,0.5) ellipse (0.25 and 0.5);
\draw (0,0) .. controls +(1,0) .. (3,1) .. controls +(1,.5) .. (6,1.5) .. controls +(1.5,0) and +(1.5,0) .. (6,3.5);
\draw (0,5) .. controls +(2,0) .. (3,4) .. controls +(1,-.5) .. (6,3.5);
\draw (0,4.5) ellipse (0.25 and 0.5);
\draw (0,1) .. controls +(2,0) and +(2,0) .. (0,4);
\draw (-0.217+2.5,2.5) arc (210:330:0.5);
\draw (-.13+2.5,-.115+2.5) arc (150:30:0.4);
\draw (-0.217+4,2.5) arc (210:330:0.5);
\draw (-.13+4,-.115+2.5) arc (150:30:0.4);
\draw (-0.217+5.5,2.5) arc (210:330:0.5);
\draw (-.13+5.5,-.115+2.5) arc (150:30:0.4);
\end{tikzpicture}
}


\newcommand{\PICD}{
\begin{tikzpicture}[baseline=-\the\dimexpr\fontdimen22\textfont2\relax, scale=0.5, yshift=-2.5cm]
\draw[dashed, double distance=.5cm] (0,.5)..controls+(-3,0) and +(-3,0)..(0,4.5);
\draw (0,0.5) ellipse (0.25 and 0.5);
\draw (0,0) .. controls +(2,0) .. (3,1) .. controls +(1,.75) .. (6,2);
\draw (6,2.5) ellipse (0.25 and 0.5);
\draw (0,5) .. controls +(2,0) .. (3,4) .. controls +(1,-.75) .. (6,3);
\draw (0,4.5) ellipse (0.25 and 0.5);
\draw (0,1) .. controls +(2,0) and +(2,0) .. (0,4);
\draw (-0.217+2.5,2.5) arc (210:330:0.5);
\draw (-.13+2.5,-.115+2.5) arc (150:30:0.4);
\end{tikzpicture}
}


\newcommand{\PICE}{
\begin{tikzpicture}[baseline=-\the\dimexpr\fontdimen22\textfont2\relax, scale=0.5, yshift=-2.5cm]
\draw (-12.5,2.5) .. controls +(0,1.5) and +(-1,1) .. (-8.5,3.5) .. controls +(1,-0.5) .. (-7,3);
\draw (-12.5,2.5) .. controls +(0,-1.5) and +(-1,-1) .. (-8.5,1.5) .. controls +(1,0.5) .. (-7,2);
\draw (-7,2.5) ellipse (0.25 and 0.5);
\draw (-0.217-9.5,2.5) arc (210:330:0.5);
\draw (-.13-9.5,-.115+2.5) arc (150:30:0.4);
\draw (-0.217-11.5,2.5) arc (210:330:0.5);
\draw (-.13-11.5,-.115+2.5) arc (150:30:0.4);
\end{tikzpicture}
}


\newcommand{\PICF}{
\begin{tikzpicture}[baseline=-\the\dimexpr\fontdimen22\textfont2\relax, scale=0.5, yshift=-2.5cm]
\draw (0,0.5) node{$\scriptstyle{a}$} ellipse (0.25 and 0.5);
\draw (0,0) .. controls +(2,0) .. (3,1) .. controls +(1,.75) .. (6,2);
\draw (6,2.5) node{$\scriptstyle{c}$} ellipse (0.25 and 0.5);
\draw (0,5) .. controls +(2,0) .. (3,4) .. controls +(1,-.75) .. (6,3);
\draw (0,4.5) node{$\scriptstyle{b}$} ellipse (0.25 and 0.5);
\draw (0,1) .. controls +(2,0) and +(2,0) .. (0,4);
\end{tikzpicture}
}


\newcommand{\PICG}{
\begin{tikzpicture}[baseline=-\the\dimexpr\fontdimen22\textfont2\relax, scale=0.5, yshift=-2.5cm]
\draw (-12.5,2.5) .. controls +(0,1.5) and +(-1,1) .. (-8.5,3.5) .. controls +(1,-0.5) .. (-7,3);
\draw (-12.5,2.5) .. controls +(0,-1.5) and +(-1,-1) .. (-8.5,1.5) .. controls +(1,0.5) .. (-7,2);
\draw (-7,2.5) node{$\scriptstyle{c}$} ellipse (0.25 and 0.5);
\draw (-0.217-10.5,2.5) arc (210:330:0.5);
\draw (-.13-10.5,-.115+2.5) arc (150:30:0.4);
\end{tikzpicture}
}


\newcommand{\PICH}{
\begin{tikzpicture}[baseline=-\the\dimexpr\fontdimen22\textfont2\relax, scale=0.5]
\draw (0,0) circle (3cm);
\node[cross out,draw] at (0:3) {};
\node at (0:3) [right]{$a$};
\node[cross out,draw] at (120:3) {};
\node at (120:3) [above left]{$b$};
\node[cross out,draw] at (240:3) {};
\node at (240:3) [below left]{$c$};
\end{tikzpicture}
}


\newcommand{\PICI}{
\begin{tikzpicture}[baseline=-\the\dimexpr\fontdimen22\textfont2\relax, scale=0.5]
\draw (0,0) circle (3cm);
\draw (-5:3)--(-5:4)--(5:4)--(5:3);
\node at (0:4) [right]{$a$};
\draw (115:3)--(115:4)--(125:4)--(125:3);
\node at (120:4) [above left]{$b$};
\draw (235:3)--(235:4)--(245:4)--(245:3);
\node at (240:4) [below left]{$c$};
\end{tikzpicture}
}


\newcommand{\PICQOA}{
\begin{tikzpicture}[baseline=-\the\dimexpr\fontdimen22\textfont2\relax, yshift=-2.5cm]
\draw (-7,1) ellipse (.5 and 1);
\draw (-7,0)--(-4,0);
\draw (-4,1) ellipse (.5 and 1);
\draw (-7,2)--(-4,2);

\draw [fill=white] (-7.3,.667) rectangle  +(-.5,-.333);
\draw [fill=white] (-7.3,1.333) rectangle  +(-.5,.333);
\draw [white,draw opacity=0,fill=white] (-7,1) ellipse (.5 and 1);

\draw [fill=white] (-3.7,1.417) rectangle  +(.5,.333);
\draw [fill=white] (-3.53,.833) rectangle  +(.5,.333);
\draw [fill=white] (-3.7,.25) rectangle  +(.5,.333);
\draw [white,draw opacity=0,fill=white] (-4,1) ellipse (0.5 and 1);

\node at (-3.2,1.584) [left]{$\scriptscriptstyle{b_1}$};
\node at (-3.03,1) [left]{$\scriptscriptstyle{b}$};
\node at (-3.2,.417) [left]{$\scriptscriptstyle{b_2}$};
\node at (-7.8,.506) [right]{$\scriptscriptstyle{b_4}$};
\node at (-7.8,1.494) [right]{$\scriptscriptstyle{b_3}$};

\draw [dashed] (-3.03,.833)--(-.97,.833);
\draw [dashed] (-3.03,1.167)--(-.97,1.167);

\draw (0,1) ellipse (0.5 and 1);
\draw (0,0) .. controls +(2,0) .. (3,1) .. controls +(1,.75) .. (6,2);
\draw (6,2.5) ellipse (0.25 and 0.5);
\draw (0,5) .. controls +(2,0) .. (3,4) .. controls +(1,-.75) .. (6,3);
\draw (0,4.5) ellipse (0.25 and 0.5);
\draw (0,2) .. controls +(2,0) and +(2,0) .. (0,4);
\draw (-0.217+2.5,2.5) arc (210:330:0.5);
\draw (-.13+2.5,-.115+2.5) arc (150:30:0.4);

\draw [fill=white] (-.3,1.417) rectangle  +(-.5,.333);
\draw [fill=white] (-.47,.833) rectangle  +(-.5,.333);
\draw [fill=white] (-.3,.25) rectangle  +(-.5,.333);
\draw [white,draw opacity=0,fill=white] (0,1) ellipse (0.5 and 1);

\draw [fill=white] (6.15,2.333) rectangle  +(.5,.333);
\draw [white,draw opacity=0,fill=white] (6,2.5) ellipse (.25 and .5);

\node at (-.8,1.584) [right]{$\scriptscriptstyle{a_1}$};
\node at (-.97,1) [right]{$\scriptscriptstyle{a}$};
\node at (-.8,.417) [right]{$\scriptscriptstyle{a_2}$};
\node at (6.65,2.5) [left]{$\scriptscriptstyle{a_3}$};
\end{tikzpicture}
}


\newcommand{\PICQOB}{
\begin{tikzpicture}[baseline=-\the\dimexpr\fontdimen22\textfont2\relax, yshift=-2.5cm]
\draw (-5,0) -- (0,0) .. controls +(2,0) .. (3,1) .. controls +(1,.75) .. (6,2);
\draw (-2,3) ellipse (1 and .5);
\draw (-3,3) .. controls +(0,-1) .. (-4,2) -- (-5,2);
\draw (-5,1) ellipse (.5 and 1);
\draw (0,5) .. controls +(2,0) .. (3,4) .. controls +(1,-.75) .. (6,3);
\draw (0,4.5) ellipse (0.25 and 0.5);
\draw (-1,3) .. controls +(0,-1) .. (0,2) .. controls +(2,0) and +(2,0) .. (0,4);
\draw (6,2.5) ellipse (0.25 and 0.5);
\draw (-0.217+2.5,2.5) arc (210:330:0.5);
\draw (-.13+2.5,-.115+2.5) arc (150:30:0.4);

\draw [fill=white] (-1.133,3.2) rectangle  +(-.333,.5);
\draw [fill=white] (-1.6,3.35) rectangle  +(-.333,.5);
\draw [fill=white] (-2.067,3.35) rectangle  +(-.333,.5);
\draw [fill=white] (-2.534,3.2) rectangle  +(-.333,.5);
\draw [white,draw opacity=0,fill=white] (-2,3) ellipse (1 and .5);

\draw [fill=white] (6.15,2.333) rectangle  +(.5,.333);
\draw [white,draw opacity=0,fill=white] (6,2.5) ellipse (.25 and .5);

\draw [fill=white] (-5.3,.667) rectangle  +(-.5,-.333);
\draw [fill=white] (-5.3,1.333) rectangle  +(-.5,.333);
\draw [white,draw opacity=0,fill=white] (-5,1) ellipse (.5 and 1);

\node at (6.65,2.5) [left] {$\scriptscriptstyle{a_3}$};
\node at (-5.8,.506) [right] {$\scriptscriptstyle{b_4}$};
\node at (-5.8,1.494) [right] {$\scriptscriptstyle{b_3}$};
\node at (-1.3,3.7) [below] {$\scriptscriptstyle{a_1}$};
\node at (-1.737,3.85) [below] {$\scriptscriptstyle{b_1}$};
\node at (-2.234,3.85) [below] {$\scriptscriptstyle{b_2}$};
\node at (-2.701,3.7) [below] {$\scriptscriptstyle{a_2}$};
\end{tikzpicture}
}


\newcommand{\PICQOC}{
\begin{tikzpicture}[baseline=-\the\dimexpr\fontdimen22\textfont2\relax, yshift=-1.5cm]
\draw (-7,1) ellipse (.5 and 1);
\draw (-7,0)--(-4,0);
\draw (-4,1) ellipse (.5 and 1);
\draw (-7,2)--(-4,2);

\draw [fill=white] (-7.3,.667) rectangle  +(-.5,-.333);
\draw [fill=white] (-7.3,1.333) rectangle  +(-.5,.333);
\draw [white,draw opacity=0,fill=white] (-7,1) ellipse (.5 and 1);

\draw [fill=white] (-3.7,1.417) rectangle  +(.5,.333);
\draw [fill=white] (-3.53,.833) rectangle  +(.5,.333);
\draw [fill=white] (-3.7,.25) rectangle  +(.5,.333);
\draw [white,draw opacity=0,fill=white] (-4,1) ellipse (0.5 and 1);

\node at (-3.2,1.584) [left]{$\scriptscriptstyle{b_1}$};
\node at (-3.03,1) [left]{$\scriptscriptstyle{b}$};
\node at (-3.2,.417) [left]{$\scriptscriptstyle{b_2}$};
\node at (-7.8,.506) [right]{$\scriptscriptstyle{a}$};
\node at (-7.8,1.494) [right]{$\scriptscriptstyle{b_3}$};

\draw [dashed, double distance=.33cm] (-7.8,.5)..controls+(-2,0)and+(-2,0)..(-7.8,3)--(-3.03,3)..controls+(2,0)and+(2,0)..(-3.03,1);
\end{tikzpicture}
}


\newcommand{\PICQOD}{
\begin{tikzpicture}[baseline=-\the\dimexpr\fontdimen22\textfont2\relax, yshift=-2.5cm]
\draw (12.5,2.5) .. controls +(0,1.5) and +(1,1) .. (8.5,4) .. controls +(-1,-0.5) .. (7,3.5);
\draw (12.5,2.5) .. controls +(0,-1.5) and +(1,-1) .. (8.5,1) .. controls +(-1,0.5) .. (7,1.5);
\draw (7,2.5) ellipse (.5 and 1);
\draw (-0.217+10,2.5) arc (210:330:0.5);
\draw (-.13+10,-.115+2.5) arc (150:30:0.4);
\draw [fill=white] (6.7,2.917) rectangle  +(-.5,.333);
\draw [fill=white] (6.53,2.333) rectangle  +(-.5,.333);
\draw [fill=white] (6.7,1.75) rectangle  +(-.5,.333);
\draw [white,draw opacity=0,fill=white] (7,2.5) ellipse (.5 and 1);

\node at (6.2,3.084) [right]{$\scriptscriptstyle{b_2}$};
\node at (6.03,2.5) [right]{$\scriptscriptstyle{b_1}$};
\node at (6.2,1.917) [right]{$\scriptscriptstyle{b_3}$};
\end{tikzpicture}
}


\newcommand{\PICQOE}{
\begin{tikzpicture}[baseline=-\the\dimexpr\fontdimen22\textfont2\relax, yshift=-1cm]
\draw (-7,1) ellipse (.5 and 1);
\draw (-7,0)--(-4,0);
\draw (-4,1) ellipse (.5 and 1);
\draw (-7,2)--(-4,2);

\draw [fill=white] (-7.3,.667) rectangle  +(-.5,-.333);
\draw [fill=white] (-7.3,1.333) rectangle  +(-.5,.333);
\draw [white,draw opacity=0,fill=white] (-7,1) ellipse (.5 and 1);

\draw [fill=white] (-3.7,1.417) rectangle  +(.5,.333);
\draw [fill=white] (-3.53,.833) rectangle  +(.5,.333);
\draw [fill=white] (-3.7,.25) rectangle  +(.5,.333);
\draw [white,draw opacity=0,fill=white] (-4,1) ellipse (0.5 and 1);

\node at (-3.2,1.584) [left]{$\scriptscriptstyle{a}$};
\node at (-3.03,1) [left]{$\scriptscriptstyle{b_1}$};
\node at (-3.2,.417) [left]{$\scriptscriptstyle{b}$};
\node at (-7.8,.506) [right]{$\scriptscriptstyle{b_3}$};
\node at (-7.8,1.494) [right]{$\scriptscriptstyle{b_2}$};

\draw [dashed, double distance=.33cm] (-3.2,.417)..controls+(.75,0)and+(.75,0)..(-3.2,1.584);
\end{tikzpicture}
}


\newcommand{\PICQOF}{
\begin{tikzpicture}[baseline=-\the\dimexpr\fontdimen22\textfont2\relax, yshift=-2.5cm]
\draw (-7,2.5) ellipse (.5 and 1);
\draw (-7,1.5)..controls+(2,0)..(-4,.5)..controls+(1,-.5)..(-1,0);
\draw (-1,.5) ellipse (0.25 and 0.5);
\draw (-7,3.5)..controls+(2,0)..(-4,4.5)..controls+(1,.5)..(-1,5);
\draw (-1,4.5) ellipse (0.25 and 0.5);
\draw (-1,4)..controls+(-2,0)and+(-2,0)..(-1,1);

\draw [fill=white] (-7.3,3.167) rectangle  +(-.5,-.333);
\draw [fill=white] (-7.3,1.833) rectangle  +(-.5,.333);
\draw [white,draw opacity=0,fill=white] (-7,2.5) ellipse (.5 and 1);

\draw [fill=white] (-.85,.333) rectangle  +(.5,.333);
\draw [white,draw opacity=0,fill=white] (-1,.5) ellipse (0.25 and 0.5);

\node at (-7.8,2.006) [right]{$\scriptscriptstyle{b_3}$};
\node at (-7.8,2.996) [right]{$\scriptscriptstyle{b_2}$};

\node at (-.35,.5) [left]{$\scriptscriptstyle{b_1}$};
\end{tikzpicture}
}


\newcommand{\PICABGa}
{
\begin{tikzpicture}[baseline=-\the\dimexpr\fontdimen22\textfont2\relax]
\draw (0:0) circle (1);
\filldraw (240:1) circle (2pt) node[below left]{$\phi^{I_{n_1+1}}$};
\filldraw (270:1) circle (2pt) node[below]{$\phi^{I_{1}}$};
\filldraw (300:1) circle (2pt) node[below right]{$\phi^{I_{2}}$};
\draw[line width=1pt, loosely dotted] (-45:1.2) arc[radius=1.2cm,start angle=-45, end angle=225];
\end{tikzpicture}
}


\newcommand{\PICABGax}
{
\begin{tikzpicture}[baseline=-\the\dimexpr\fontdimen22\textfont2\relax]
\draw (0:0) circle (1);
\filldraw (240:1) circle (2pt) node[below left]{$\scriptscriptstyle n_1+1$};
\filldraw (270:1) circle (2pt) node[below]{$\scriptscriptstyle 1$};
\filldraw (300:1) circle (2pt) node[below right]{$\scriptscriptstyle 2$};
\draw[line width=1pt, loosely dotted] (-45:1.2) arc[radius=1.2cm,start angle=-45, end angle=225];
\end{tikzpicture}
}


\newcommand{\PICABGb}
{
\begin{tikzpicture}[baseline=-\the\dimexpr\fontdimen22\textfont2\relax]
\draw (0:0) circle (1);
\filldraw (60:1) circle (2pt) node[above right]{$\phi^{J_{n_2+1}}$};
\filldraw (90:1) circle (2pt) node[above]{$\phi^{J_1}$};
\filldraw (120:1) circle (2pt) node[above left]{$\phi^{J_{2}}$};
\draw[line width=1pt, loosely dotted] (135:1.2) arc[radius=1.2cm,start angle=135, end angle=405];
\end{tikzpicture}
}


\newcommand{\PICABGbx}
{
\begin{tikzpicture}[baseline=-\the\dimexpr\fontdimen22\textfont2\relax]
\draw (0:0) circle (1);
\filldraw (60:1) circle (2pt) node[above right]{$\scriptscriptstyle n_2+1$};
\filldraw (90:1) circle (2pt) node[above]{$\scriptscriptstyle 1$};
\filldraw (120:1) circle (2pt) node[above left]{$\scriptscriptstyle 2$};
\draw[line width=1pt, loosely dotted] (135:1.2) arc[radius=1.2cm,start angle=135, end angle=405];
\end{tikzpicture}
}


\newcommand{\PICABG}
{
\begin{tikzpicture}[baseline=-\the\dimexpr\fontdimen22\textfont2\relax]
\draw (20:1) arc[radius=1cm,start angle=20, end angle=340];
\filldraw (30:1) circle (2pt) node[above right]{$\phi^{I_{i+1}}$};
\filldraw (240:1) circle (2pt) node[below left]{$\phi^{I_{n_1+1}}$};
\filldraw (270:1) circle (2pt) node[below]{$\phi^{I_{1}}$};
\filldraw (330:1) circle (2pt) node[below right]{$\phi^{I_{i-1}}$};
\draw[line width=1pt, loosely dotted] (45:1.2) arc[radius=1.2cm,start angle=45, end angle=225];
\draw[line width=1pt, loosely dotted] (285:1.2) arc[radius=1.2cm,start angle=285, end angle=315];
\draw (3,0) +(160:1) node (mynode){};
\draw (20:1) .. controls +(0,-.3) and +(0,-.3) .. (mynode.center);
\draw (3,0) +(-160:1) node (mynode2){};
\draw (-20:1) .. controls +(0,.3) and +(0,.3) .. (mynode2.center);
\draw (3,0) +(-160:1) arc[radius=1cm,start angle=-160, end angle=160];
\filldraw (3,0) +(90:1) circle (2pt) node[above]{$\phi^{J_1}$};
\filldraw (3,0) +(150:1) circle (2pt) node[right]{$\phi^{J_{j-1}}$};
\filldraw (3,0) +(210:1) circle (2pt) node[right]{$\phi^{J_{j+1}}$};
\filldraw (3,0) +(60:1) circle (2pt) node[above right]{$\phi^{J_{n_2+1}}$};
\draw[line width=1pt, loosely dotted] (3,0) +(105:1.2) arc[radius=1.2cm,start angle=105, end angle=135];
\draw[line width=1pt, loosely dotted] (3,0) +(225:1.2) arc[radius=1.2cm,start angle=-135, end angle=45];
\end{tikzpicture}
}


\newcommand{\PICABGx}
{
\begin{tikzpicture}[baseline=-\the\dimexpr\fontdimen22\textfont2\relax]
\draw (20:1) arc[radius=1cm,start angle=20, end angle=340];
\filldraw (30:1) circle (2pt) node[above right]{$\scriptscriptstyle i+1-1$};
\filldraw (240:1) circle (2pt) node[below left]{$\scriptscriptstyle n_1+1-1$};
\filldraw (270:1) circle (2pt) node[below]{$\scriptscriptstyle 1$};
\filldraw (330:1) circle (2pt) node[below right]{$\scriptscriptstyle i-1$};
\draw[line width=1pt, loosely dotted] (45:1.2) arc[radius=1.2cm,start angle=45, end angle=225];
\draw[line width=1pt, loosely dotted] (285:1.2) arc[radius=1.2cm,start angle=285, end angle=315];
\draw (3,0) +(160:1) node (mynode){};
\draw (20:1) .. controls +(0,-.3) and +(0,-.3) .. (mynode.center);
\draw (3,0) +(-160:1) node (mynode2){};
\draw (-20:1) .. controls +(0,.3) and +(0,.3) .. (mynode2.center);
\draw (3,0) +(-160:1) arc[radius=1cm,start angle=-160, end angle=160];
\filldraw (3,0) +(90:1) circle (2pt) node[above]{$\scriptscriptstyle 1+n_1$};
\filldraw (3,0) +(150:1) circle (2pt) node[right]{$\scriptscriptstyle j-1+n_1$};
\filldraw (3,0) +(210:1) circle (2pt) node[right]{$\scriptscriptstyle j+1+n_1-1$};
\filldraw (3,0) +(60:1) circle (2pt) node[above right]{$\scriptscriptstyle n_2+1+n_1-1$};
\draw[line width=1pt, loosely dotted] (3,0) +(105:1.2) arc[radius=1.2cm,start angle=105, end angle=135];
\draw[line width=1pt, loosely dotted] (3,0) +(225:1.2) arc[radius=1.2cm,start angle=-135, end angle=45];
\end{tikzpicture}
}

\theoremstyle{plain}
\newtheorem{theorem}{Theorem}
\newtheorem{lemma}[theorem]{Lemma}
\newtheorem{proposition}[theorem]{Proposition}
\newtheorem{corollary}[theorem]{Corollary}
\theoremstyle{definition}
\newtheorem{definition}[theorem]{Definition}
\theoremstyle{remark}
\newtheorem{example}[theorem]{Example}
\newtheorem{remark}[theorem]{Remark}
\newtheorem{convention}[theorem]{Convention}

\title{Modular operads and the quantum open-closed homotopy algebra\vspace{1.0cm}}

\author{Martin Doubek}
\email{martindoubek@seznam.cz}
\affiliation{Charles University, Faculty of Mathematics and Physics, Sokolovsk\'a 83, 186 75 Prague, Czech Republic}
\author{Branislav Jur\v co}
\email{branislav.jurco@gmail.com}
\affiliation{Charles University, Faculty of Mathematics and Physics, Sokolovsk\'a 83, 186 75 Prague, Czech Republic}
\author{Korbinian M\"unster}
\email{korbinian.muenster@physik.uni-muenchen.de} 
\affiliation{Arnold Sommerfeld Center for Theoretical Physics, Theresienstrasse 37, D-80333 Munich, Germany}

\begin{abstract}
\vspace{1.7cm}
\centerline{\bf Abstract \vspace{0.4cm}}
We verify that certain algebras appearing in string field theory are algebras over Feynman transform of modular operads which we describe explicitly.
Equivalent description in terms of solutions of generalized BV master equations are explained from the operadic point of view.
\end{abstract}

\maketitle
\thispagestyle{empty}
\newpage
\tableofcontents
\newpage


%
%

\section{Introduction} 

This article is, in particular, an extension of the work \cite{MarklLoop} by Markl.
He showed that algebras considered by Zwiebach in \cite{ZwiebachClosed} in the context of closed string theory are algebras over the Feynman transform of the modular envelope of the cyclic operad $\oCom$.
This packed the complicated axioms into standard constructions over $\oCom$, an easily understood algebraic object.
Moreover, this opens up the way for applications of operad homotopy theory to study of these algebras.
For example, minimal models and transfer theorems have some physical relevance for gauge fixing \cite{korbi}, \cite{Lazarev1}.
The key point of the applications of the homotopy theory is that the Feynman transform is a modular operad which is always cofibrant in a suitable model structure\footnote{We need a model structure on the category of twisted modular operads (see Section \ref{SECTwiModOpe}). This has not been systematically studied so far, and we don't attempt to do so in this article.}.
There are now standard approaches for minimal models and transfer theorems for algebras over cofibrant operads, which hopefully carry over to our setting (e.g. \cite{Lazarev1}, where the homotopy transfer is discussed).

The above mentioned Zwiebach's work deals with closed string theory.
In this article, we interpret the analogous algebras in open and open-closed string theory from the operadic point of view.
For open strings\footnote{At least for the topological string.}, the role of the modular envelope of $\oCom$ is played by the operad $\oQO$ (Quantum Open), which is easily understood in terms of $2$-dimensional surfaces with boundary and some extra structure.
We emphasize that only the homeomorphism classes of surfaces are considered, thus the operad has an easy combinatorial description.
We consider the Feynman transform $\feyn{\oQO}$, an analogue of bar construction in the realm of modular operads, of $\oQO$ and describe the axioms of algebras over $\feyn{\oQO}$ explicitly in terms of generating operations and relations.
A particular case recovered are the axioms of quantum $A_\infty$ algebras of Herbst \cite{Herbst}.

In \cite{BarannikovModopBV}, Barannikov explained how algebras over the Feynman transform are equivalently described by solutions of a "master equation" in certain generalized BV algebra.
Applying this theory in the closed case, we obtain the BV algebra of Zwiebach's \cite{ZwiebachClosed}.
Application in the open case yields an improvement of the Herbst's \cite{Herbst}.

The approaches to $\feyn{\oQO}$ mentioned in the two above paragraphs are dual to each other:
The approach via generators and relations (Section \ref{SECAxiomsForQuantumAoo}) in the spirit of the usual $A_\infty$ algebras manifestly uses the dual of the composition in $\oQO$, while the approach via master equation (Section \ref{SECBarQO}) manifestly uses the composition directly.

Finally, we introduce a $2$-coloured modular operad $\oQOC$ (Quantum Open Closed) describing the algebraic structures in the open-closed case.
We make the generalized BV algebra explicit, thus obtaining a briefly mentioned result of \cite{KajiuraOpen-closed2} by Kajiura and Stasheff, which is deeply based on the open-closed string-field theory description by Zwiebach \cite{ZwiebachOpen-closed}. 
 
The original motivation for writing this article was understanding the work \cite{qocha} by M\"unster and Sachs.
Unfortunately, at the moment were not able to prove directly the equivalence of our approach to theirs. We hope to come back to this question in the future. 
The indirect relation is the following: In \cite{qocha}, the quantum open-closed homotopy algebra structure is obtained as a consequence of Zwiebach's open-closed quantum master equation, in our approach the open-closed homotopy algebra structure is, via Barannikov's approach \cite{BarannikovModopBV}, equivalent to a solution (of a non-commutative version) of the same.
\bigskip

There is an interesting pattern appearing:
Let's return back to closed strings to make things more precise.
The cyclic operad $\oCom$ in fact consist of (linear span of) $2$-dimensional surfaces of genus zero and boundary components, called closed (string) ends.
These surfaces relate to the vertices in the classical limit (genus zero) of the quantum (all genera) closed string field theory.
Let the operad $\oQC$ consists of $2$-dimensional surfaces of arbitrary genus with closed ends.
These surfaces relate to vertices in the Feynman diagrams in the \emph{quantum} closed string field theory.
It is easy to see that $\oQC$ is in fact the above mentioned modular envelope $\oMod{\oCom}$.

Thus the passage from classical to quantum vertices corresponds to taking the modular envelope of the corresponding operad.

The same pattern can be seen in the open case.
The cyclic operad $\oAss$ consists of $2$-dimensional surfaces of genus zero and one boundary component with marked points called open ends.
These surfaces relate to open vertices in the classical open string field theory.
The operad $\oQO$ consists of $2$-dimensional surfaces of arbitrary genus with arbitrary number of boundary components with closed ends.
These surfaces relate to open vertices in the Feynman diagrams in the \emph{quantum} open-closed string field theory\footnote{Recall, there is no consistent quantum string filed theory with open strings only.}.
As in the closed case, $\oQO$ is the modular envelope $\oMod{\oAss}$.

However, in the full open-closed case, the pattern seems to broke.
The operad for classical vertices, appearing in the work \cite{KajiuraOpen-closed1} by Kajiura and Stasheff, is not even cyclic, so $\oMod{\oP}$ doesn't make sense.
This is discussed in Section \ref{SECFQ}.
\bigskip

We finish the introduction by discussing some technical aspects of the paper.
The closed case discussed in \cite{MarklLoop} is very simple and can be dealt without paying too much attention to formal details.
In the open case, things get more complicated and one should be more careful.

To start with, we need a definition of (twisted) modular operad which is easy to verify in practice.
The standard definition in terms of triples (e.g. \cite{MarklOperads},\cite{GetzlerModop}) is inconvenient for this purpose.
Likewise, the biased definition in terms of collections $\{\oP(n,G)\}$ indexed by arities $n$ (and genus $G$) and composition $\ooo{i}{j}$ and contraction $\oxi{ij}$ (which is usual for ordinary operads) involves axioms which are too complicated.
So we choose an intermediate approach - the collections $\oP(C,G)$ indexed by finite sets $C$.
This way, the axioms can be stated succinctly and their geometric motivation is obvious (Sections \ref{SECModOp} and \ref{SECTwiModOpe}).
For our inherently geometric examples of operads $\oQC,\oQO$ and $\oQOC$, this definition is always easily verified.
The passage between collections indexed by integers and collections indexed by sets is discussed in some detail in Section \ref{SECSigmaVSSet}.

We also don't treat cyclic operads as ordinary operads with extra structure, but rather as objects on their own.
This emphasizes the geometric nature of the axioms.
The same approach has been adopted e.g. in \cite{Lukacs}.

To make the Feynman transform work, we need twisted operads.
The axioms stay succinct: the operadic compositions $\ooo{a}{b}$ and contractions $\oxi{ab}$ become degree $1$ morphisms, and a minus sign is introduced into the associativity axioms (Section \ref{SECTwiModOpe}).
Thus calculations with twisted operads are syntactically similar to those with untwisted operads.
The difference is analogous to computations in an algebra $A$ with degree $0$ multiplication satisfying the standard associativity relation $(ab)c=a(bc)$, and calculations in its suspension $\downarrow\!\!A$, where the multiplication has degree $1$ and satisfies $(ab)c=-(-1)^{\dg{a}}a(bc)$.
We emphasize that although $A$ and $\downarrow\!\!A$ in the above example are in a sense equivalent, the suspension trick doesn't work for modular operad and thus the use of twisted modular operads probably can't be avoided.
In this paper, similarly to \cite{KWZ}, we try to promote the use of twisted structures by showing that clear and explicit calculations can be performed with them.
This is best seen in the proof of Theorem \ref{THMFeynBaran}.

As the above approach is slightly nonstandard, we felt obliged to provide details.
Thus the notation is a bit overloaded and pace is slow.
We also kept the operadic prerequisites at minimum -- all the basic definitions are stated in full.
The only possibly technical part is the Feynman transform and it is wrapped in Theorem \ref{LEMMAAlgOverFeynTrans} giving a practical description of algebras over it.


\section{Conventions and notation}

\begin{enumerate}
\item $\N$ is the set of positive integers, $\N_0:=\N\cup\{0\}$.
\item $\fld$ is a field $\fld$ of characteristics $0$.
The multiplication in $\fld$ will be denoted $\cdot$ or omitted.
All (dg) vector spaces are considered over $\fld$.
\item Dg vectors spaces have differential of degree $+1$.
Morphisms of dg vector spaces are degree $0$ linear maps commuting with differentials.
\item $\sqcup$ is \emph{disjoint} union.
Whenever $A\sqcup B$ appears, $A,B$ are automatically assumed disjoint.
\item $\xrightarrow{\sim}$ denotes an iso (in particular a bijection).
\item $\uparrow$ is suspension.
\item $A^{\#}$ is the linear dual of $A$.
\end{enumerate}

\section{Modular Operads, Feynman Transform and Master Equation}\label{sec:operads}

\subsection{Modular operad} \label{SECModOp}

\begin{definition} \label{DEFCorr}
Denote $\cCo$ the category of stable corollas: the objects are pairs $(C,G)$ with $C$ a finite set and $G$ a nonnegative integer such that the stability condition is satisfied: $$2(G-1)+|C|>0.$$
A morphism $(C,G)\to (D,G')$ is defined only if $G=G'$ and it is just a bijection $C\xrightarrow{\sim}D$.
\end{definition}

\begin{definition} \label{DEFModOp}
Modular operad $\oP$ consists of a collection
$$\set{\oP(C,G)}{(C,G)\in\cCo}$$ of dg vector spaces and three collections 
\begin{gather*}
	\set{\oP(\rho):\oP(C,G)\to\oP(D,G)}{\rho:(C,G)\to(D,G)\textrm{ a morphism in }\cCo} \\
	\set{\ooo{a}{b}:\oP(C_1\sqcup\{a\},G_1)\ot\oP(C_2\sqcup\{b\},G_2)\!\to\!\oP(C_1\sqcup C_2,G_1+G_2)}{(C_1,G_1),\!(C_2,G_2)\in\cCo} \\
	\set{\xi_{ab}:\oP(C\sqcup\{a,b\},G)\to\oP(C,G+1)}{(C,G)\in\cCo}.
\end{gather*}
of degree $0$ morphisms of dg vector spaces.
These data are required to satisfy the following axioms:
\begin{enumerate}
\item $\ooo{a}{b} (x\ot y)  = (-1)^{\dg{x}\dg{y}} \ooo{b}{a} (y\ot x)$ \hskip 1cm
			for any $x\in\oP(C_1\sqcup\{a\},G_1), y\in\oP(C_2\sqcup\{b\},G_2)$,
\item $\oP(\id_C)=\id_{\oP(C)}, \quad \oP(\rho\sigma)=\oP(\rho) \ \oP(\sigma)$ \hskip 1cm
			for any morphisms $\rho,\sigma$ in $\cCo$,
\item $(\oP(\rho|_{C_1}\sqcup\sigma|_{C_2})) \ \ooo{a}{b} = \ooo{\rho(a)}{\sigma(b)} \ (\oP(\rho)\ot\oP(\sigma))$
\item $\oP(\rho|_{C}) \ \xi_{ab} = \xi_{\rho(a)\rho(b)}\oP(\rho)$
\item $\xi_{ab} \ \xi_{cd} = \xi_{cd} \ \xi_{ab}$
\item $\xi_{ab} \ \ooo{c}{d} = \xi_{cd} \ \ooo{a}{b}$
\item $\ooo{a}{b} \ (\xi_{cd}\ot\id) = \xi_{cd} \ \ooo{a}{b}$
\item $\ooo{a}{b} \ (\id\ot\ooo{c}{d}) = \ooo{c}{d} \ (\ooo{a}{b}\ot\id)$
\end{enumerate}
whenever the expressions make sense.
\end{definition}

\begin{remark} \label{REMSModCyc}
If we consider only the first two collections satisfying only Axiom $2.$, the resulting structure is called $\cCo$-module (more familiar name would be $\cSg$-module, but we reserve this name for slightly different structure, see Section \ref{SECSigmaVSSet} below), which is simply a functor from $\cCo$ to dg vector spaces.
Obviously, by forgetting structure, a modular operad gives rise to its underlying $\cCo$-module.

If we consider only the first three collections satisfying only Axioms $1.,2.,3.,8.$, the resulting structure is called cyclic operad.
By restricting to $G=0$ and forgetting structure, a modular operad gives rise to its underlying cyclic operad.

All these notions are equivalent to their usual counterparts in e.g. \cite{MarklOperads}.
For example, Axiom $2.$ stands for the $\Sigma$-action, $3.,4.$ express the equivariance and $5.-8.$ express the associativity of the structure maps.
\end{remark}

\subsection{Feynman transform}

The Feynman transform of a modular operad $\oP$ is a twisted\footnote{This will be defined in a moment.} modular operad denoted $\feyn{\oP}$.
Roughly speaking, $\feyn{\oP}$ is spanned by graphs with vertices decorated by elements of $\oP^{\#}$.

We make this more precise in the following example of an element of $\feyn{\oP}(C,G)$.
Consider a graph $\gr{G}$
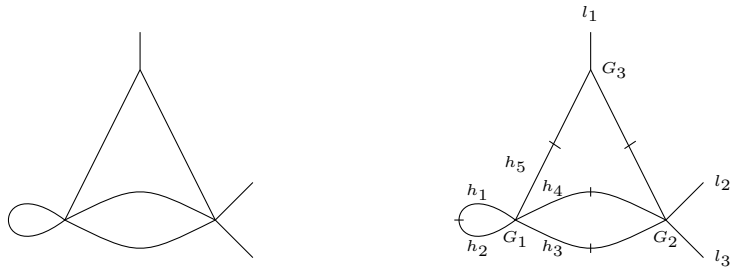
$$\PICGraph$$
A graph consists of vertices and half-edges.
Exactly one end of every half-edge is attached to a vertex.
The other end is either unattached (such an half-edge is called a leg) or attached to an end of another half-edge (in that case, these two half-edges form an edge).
Every end is attached to at most one vertex/end.
The half-edge structure of $\gr{G}$ is indicated on the picture on the right.
We require that every vertex $V_i$ is assigned a nonnegative integer $G_i$ such that $$\dim_{\Q}H_1(\gr{G},\Q) + \sum_i G_i = G.$$
We also require that the legs of $\gr{G}$ ($l_1,l_2,l_3$ in our case) are in bijection with $C$.
Finally we require $$2(G_i-1)+|V_i|>0$$ for every vertex $V_i$, where $|V_i|$ denotes the number of half-edges attached to $V_i$.
The graph $\gr{G}$ is ``decorated'' by an element $$(\uparrow\! e_1\wedge\cdots\wedge\uparrow\! e_5) \ot (P_1\ot P_2\ot P_3),$$
where $e_1,e_2,\ldots$ are all edges of $\gr{G}$, $\uparrow\! e_i$'s are formal elements of degree $+1$, $\wedge$ stands for the graded symmetric tensor product and finally $P_1\in\oP(\{h_1,\ldots,h_5\},G_1)^{\#}$ and similarly for $P_2$ and $P_3$ at vertices with $G_2$ and $G_3$.
Then the iso class of $\gr{G}$ together with $(\uparrow\! e_1\wedge\cdots\wedge\uparrow\! e_5) \ot (P_1\ot P_2\ot P_3)$ is an actual element of $\feyn{\oP}(C,G)$.

The operations $(\ooo{a}{b})_{\feyn{\oP}}$ and $(\xi_{ab})_{\feyn{\oP}}$ are defined by grafting of graphs, attaching together two previously unattached ends of two half-edges.

There is a Feynman differential $\partial_{\feyn{\oP}}$ on $\feyn{\oP}$ which adds an edge and modifies the decoration using the dual of $(\ooo{a}{b})_{\oP}$ or $(\xi_{ab})_{\oP}$.

Precise definitions are quite complicated technically (we refer to \cite{MarklOperads}).
Fortunately, we only need a tiny part of the Feynman transform theory, namely Theorem \ref{LEMMAAlgOverFeynTrans} which will come in a moment.

To avoid problems with duals, we assume that the dg vector space $\oP(C,G)$ is finite dimensional for any $(C,G)\in\cCo$ whenever $\feyn{\oP}$ appears.
This is sufficient for our applications, though it can probably be avoided using cooperads.
But, to our best knowledge, cooperads have never been investigated in the modular context.

One technical issue we treat in detail here is the notion of twisted operad, since $\feyn{\oP}$ is not an operad but a twisted operad.


\subsection{\texorpdfstring{Twisted modular operad}{Twisted modular operad}} \label{SECTwiModOpe}

\begin{definition} \label{DEFTwistedModOp}
A twisted modular operad $\oT$ consists of a collection
$$\set{\oT(C,G)}{(C,G)\in\cCo}$$
of dg vector spaces and a collection
$$\set{\oT(\rho):\oT(C,G)\to\oT(D,G)}{\rho:C\xrightarrow{\sim}D\textrm{ a bijection}, (C,G),(D,G)\in\cCo}$$
of degree $0$ morphisms of dg vector spaces and two collections
\begin{gather*}
	\set{\ooo{a}{b}:\oT(C_1\sqcup\{a\},G_1)\ot\oT(C_2\sqcup\{b\},G_2)\!\to\!\oT(C_1\sqcup C_2,G_1+G_2)}{(C_1,G_1),\!(C_2,G_2)\in\cCo} \\
	\set{\xi_{ab}:\oT(C\sqcup\{a,b\},G)\to\oT(C,G+1)}{(C,G)\in\cCo}.
\end{gather*}
of degree $+1$ morphisms of dg vector spaces.
These data are required to satisfy the following axioms:
\begin{enumerate}
\item $\ooo{a}{b} (x\ot y)  = (-1)^{\dg{x}\dg{y}} \ooo{b}{a} (y\ot x)$ \hskip 1cm
			for any $x\in\oT(C_1\sqcup\{a\},G_1), y\in\oT(C_2\sqcup\{b\},G_2)$,
\item $\oT(\id_C)=\id_{\oT(C)}, \quad \oT(\rho\sigma)=\oT(\rho) \ \oT(\sigma)$ \hskip 1cm
			for any morphisms $\rho,\sigma$ in $\cCo$,
\item $\oT(\rho|_{C_1}\sqcup\sigma|_{C_2}) \ \ooo{a}{b} = \ooo{\rho(a)}{\sigma(b)} \ (\oT(\rho)\ot\oT(\sigma))$
\item $\oT(\rho|_{C}) \ \xi_{ab} = \xi_{\rho(a)\rho(b)}\oT(\rho)$
\item $\xi_{ab} \ \xi_{cd} = - \xi_{cd} \ \xi_{ab}$
\item $\xi_{ab} \ \ooo{c}{d} = - \xi_{cd} \ \ooo{a}{b}$
\item $\ooo{a}{b} \ (\xi_{cd}\ot\id) = - \xi_{cd} \ \ooo{a}{b}$
\item $\ooo{a}{b} \ (\id\ot\ooo{c}{d}) = - \ooo{c}{d} \ (\ooo{a}{b}\ot\id)$
\end{enumerate}
whenever the expressions make sense.
\end{definition}

This notion is equivalent to the modular $\KK$-operad of Getzler and Kapranov \cite{GetzlerModop} (also called $\KK$-twisted modular operad), where $\KK$ is the determinant-of-edges coefficient system (also called hyperoperad).
However, these explicit axioms have never before appeared in the literature.


\subsection{\texorpdfstring{$\cSg$-modules}{Sigma-modules}} \label{SECSigmaVSSet}

So now we wish to define algebras over the Feynman transform, and for that, we need a twisted endomorphism operad $\oEnd{A}$.
Informally speaking, $\oEnd{A}$ consists of covariant tensors and the operadic composition and contraction is given by contraction of the tensors using a symplectic form of degree $-1$.
There is the following inconvenience:
In $\oEnd{A}(C,G)$, the \emph{set} $C$ should index inputs of a $|C|$-times covariant tensor.
But while there is no canonical order on $C$, the inputs of the tensor are ordered by definition.
This makes the definition of $\oEnd{A}$ in terms of a $\cCo$-module clumsy and unintuitive.

Thus it is helpful to restrict the category $\cCo$ of corollas to a smaller one where a canonical order is available:

\begin{definition}
$\cSg$ is the skeleton of $\cCo$ consisting of all stable corollas of the form $([n],G)$, $n\in\N_0$.
$\cSg$-module is a functor from $\cSg$ to dg vector spaces.
\end{definition}

Given a (twisted) modular operad $\oP$, we can restrict its underlying $\cCo$-module to $\cSg$-module.
Then we construct an analogue of operadic composition and contraction on the restricted module as follows:

We first need some fixed auxiliary permutations:
\begin{definition}
For $n\in\N_0$ and $\{a_1,a_2,\ldots\}\subset\N_0$, define 
$$n+\{a_1,a_2,\ldots\}:=\{n+a_1,n+a_2,\ldots\}.$$
Given $n\in\N_0$, define a bijection 
$$\std^{ij}:[n+2]-\{i,j\}\to[n]$$
by requiring it to be increasing.
Given $n_1,n_2\in\N_0$, define bijections 
\begin{gather*}
\std_1^i:[n_1+1]-\{i\}\to[n_1] \\
\std_2^j:[n_2+1]-\{j\}\to n_1+[n_2]
\end{gather*}
by requiring them to be increasing.
\end{definition}

\begin{definition} \label{DEFcircol}
Given a (twisted) modular operad $\oP$ with structure morphisms $\ooo{a}{b}$ and $\oxi{ab}$, define $\overline{\oP}$ to consist of three collections:
$$\set{\overline{\oP}(n,G)}{([n],G)\in\cCo},$$
a collection of dg $\fld\Sigma_n$-modules, and collections
\begin{gather*}
\set{\ooot{i}{j}:\overline{\oP}(n_1+1,G_1)\ot\overline{\oP}(n_2+1,G_2)\to\overline{\oP}(n_1+n_2,G_1+G_2)}{([n_1],G_1),\!([n_2],G_2)\in\cCo}, \\
\set{\oxit{ij}:\overline{\oP}(n+2,G)\to\overline{\oP}(n,G+1)}{([n],G)\in\cCo}
\end{gather*}
of degree $0$ ($1$) morphisms of dg vector spaces determined by formulas
\begin{gather*}
\overline{\oP}(n,G) := \oP([n],G), \\
\ooot{i}{j} := \oP(\std^i_1\lambda_1^{-1}|_{C_1}\sqcup\std^j_2\lambda_2^{-1}|_{C_2}) \ooo{a}{b} (\oP(\lambda_1)\ot\oP(\lambda_2)), \quad i:=\lambda_1^{-1}(a), j:=\lambda_2^{-1}(b), \\
\oxit{ij} := \oP(\std^{ij}\lambda^{-1}|_C)\oxi{ab}\oP(\lambda), \quad i:=\lambda^{-1}(a), j:=\lambda^{-1}(b),
\end{gather*}
where $\lambda:[n+2]\to C\sqcup\{a,b\}$, $\lambda_1:[n_1+1]\to C_1\sqcup\{a\}$ and $\lambda_2:[n_2+1]\to C_2\sqcup\{b\}$ are arbitrary bijection such that $C_1\cap C_2=\emptyset$.
\end{definition}

Obviously, $\overline{\oP}$ with $\ooot{i}{j}$'s and $\oxit{ij}$'s forgotten  is a $\cSg$-module.
One easily verifies that the definitions of $\ooot{i}{j}$ and $\oxit{ij}$ are independent of the choice of $\lambda,\lambda_1,\lambda_2$.

\begin{remark}
Certain choices can simplify the formulas:
If $\lambda=\id_{[n+2]}$, then $\oxit{ij} := \oP(\std^{ij})\oxi{ij}$.
If $\lambda_1$ maps $[n_1+1]-\{i\}$ increasingly onto $C_1$ and $\lambda_2$ maps $[n_2+1]-\{j\}$ increasingly onto $C_2$, then $\std^i_1\lambda_1^{-1}|_{C_1}\sqcup\std^j_2\lambda_2^{-1}|_{C_2}$ is a $(n_1,n_2)$-unshuffle mapping $C_1$ onto $[n_1]$.

Recall that $\rho\in\Sigma_n$ is called $(n_1,n_2)$-shuffle iff $\rho(1)<\cdots<\rho(n_1)$ and $\rho(n_1+1)<\cdots<\rho(n_1+n_2)$.
The set of all $(n_1,n_2)$-shuffles is denoted $\Sh{n_1,n_2}$.
Then $\rho\in\Sigma_n$ is called $(n_1,n_2)$-unshuffle iff $\rho^{-1}\in\Sh{n_1,n_2}$.
\end{remark}

The operations $\ooot{i}{j}$ and $\oxit{ij}$ satisfy certain properties, call them (P), analogous to those of Definitions \ref{DEFModOp} and \ref{DEFTwistedModOp}.
Denote $\cOp_\cCo$ the category of (twisted) modular operads.
It would be natural to define a new category $\cOp_\cSg$ of $\cSg$-modules with $\ooot{i}{j}$'s and $\oxit{ij}$'s satisfying (P), then notice that the construction $\oP\mapsto\overline{\oP}$ of Definition \ref{DEFcircol} induces an equivalence between $\cOp_\cCo$ and $\cOp_\cSg$, and finally use whichever of the two categories is convenient in a given situation.

Unfortunately, it turns out that the formulas corresponding to (P) are too complicated for any practical purposes.
Still, we use $\cOp_\cSg$ in several places of the paper.
For example, the endomorphism operad $\oEnd{A}$ has a particularly easy definition in $\cOp_\cSg$ (see Definition \ref{DEFEndSigma} below).

So we adopt the following point of view:
If we work with and object of $\cOp_\cSg$, we always assume it is of the form $\overline{\oP}$ for some $\oP\in\cOp_\cCo$.
This is the way we are going to rigorously define the endomorphism operad:
we define $\oEnd{A}\in\cOp_\cCo$ first, then construct $\oEndt{A}\in\cOp_\cSg$, and finally observe that $\oEndt{A}$ is nice and simple.
This should justify our exposition in the next section.


\subsection{\texorpdfstring{Endomorphism twisted modular operad}{Endomorphism twisted modular operad}}

Let $(A,d)$ be a dg vector space of dimension $n$ equipped with a symplectic form $\omega$ of degree $-1$, i.e.
$$\omega(u,v)\neq 0 \quad \Rightarrow\quad \dg{u}+\dg{v}=1$$
for any homogeneous elements of $A$.
Let $d(\omega)=0$, i.e.
$$\omega(d\ot\id+\id\ot d) = 0,$$
so that $(A,d,\omega)$ is a dg symplectic vector space.
Let $\{a_i\}$ be a homogeneous basis of $A$.
Define 
\begin{gather} \label{EQDefOmegaInverse}
b_i := \sum_{j=1}^{n} (-1)^{\dg{a_j}}\omega^{ij}a_j,
\end{gather}
where $\omega^{ij}$'s are the components of the matrix inverse of $\omega_{ij}:=\omega(a_i,a_j)$.
Now we can contract indices of tensors $\op{Hom}_{\fld}(A^{\ot n},\fld)$ using $\omega^{-1} := \sum_{i=1}^n a_i\ot b_i \in A\ot A$, but we prefer to express the contractions using the bases $\{a_i\}$ and $\{b_i\}$.

From now on, dg symplectic vector space will refer to a structure such as above including the bases $\{a_i\}$ and $\{b_i\}$.

Let's define a candidate for $\oEndt{A}$.
This is the nice and simple result we wish for:
\begin{definition} \label{DEFEndSigma}
\begin{gather*}
\oEndt{A}(n,G):=\op{Hom}_{\fld}(A^{\ot n},\fld) \\
d(f) := (-1)^{\dg{f}} f \circ d_{A^{\ot n}} = \sum_{i=1}^{n} (-1)^{\dg{f}} f(\cdots\ot\underbrace{d}_{\mathclap{i\textrm{-th}}}\ot\cdots) \\
\ooot{i}{j}(f\ot g):=\sum_k (-1)^{\dg{f}+\dg{g}\dg{b_k}}f(\cdots\ot\underbrace{a_k}_{\mathclap{i\textrm{-th}}}\ot\cdots)\cdot g(\cdots\ot\underbrace{b_k}_{\mathclap{j\textrm{-th}}}\ot\cdots) \\
\oxit{ij}(f) := \sum_k (-1)^{\dg{f}} f(\cdots\ot\underbrace{a_k}_{\mathclap{i\textrm{-th}}}\ot\cdots\ot\underbrace{b_k}_{\mathclap{j\textrm{-th}}}\ot\cdots) \\
\oEndt{A}(\rho)(f) = f\circ\rho^{-1}
\end{gather*}
for any permutation $\rho$.
\end{definition}

Now we need to verify that this candidate is indeed of the form $\oEndt{A}$ for some $\oEnd{A}\in\cOp_\cCo$.
We need a preliminary on unordered tensor product:

\begin{definition}
For any set $C$, we define $$\bigotimes_C A := \left. \bigoplus_{\psi:C\xrightarrow{\sim}[|C|]}\!\!\!A^{\ot|C|} \right/\sim,$$
where $\psi$'s are bijections.
The equivalence $\sim$ is given by 
\begin{gather} \label{EQSimVsEqual}
i_\psi(v_1\ot\cdots\ot v_{|C|}) \sim i_{\tau\psi}\ (\tau(v_1\ot\cdots\ot v_{|C|}))
\end{gather}
for any $\psi:C\xrightarrow{\sim}[|C|]$, any $\tau:[|C|]\xrightarrow{\sim}[|C|]$ and any $v_i$'s in $A$.
We have denoted $i_\psi : A^{\ot|C|} \to \bigoplus_{\psi} A^{\ot|C|}$ the canonical inclusion into the $\psi$-th summand.
Recall that the action of $\tau\in\Sigma_{|C|}$ on $V^{\ot|C|}$ is $$\tau(v_1\ot\cdots\ot v_{|C|}) = \pm\ v_{\tau^{-1}(1)}\ot\cdots\ot v_{\tau^{-1}(|C|)}$$ with the usual Koszul sign.
We denote $\tau\in\Sigma_n$ and corresponding $\tau:V^{\ot|C|}\to V^{\ot|C|}$ by the same symbol.
Let 
\begin{gather} \label{DEFIota}
\iota_\psi : A^{\ot|C|} \to \bigotimes_C A
\end{gather}
be the natural inclusion $i_\psi$ followed by the natural projection.
\end{definition}
Hence 
\begin{gather} \label{EQTinyUseful}
\iota_\psi = \iota_{\tau\psi} \circ \tau
\end{gather}
follows from \eqref{EQSimVsEqual}.
Observe that $\iota_\psi$ is an iso for any bijection $\psi$.

\begin{definition} \label{DEFEndCorolla}
Let
$$\oEnd{A}(C,G) := \textrm{Hom}_{\fld}(\bigotimes_C A,\fld)$$
for any $(C,G)\in\cCo$.

This graded vector space has a differential given by
$$d(f) \comp \iota_\psi := (-1)^{\dg{f}} \sum_{i=0}^{|C|-1} (f \comp \iota_\psi) (\id^{\ot i}\ot d\ot\id^{\ot |C|-i-1})$$
for any $f\in\oEnd{A}(C,G)$.

Let $f\in\oEnd{A}(C_1\sqcup\{a\},G_1), g\in\oEnd{A}(C_2\sqcup\{b\},G_2)$ and define 
\begin{gather} \label{EQoo}
\ooo{a}{b}(f\ot g) \comp \iota_\psi := \sum_{i=1}^n (-1)^{\dg{f}+\dg{g}\dg{b_i}} (f\comp\iota_{\psi_1})(\id^{\ot|C_1|}\ot a_i) \cdot (g\comp\iota_{\psi_2})(\id^{\ot|C_2|}\ot b_i),
\end{gather}
where $\psi : C_1\sqcup C_2 \xrightarrow{\sim} [|C_1|+|C_2|]$ is arbitrary satisfying $\psi(C_1)=[|C_1|]$ and $\psi_1, \psi_2$ are then defined by $\psi_1(c_1):=\psi(c_1),\ \psi_1(a):=|C_1|+1$ and $\psi_2(c_2):=\psi(c_2)-|C_1|,\ \psi_2(b):=|C_2|+1$ for all $c_1\in C_1, c_2\in C_2$.

For $f\in\oEnd{A}(C\sqcup\{a,b\},G)$, we define
\begin{gather} \label{EQxi}
\xi_{ab}(f) \comp \iota_\psi := (-1)^{\dg{f}} \sum_{i=1}^n (f\comp\iota_{\psi'})(\id^{\ot|C|}\ot a_i\ot b_i),
\end{gather} 
where $\psi:C\xrightarrow{\sim}[|C|]$ is arbitrary and $\psi'|_C=\psi|_C,\ \psi'(a):=|C|+1, \psi'(b):=|C|+2$.

Finally, for $\rho:C\xrightarrow{\sim}D$ and $\psi:D\xrightarrow{\sim}[|D|]$, we define $\oEnd{A}(\rho):\oEnd{A}(C,G)\to\oEnd{A}(D,G)$ by 
$$\oEnd{A}(\rho)(f) \comp \iota_\psi := f\comp\iota_{\psi\rho}.$$
Consequently, if $C=D$, then $\oEnd{A}(\rho)(f) \comp \iota_\psi = f\comp\iota_{\psi}\comp\psi\rho^{-1}\psi^{-1}$.
\end{definition}

The reader can now verify:
\begin{theorem} \label{THMEndIsOperad}
$\oEnd{A}$ of Definition \ref{DEFEndCorolla} is a twisted modular operad.
Taking its $\oEndt{A}$ (in the sense of Definition \ref{DEFcircol}) yields a result described in Definition \ref{DEFEndSigma}, if we identify
$$f\in\op{Hom}_{\fld}(\bigotimes_{[n]} A,\fld) = \oEnd{A}(n,G)$$
with
$$f\circ\iota_{\id_{[n]}} \in \op{Hom}_{\fld}(A^{\ot n},\fld).$$
\end{theorem}

The above identification will be done implicitly in the sequel.
It is formally justified by first writing an equation in terms of maps from the unordered tensor products and then composing both sides with $\iota_{\id_{[n]}}$.

\begin{convention}
We will write just $\oP$ for both $\oP\in\cOp_\cCo$ and $\overline{\oP}\in\cOp_\cSg$.
The former object consists of collections 
$$\{\oP(C,G)\},\ \{\ooo{a}{b}\},\ \{\oxi{ab}\},$$
the latter of
$$\{\oP(n,G)\},\ \{\ooot{i}{j}\},\ \{\oxit{ij}\}.$$
\end{convention}

\begin{remark}
In string field theory, $A$ corresponds to a subspace of the space of states of a proper conformal field theory, $d$ to the BRST operator and $\omega$ is related to the BPZ product with proper zero mode insertions.
\end{remark}


\subsection{\texorpdfstring{Algebra over a twisted operad}{Algebra over a twisted operad}} \label{SECTAlgOverTwisted}

\begin{definition} \label{DEFAlgOverTwisted}
Let $\oT$ be a twisted modular operad.
An algebra over $\oT$ on a dg symplectic vector space $A$ is a twisted modular operad morphism $$\alpha : \oT \to \oEnd{A},$$
i.e. it is a collection $$\{\alpha(C,G):\oT(C,G)\to\oEnd{A}(C,G)\ | \ (C,G)\in\cCo\}$$ of dg vector space morphisms such that (in the sequel, we drop the notation $(C,G)$ at $\alpha(C,G)$ for brevity)
\begin{enumerate}
	\item $\alpha \comp \oT(\rho) = \oEnd{A}(\rho) \comp \alpha$ for any bijection $\rho$
	\item $\alpha \comp (\ooo{a}{b})_{\oT} = (\ooo{a}{b})_{\oEnd{A}} \comp (\alpha\ot\alpha)$
	\item $\alpha \comp (\xi_{ab})_{\oT} = (\xi_{ab})_{\oEnd{A}} \comp \alpha$
\end{enumerate}
\end{definition}

Hence every element $t\in\oT(C,G)$ is assigned a linear map $\alpha(t) : \bigotimes_C A\to\fld$.
In practice, however, one is rather interested in linear maps $A^{\ot|C|}\to\fld$.
Of course, $A^{\ot|C|}\cong\bigotimes_C A$, but this is not canonical!
We get around this nuisance as follows:
Observe that once we know $\alpha([n],G)$ for all corollas with $n,G\geq 0$, $\alpha$ is determined on other corollas by Axiom $1.$ of Definition \ref{DEFAlgOverTwisted}.
But for $[n]$, there is a canonical iso $A^{\ot n}\cong\bigotimes_{[n]}A$, namely $\iota_{\id_{[n]}}$ of Definition \ref{DEFIota}.
Hence we may replace $f : \bigotimes_{[n]}A\to\fld$ by $f \comp \iota_{\id_{[n]}} : A^{\ot n}\to\fld$ (compare to Theorem \ref{THMEndIsOperad}).
%
%

\subsection{\texorpdfstring{Algebra over the Feynman transform}{Algebra over the Feynman transform}}

The following theorem is essentially the only thing we need from the theory of Feynman transform.
It has already implicitly appeared in e.g. \cite{MarklLoop} and \cite{BarannikovModopBV}.

\begin{theorem} \label{LEMMAAlgOverFeynTrans}
Algebra over the Feynman transform $\feyn{\oP}$ of a modular operad $\oP$ on a dg vector space $A$ is equivalently determined by a collection
$$\set{\alpha(C,G):\oP(C,G)^{\#}\to\oEnd{A}(C,G)}{(C,G)\in\cCo}$$
of degree $0$ linear maps (no compatibility with differential on $\oP(C,G)^{\#}$!) such that 
\begin{align}
	\oEnd{A}(\rho) \comp \alpha(C,G) &= \alpha(D,G) \comp \oP(\rho^{-1})^{\#} \quad\textrm{for any bijection }\rho:C\xrightarrow{\sim} D \textrm{ and} \label{EQAlfOverFTOne} \\
	d \comp \alpha(C,G) &= \alpha(C,G) \comp \partial_{\oP}^{\#} + (\xi_{ab})_{\oEnd{A}} \comp \alpha(C\sqcup\{a,b\},G-1) \comp (\xi_{ab})^{\#}_{\oP} \ + \label{EQAlfOverFTTwo} \\
	&\hspace{-5em} + \frac{1}{2} \sum_{\substack{C_1\sqcup C_2 = C \\ G_1+G_2=G}} (\ooo{a}{b})_{\oEnd{A}} \comp \left( \alpha(C_1\sqcup\{a\},G_1)\ot\alpha(C_2\sqcup\{b\},G_2) \right) \comp (\oooo{a}{b}{C_1\sqcup\{a\},G_1}{C_2\sqcup\{b\},G_2})^{\#}_{\oP}, \nonumber
\end{align}
where $$(\oooo{a}{b}{C_1\sqcup\{a\},G_1}{C_2\sqcup\{b\},G_2})^{\#}_{\oP} : \oP(C,G)^{\#} \to \oP(C_1\sqcup\{a\},G_1)^{\#} \ot \oP(C_2\sqcup\{b\},G_2)^{\#}$$ is the dual of $\ooo{a}{b}$ on $\oP$ and $\partial_{\oP}$ is the differential on $\oP$.
\end{theorem}

To make this compatible with \cite{MarklOperads} and \cite{GetzlerModop}, we have to interpret the sum as follows: If $C_1=C_2=\emptyset$ and $G_1=G_2$, then the corresponding term appears twice in the sum by definition.

To make algebras over Feynman transform explicit as easily as possible, we modify this theorem.
The first idea is that, in our applications, the description of $\oP$ is easier using $\cCo$-modules, while description of $\oEnd{A}$ is easier using $\mathsf{\Sigma}$-modules.
The second idea is that it is enough to determine $\alpha([n],G)$ for all $n,G$.
Moreover, $\alpha([n],G)$ is determined by its values on orbit representatives.
The $\kappa$'s below are intended to transform a generic element to such an orbit representative; a choice of the representatives and $\kappa$ will be adapted to a concrete application.

\begin{lemma} \label{LEMMAHardToUse}
Algebra over the Feynman transform $(\feyn\oP,\partial_{\feyn\oP})$ on a dg vector space $A$ is uniquely determined by a collection
$$\set{\alpha([n],G):\oP([n],G)^{\#}\to\oEnd{A}(n,G)}{([n],G)\in\cCo}$$
of degree $0$ linear maps (no compatibility with differential!) such that\footnote{In the sequel, we simplify the notation a bit: the $([n],G)$ at $\alpha([n],G)$ is usually omitted and so is the symbol $\circ$ for composition of maps.} 
\begin{gather*}
	\oEnd{A}(\rho) \ \alpha = \alpha \ \oP(\rho^{-1})^{\#} \quad\textrm{for any }\rho\in\Sigma_n \textrm{ and} \\
	d \alpha = \alpha \partial_{\oP^{\#}} +{} \\
	{}+ \oEnd{A}((\rho^{\kappa(a)\kappa(b)}\kappa|_{[n]})^{-1})(\overline{\xi}_{\kappa(a)\kappa(b)})_{\oEnd{A}}\alpha \oP(\kappa^{-1})^{\#} (\xi_{ab})^{\#}_{\oP} +{} \\
	\mathclap{ {}+ \frac{1}{2} \sum_{\mathclap{\substack{{}\\ C_1\sqcup C_2=[n] \\ G_1+G_2=G}}} \oEnd{A}(\rho_1^{\kappa_1(a)}\kappa_1|_{C_1}\sqcup\rho_2^{\kappa_2(b)}\kappa_2|_{C_2})^{-1} (\ooot{\kappa_1(a)}{\kappa_2(b)})_{\oEnd{A}} (\alpha\!\ot\!\alpha) (\oP(\kappa_1^{-1})^{\#}\!\ot\!\oP(\kappa_2^{-1})^{\#}) (\oooo{a}{b}{C_1\sqcup\{a\},G_1}{C_2\sqcup\{b\},G_2})^{\#}_{\oP}, }
\end{gather*}
where $\kappa:[n]\sqcup\{a,b\}\to[n+2]$, $\kappa_1:C_1\sqcup\{a\}\to[|C_1|+1]$ and $\kappa_2:C_2\sqcup\{b\}\to[|C_2|+1]$ are arbitrary bijections. \qed
\end{lemma}

\subsection{Barannikov's theory} \label{SSECBarannikov}

In \cite{BarannikovModopBV}, Barannikov observed that a twisted modular operad morphism $\feyn{\oP}\to\oEnd{A}$ is equivalently described as a solution of certain master equation in an algebra succinctly defined in terms of $\oP$.
For ordinary operads, there is a systematic approach to a similar problem via the convolution operad and its associated Maurer-Cartan equation, e.g. Section $6.4.2$ of \cite{LV}.
This theory has been extended in \cite{KWZ} to include Barannikov's example.

In this section, we restate the corresponding theorem in our formalism, reprove it and then adapt it to our applications.


\begin{definition} \label{DEFBVOperaions}
Let $\oP$ be a modular operad.
Recall that we assume
$$\dim_\fld\oP([n],G)<\infty \quad \textrm{for all }n,G.$$
Define 
\begin{gather*}
P(n,G):=\left(\oP([n],G)\ot\oEnd{A}([n],G)\right)^{\Sigma_n} \\
P:= \prod_{\substack{n\geq 0\\ G\geq 0}} P(n,G)
\end{gather*}
with $P(n,G)$ being the space of invariants under the diagonal $\Sigma_n$ action on the tensor product.
For $e\in P(n,G)$, let
$$d(e) := \left( d_{\oP([n],G)}\ot\id_{\oEnd{A}([n],G)} - \id_{\oP([n],G)}\ot d_{\oEnd{A}([n],G)} \right) (e).$$
For $f\in P(n+2,G+1)$, let 
$$\Delta(f) := \left((\xi_{ab})_{\oP} \ot (\xi_{ab})_{\oEnd{A}}\right) (\oP(\theta)\ot\oEnd{A}(\theta)) (f)$$
for an arbitrary bijection $\theta:[n+2]\to[n]\sqcup\{a,b\}$.
Finally, for $g\in P(n_1+1,G_1)$ and $h\in P(n_2+1,G_2)$, let 
$$\{g,h\}:= \sum_{\mathclap{\substack{C_1\sqcup C_2=[n_1+n_2] \\ |C_1|=n_1,\ |C_2|=n_2}}} \left((\ooo{a}{b})_{\oP} \ot (\ooo{a}{b})_{\oEnd{A}}\right) \tau (\oP(\theta_1)\ot\oEnd{A}(\theta_1)\ot\oP(\theta_2)\ot\oEnd{A}(\theta_2)) (g\ot h),$$
where $\theta_1:[n_1+1]\to C_1\sqcup\{a\}$ and $\theta_2:[n_2+1]\to C_2\sqcup\{b\}$ are arbitrary bijections and $\tau$ is the composite
\begin{gather*}
\big(\oP(C_1\sqcup\{a\},G_1)\!\ot\!\oEnd{A}(C_1\sqcup\{a\},G_1)\big)^{\Sigma_{C_1\sqcup\{a\}}}\!\!\ot\!\big(\oP(C_2\sqcup\{b\},G_2)\!\ot\!\oEnd{A}(C_2\sqcup\{b\},G_2)\big)^{\Sigma_{C_2\sqcup\{b\}}} \! \hookrightarrow \\
\oP(C_1\sqcup\{a\},G_1)\ot\oEnd{A}(C_1\sqcup\{a\},G_1)\ot\oP(C_2\sqcup\{b\},G_2)\ot\oEnd{A}(C_2\sqcup\{b\},G_2) \to \\
\oP(C_1\sqcup\{a\},G_1)\ot\oP(C_2\sqcup\{b\},G_2)\ot\oEnd{A}(C_1\sqcup\{a\},G_1)\ot\oEnd{A}(C_2\sqcup\{b\},G_2)
\end{gather*}
exchanging the two middle factors.

These formulas extend by infinite linearity to operations $d,\Delta$ and $\{\}$ on $P$.
\end{definition}

It is easily seen that these operations take values in $P$ and don't depend on the choice of $\theta,\theta_1,\theta_2$.
The following formulas are also easy to verify:

\begin{lemma}
\begin{gather*}
\Delta(f) = \left((\oxit{ij})_{\oP} \ot (\oxit{ij})_{\oEnd{A}}\right) (f), \\
\{g,h\}:= \sum_{\rho\in\Sh{n_1,n_2}} (\oP(\rho)\ot\oEnd{A}(\rho)) \left((\ooot{i}{j})_{\oP} \ot (\ooot{i}{j})_{\oEnd{A}}\right) \tau (g\ot h),
\end{gather*}
where $i,j$ are arbitrary and can even be chosen differently for each $\rho$ in the second formula.
\qed
\end{lemma}

\begin{theorem}[\cite{BarannikovModopBV}] \label{THMFeynBaran}
Algebra over the Feynman transform $\feyn{\oP}$ on a dg symplectic space $A$ is equivalently given by a degree $0$ element, called generating function or action,
$$S\in P:=\prod_{\mathclap{n,G\geq 0}} \left(\oP([n],G)\ot\oEnd{A}([n],G)\right)^{\Sigma_n}$$
satisfying the master equation
$$d(S)+\Delta(S)+\frac{1}{2}\{S,S\}=0$$
in the algebra $(P,d,\Delta,\{\})$ defined above.

$(P,d,\Delta,\{\})$ is a ``generalized BV algebra'', i.e. all the operations $d:P\to P$, $\Delta:P\to P$ and $\{\}:P\ot P\to P$ have degree $+1$ and  $\{\}$ is symmetric and these satisfy
\begin{gather}
\{\{f,g\},h\}+(-1)^{\dg{h}(\dg{f}+\dg{g})}\{\{h,f\},g\}+(-1)^{\dg{f}(\dg{g}+\dg{h})}\{\{g,h\},f\}=0 \quad \textrm{for any }f,g,h\in P, \label{EQBVJacobi}\\
d^2=0, \nonumber \\
d\{\}+\{\}(d\ot\id+\id\ot d)=0, \nonumber\\
\Delta^2=0, \nonumber \\
\Delta\{\}+\{\}(\Delta\ot\id+\id\ot\Delta)=0, \label{EQBVDeltaBracket}\\
\Delta d + d\Delta=0. \nonumber
\end{gather}
\end{theorem}

Notice that the generalized BV algebra is a suspension of Lie algebra with two commuting differentials satisfying the Leibniz relation w.r.t. the bracket.

\begin{proof}
Consider the iso
\begin{align}
Z : \op{Hom}_{\Sigma_C}(\oP(C,G)^{\#},\oEnd{A}(C,G)) &\cong \left( \oP(C,G)\ot\oEnd{A}(C,G) \right)^{\Sigma_C} \label{EQBigXiIso} \\
\kappa &\mapsto Z(\kappa) := \sum_{i} (-1)^{\dg{\kappa}\dg{p_i}} p_i\ot\kappa(p_i^{\#}) \no
\end{align}
where the LHS is the space of all linear $\Sigma_C$-equivariant maps, $\{p_i\}$ is a $\fld$-basis of $\oP(C,G)$ and $\{p_i^{\#}\}$ is its dual basis.

Assume $\alpha$ as in Theorem \ref{LEMMAAlgOverFeynTrans} determines an algebra over $\feyn{\oP}$.
Let $S_{n,G}:= Z(\alpha([n],G)) = \sum_i p_i\ot\alpha([n],G)(p_i^{\#})(-1)^{\dg{\alpha}\dg{p_i}}$, where $\{p_i\}$ is a basis of $\oP([n],G)$.
Let $S := \sum_{n,G} S_{n,G}$.
Obviously $S \in P$.
In Equation \eqref{EQAlfOverFTTwo} for $C=[n]$ and fixed $G$, denote
\begin{gather*}
\beta = \alpha([n],G) \comp \partial_{\oP}^{\#} - d \comp \alpha([n],G) \\
\gamma = (\xi_{ab})_{\oEnd{A}} \comp \alpha([n]\sqcup\{a,b\},G-1) \comp (\xi_{ab})^{\#}_{\oP} \\
\delta = \frac{1}{2} \sum_{\substack{C_1\sqcup C_2 = [n] \\ G_1+G_2=G}} (\ooo{a}{b})_{\oEnd{A}} \comp \left( \alpha(C_1\sqcup\{a\},G_1)\ot\alpha(C_2\sqcup\{b\},G_2) \right) \comp (\oooo{a}{b}{C_1\sqcup\{a\},G_1}{C_2\sqcup\{b\},G_2})^{\#}_{\oP}.
\end{gather*}
Now we apply $Z$ to \eqref{EQAlfOverFTTwo} and get $0=Z(\beta)+Z(\gamma)+Z(\delta)$.
We will evaluate $Z(\delta)$, the other terms are easier and we leave them to the reader.
To simplify the notation, we write $\alpha$ for $\alpha(C,G)$.
Since $\dg{\delta}=1$, we get
\begin{gather} \label{EQZetI}
2Z(\delta) = \sum_{i} p_i \ot \sum_{C_1,C_2,G_1,G_2} (\ooo{a}{b})_{\oEnd{A}} (\alpha\ot\alpha) (\oooo{a}{b}{C_1\sqcup\{a\},G_1}{C_2\sqcup\{b\},G_2})^{\#}_{\oP} (p_i^{\#}) (-1)^{\dg{p_i}}.
\end{gather}
Let $\{x_k\}$ be a basis of $\oP(C_1\sqcup a,G_1)$ and $\{y_l\}$ be a basis of $\oP(C_2\sqcup b,G_2)$.
Then let
$$(\ooo{a}{b})_{\oP}(x_k\ot y_l) = \sum_m O_{kl}^i p_i$$
for some $O_{kl}^i\in\fld$.
Thus
$$(\oooo{a}{b}{C_1\sqcup\{a\},G_1}{C_2\sqcup\{b\},G_2})^{\#}_{\oP}(p_i^{\#}) = \sum_{k,l} O_{kl}^i x_k^\# \ot y_l^\# (-1)^{\dg{x_k}\dg{y_l}}$$
since the dual base vector to $x_k \ot y_l$ is $x_k^\# \ot y_l^\# (-1)^{\dg{x_k}\dg{y_l}}$.
Observe that $O_{kl}^i\neq 0$ implies $\dg{x_k}+\dg{y_l}=\dg{p_i}$ and we therefore replace $(-1)^{\dg{p_i}}$ by $(-1)^{\dg{x_k}+\dg{y_l}}$ in \eqref{EQZetI}:
\begin{gather*}
\sum_{C_1,C_2,G_1,G_2} \sum_{i,k,l} O_{kl}^i p_i \ot (\ooo{a}{b})_{\oEnd{A}} (\alpha\ot\alpha) (x_k^\# \ot y_l^\#) (-1)^{\dg{x_k}\dg{y_l}+\dg{x_k}+\dg{y_l}} = \\
= \sum_{C_1,C_2,G_1,G_2} \sum_{k,l} (\ooo{a}{b})_{\oP} (x_k\ot y_l) \ot (\ooo{a}{b})_{\oEnd{A}} (\alpha(x_k^\#) \ot \alpha(y_l^\#)) (-1)^{\dg{x_k}\dg{y_l}+\dg{x_k}+\dg{y_l}+\dg{\alpha}\dg{x_k}} = \\
\mathclap{ = \!\! \sum_{\substack{C_1,C_2,\\ G_1,G_2}} \!\! \sum_{k,l} (\!\ooo{a}{b})_{\oP} (\oP(\theta_1)\!\ot\!\oP(\theta_2)) (x'_k\!\ot\! y'_l) \ot (\!\ooo{a}{b})_{\oEnd{A}} (\oEnd{A}(\theta_1)\!\ot\!\oEnd{A}(\theta_2)) (\alpha({x'_k}^\#) \!\ot\! \alpha({y'_l}^\#)) (-1)^{\dg{x_k}\dg{y_l}+\dg{x_k}+\dg{y_l}+\dg{\alpha}\dg{x_k}}, }
\end{gather*}
where $\theta_1,\theta_2$ are arbitrary as in Definition \ref{DEFBVOperaions}; and $x'_k:=\oP(\theta_1^{-1})(x_k)$, thus these form a basis of $\oP([n_1+1],G_1)$ such that $\dg{x'_k}=\dg{x_k}$; and similarly for $\{y'_l\}$.
Rearranging the last expression yields
\begin{gather*}
\sum_{C_1,C_2,G_1,G_2} ((\ooo{a}{b})_{\oP} \ot (\ooo{a}{b})_{\oEnd{A}}) \tau (\oP(\theta_1)\ot\oEnd{A}(\theta_1)\ot\oP(\theta_2)\ot\oEnd{A}(\theta_2)) \\
\left( \sum_k x'_k\ot\alpha({x'_k}^\#) (-1)^{\dg{\alpha}\dg{x'_k}} \ot \sum_l y'_l\ot \alpha({y'_l}^\#) (-1)^{\dg{\alpha}\dg{y'_l}} \right) (-1)^{\epsilon} = \\
= \sum_{\substack{n_1+n_2=n\\ G_1+G_2=G}} \{S_{n_1+1,G_1},S_{n_2+1,G_2}\},
\end{gather*}
where
$\epsilon := \dg{x_k}\dg{y_l}+\dg{x_k}+\dg{y_l}+\dg{\alpha}\dg{x_k}+\dg{\alpha}\dg{x'_k}+\dg{\alpha}\dg{y'_l}+\dg{x'_k}+\dg{y'_l}+\dg{y'_l}(\dg{\alpha}+\dg{x'_k}) \equiv 0 \mod 2$ (recall that $(\ooo{a}{b})_{\oEnd{A}}$ has degree $1$).

We treat $\beta$ and $\gamma$ similarly to find that \eqref{EQAlfOverFTTwo} for $C=[n]$ implies
$$0=Z(\beta)+Z(\gamma)+Z(\delta)=dS_{n,G}+\Delta S_{n+2,G-1}+\frac{1}{2}\sum_{\substack{n_1+n_2=n\\ G_1+G_2=G}} \{S_{n_1+1,G_1},S_{n_2+1,G_2}\}.$$
Now summing this over $n\geq 0$ and $G\geq 0$ yields the master equation.

Conversely, if $S$ is a degree $0$ solution of the master equation, then it gives rises to $\alpha$ satisfying conditions of the Theorem \ref{LEMMAAlgOverFeynTrans}.
Indeed, it suffices to reverse the above reasoning.

Finally, the axioms of generalized BV algebra can be verified by a straightforward computation.
The axioms involving $d$ are easy to prove and we leave them to the reader.
In the following computation, we write briefly $\sigma x$ instead of $\oP(\sigma)x$ or $\oEnd{A}(\sigma)x$; it will always be clear from the context which operad is considered.

As a warmup, we prove $\Delta^2=0$.
It is easy to verify that
\begin{gather} \label{EQDeltaSquare}
\Delta^2 = ((\xi_{ab})_{\oP}(\xi_{cd})_{\oP}\ot(\xi_{ab})_{\oEnd{A}}(\xi_{cd})_{\oEnd{A}})(\theta\ot\theta),
\end{gather}
where $\theta:[n]\to[n-4]\sqcup\{a,b,c,d\}$ is an arbitrary bijection.
Now consider $\sigma$ mapping $a\mapsto c,\ b\mapsto d,\ c\mapsto a,\ d\mapsto b$ and leaving $[n-4]$ intact.
Since $\theta$ is arbitrary, we precompose it with $\sigma$.
Thus \eqref{EQDeltaSquare} equals
$$((\xi_{ab})_{\oP}(\xi_{cd})_{\oP}\ot(\xi_{ab})_{\oEnd{A}}(\xi_{cd})_{\oEnd{A}})(\sigma\theta\ot\sigma\theta).$$
By equivariance of $\xi$, this equals
$$((\xi_{cd})_{\oP}(\xi_{ab})_{\oP}\ot(\xi_{cd})_{\oEnd{A}}(\xi_{ab})_{\oEnd{A}})(\theta\ot\theta).$$
By associativity, this equals
$$-((\xi_{ab})_{\oP}(\xi_{cd})_{\oP}\ot(\xi_{ab})_{\oEnd{A}}(\xi_{cd})_{\oEnd{A}})(\theta\ot\theta).$$
Thus $\Delta^2=0$.

Next, we prove \eqref{EQBVDeltaBracket}.
A tedious but straightforward calculation yields
\begin{gather}
\Delta\{\} = \nonumber \\ 
= \sum_{\substack{|C_1|=n_1-3\\ |C_2|=n_2-1}} ((\xi_{ab})_{\oP}(\ooo{c}{d})_{\oP}\ot(\xi_{ab})_{\oEnd{A}}(\ooo{c}{d})_{\oEnd{A}})(\theta_1\ot\theta_2\ot\theta_1\ot\theta_2)\tau + \nonumber \\
+ 2\sum_{\substack{|C_1|=n_1-2\\ |C_2|=n_2-2}} ((\xi_{ab})_{\oP}(\ooo{c}{d})_{\oP}\ot(\xi_{ab})_{\oEnd{A}}(\ooo{c}{d})_{\oEnd{A}})(\theta'_1\ot\theta'_2\ot\theta'_1\ot\theta'_2)\tau + \label{EQDeltaBracketCalc} \\
+ \sum_{\substack{|C_1|=n_1-1\\ |C_2|=n_2-3}} ((\xi_{ab})_{\oP}(\ooo{c}{d})_{\oP}\ot(\xi_{ab})_{\oEnd{A}}(\ooo{c}{d})_{\oEnd{A}})(\theta''_1\ot\theta''_2\ot\theta''_1\ot\theta''_2)\tau, \nonumber
\end{gather}
where the bijections $\theta_1:[n_1]\to C_1\sqcup\{abc\}, \theta_2:[n_2]\to C_2\sqcup\{d\}, \theta'_1:[n_1]\to C_1\sqcup\{ac\}, \theta'_2:[n_2]\to C_2\sqcup\{bd\}, \theta''_1:[n_1]\to C_1\sqcup\{c\}, \theta''_2:[n_2]\to C_2\sqcup\{abd\}$ are arbitrary.
The terms in the three sums are depicted respectively as
$$
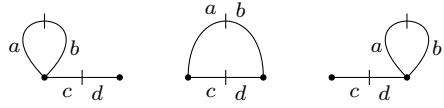
\begin{tikzpicture}[baseline=-\the\dimexpr\fontdimen22\textfont2\relax]
\filldraw (-1,0) circle (1pt);
\filldraw (0,0) circle (1pt);
\draw (-1,0) -- (0,0);
\draw (-1,0) .. controls (-2,1) and (0,1) .. (-1,0);
\draw (-.5,.1) -- (-.5,-.1);
\draw (-1,.85) -- (-1,.65);
\node at (-.7,-.2) {$\scriptstyle{c}$};
\node at (-.3,-.2) {$\scriptstyle{d}$};
\node at (-1.4,.4) {$\scriptstyle{a}$};
\node at (-.6,.4) {$\scriptstyle{b}$};
\end{tikzpicture}
\qquad
\begin{tikzpicture}[baseline=-\the\dimexpr\fontdimen22\textfont2\relax]
\filldraw (-1,0) circle (1pt);
\filldraw (0,0) circle (1pt);
\draw (-1,0) -- (0,0);
\draw (-1,0) .. controls (-1,1) and (0,1) .. (0,0);
\draw (-.5,.1) -- (-.5,-.1);
\draw (-.5,.85) -- (-.5,.65);
\node at (-.7,-.2) {$\scriptstyle{c}$};
\node at (-.3,-.2) {$\scriptstyle{d}$};
\node at (-.7,.9) {$\scriptstyle{a}$};
\node at (-.3,.9) {$\scriptstyle{b}$};
\end{tikzpicture}
\qquad
\begin{tikzpicture}[baseline=-\the\dimexpr\fontdimen22\textfont2\relax]
\filldraw (-1,0) circle (1pt);
\filldraw (0,0) circle (1pt);
\draw (-1,0) -- (0,0);
\draw (0,0) .. controls (-1,1) and (1,1) .. (0,0);
\draw (-.5,.1) -- (-.5,-.1);
\draw (0,.85) -- (0,.65);
\node at (-.7,-.2) {$\scriptstyle{c}$};
\node at (-.3,-.2) {$\scriptstyle{d}$};
\node at (-.4,.4) {$\scriptstyle{a}$};
\node at (.4,.4) {$\scriptstyle{b}$};
\end{tikzpicture}
$$

For the middle term, recall that the expression
\begin{gather} \label{EQExpIndepend}
((\xi_{ab})_{\oP}(\ooo{c}{d})_{\oP}\ot(\xi_{ab})_{\oEnd{A}}(\ooo{c}{d})_{\oEnd{A}})(\theta'_1\ot\theta'_2\ot\theta'_1\ot\theta'_2)
\end{gather}
doesn't depend on the choice of $\theta$'s.
We can thus precompose $\theta'_1$ with a permutation $\sigma_1$ exchanging $a$ and $c$ and leaving everything else in place.
Similarly we precompose $\theta'_2$ with an exchange $\sigma_2$ of $b$ and $d$.
Thus \eqref{EQExpIndepend} equals
\begin{gather*}
((\xi_{ab})_{\oP}(\ooo{c}{d})_{\oP}\ot(\xi_{ab})_{\oEnd{A}}(\ooo{c}{d})_{\oEnd{A}})(\sigma_1\theta'_1\ot\sigma_2\theta'_2\ot\sigma_1\theta'_1\ot\sigma_2\theta'_2) = \\
= ((\xi_{cd})_{\oP}(\ooo{a}{b})_{\oP}\ot(\xi_{cd})_{\oEnd{A}}(\ooo{a}{b})_{\oEnd{A}})(\theta'_1\ot\theta'_2\ot\theta'_1\ot\theta'_2) = \\
= -((\xi_{ab})_{\oP}(\ooo{c}{d})_{\oP}\ot(\xi_{ab})_{\oEnd{A}}(\ooo{c}{d})_{\oEnd{A}})(\theta'_1\ot\theta'_2\ot\theta'_1\ot\theta'_2),
\end{gather*}
where the first equality is justified by equivariance of $\ooo{}{}$ and $\xi$, the second one by associativity.
Thus the middle term of \eqref{EQDeltaBracketCalc} vanishes.

As for \eqref{EQDeltaBracketCalc}, a straightforward calculation yields
\begin{gather*}
\{\}(\Delta\ot\id) = \\ 
= \sum_{\substack{|C_1|=n_1-3\\ |C_2|=n_2-1}} ((\ooo{a}{b})_{\oP}((\xi_{cd})_{\oP}\ot\id)\ot(\ooo{a}{b})_{\oEnd{A}}((\xi_{cd})_{\oEnd{A}}\ot\id))(\overline{\theta_1}\ot\overline{\theta_2}\ot\overline{\theta_1}\ot\overline{\theta_2})\tau
\end{gather*}
for arbitrary bijections $\overline{\theta_1}:[n_1]\to C_1\sqcup\{a,c,d\}$ and $\overline{\theta_2}:[n_2]\to C_2\sqcup\{b\}$.
Let $\sigma_1$ map $a\mapsto c,\ c \mapsto a,\ d\mapsto b$ and leave $C_1$ intact.
Let $\sigma_2$ map $b\mapsto d$ and leave $C_2$ intact.
Equivariance justifies the relabeling occurring below from the first to the second line:
\begin{gather}
((\ooo{a}{b})_{\oP}((\xi_{cd})_{\oP}\ot\id)\ot(\ooo{a}{b})_{\oEnd{A}}((\xi_{cd})_{\oEnd{A}}\ot\id))(\overline{\theta_1}\ot\overline{\theta_2}\ot\overline{\theta_1}\ot\overline{\theta_2}) = \nonumber \\
= ((\ooo{c}{d})_{\oP}((\xi_{ab})_{\oP}\ot\id)\ot(\ooo{c}{d})_{\oEnd{A}}((\xi_{ab})_{\oEnd{A}}\ot\id))(\sigma_1\overline{\theta_1}\ot\sigma_2\overline{\theta_2}\ot\sigma_1\overline{\theta_1}\ot\sigma_2\overline{\theta_2}) = \label{EQBVDeltaBracketClacII} \\
= -((\xi_{ab})_{\oP}(\ooo{c}{d})_{\oP}\ot(\xi_{ab})_{\oEnd{A}}(\ooo{c}{d})_{\oEnd{A}})(\theta_1\ot\theta_2\ot\theta_1\ot\theta_2), \nonumber
\end{gather}
where $\theta_1:=\sigma_1\overline{\theta_1}:[n_1]\to C_1\sqcup\{a,b,c\}$ and $\theta_2:=\sigma_2\overline{\theta_2}:[n_2]\to C_2\sqcup\{d\}$.
So $\{\}(\Delta\ot\id)$ cancels with the first term of \eqref{EQDeltaBracketCalc}.
Likewise, $\{\}(\id\ot\Delta)$ cancels with the third term.

To prove the Jacobi identity \eqref{EQBVJacobi}, let $f\in P(n_1,G_1), g\in P(n_2,G_2), c\in P(n_3,G_3)$.
Then $\{\{f,g\},h\}$ can be expressed as a sum of terms of two types:
$$
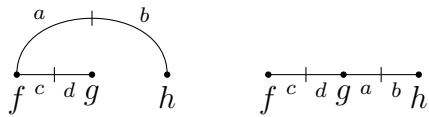
\begin{tikzpicture}[baseline=-\the\dimexpr\fontdimen22\textfont2\relax]
\filldraw (-1,0) circle (1pt) node[below]{$f$};
\filldraw (0,0) circle (1pt) node[below]{$g$};
\filldraw (1,0) circle (1pt) node[below]{$h$};
\draw (-1,0) -- (0,0);
\draw (-1,0) .. controls (-1,1) and (1,1) .. (1,0);
\draw (-.5,.1) -- (-.5,-.1);
\draw (0,.85) -- (0,.65);
\node at (-.7,-.2) {$\scriptstyle{c}$};
\node at (-.3,-.2) {$\scriptstyle{d}$};
\node at (-.7,.8) {$\scriptstyle{a}$};
\node at (.7,.8) {$\scriptstyle{b}$};
\end{tikzpicture}
\qquad
\begin{tikzpicture}[baseline=-\the\dimexpr\fontdimen22\textfont2\relax]
\filldraw (-1,0) circle (1pt) node[below]{$f$};
\filldraw (0,0) circle (1pt) node[below]{$g$};
\filldraw (1,0) circle (1pt) node[below]{$h$};
\draw (-1,0) -- (0,0) -- (1,0);
\draw (-.5,.1) -- (-.5,-.1);
\draw (.5,.1) -- (.5,-.1);
\node at (-.7,-.2) {$\scriptstyle{c}$};
\node at (-.3,-.2) {$\scriptstyle{d}$};
\node at (.7,-.2) {$\scriptstyle{b}$};
\node at (.3,-.2) {$\scriptstyle{a}$};
\end{tikzpicture}
$$
Proceeding formally, a tedious but straightforward calculation shows that
\begin{gather}
\{\{f,g\},h\} = \label{EQDoubelBracket} \\
= \sum_{\mathclap{\substack{|C_1|=n_1-2,\\ |C_2|=n_2-1,\\ |C_3|=n_3-1}}} \left( (\ooo{a}{b})_{\oP}((\ooo{c}{d})_{\oP}\ot\id)(\theta_1\!\ot\!\theta_2\!\ot\!\theta_3) \ot (\ooo{a}{b})_{\oEnd{A}}((\ooo{c}{d})_{\oEnd{A}}\ot\id)(\theta_1\!\ot\!\theta_2\!\ot\!\theta_3) \right) \psi (f\!\ot\! g\!\ot\! h) +{} \nonumber \\
{}+ \sum_{\mathclap{\substack{|C_1|=n_1-1,\\ |C_2|=n_2-2,\\ |C_3|=n_3-1}}} \left( (\ooo{a}{b})_{\oP}((\ooo{c}{d})_{\oP}\ot\id)(\theta'_1\!\ot\!\theta'_2\!\ot\!\theta'_3) \ot (\ooo{a}{b})_{\oEnd{A}}((\ooo{c}{d})_{\oEnd{A}}\ot\id)(\theta'_1\!\ot\!\theta'_2\!\ot\!\theta'_3) \right) \psi (f\!\ot\! g\!\ot\! h), \nonumber
\end{gather}
where the summations run over all partitions $C_1\sqcup C_2\sqcup C_3=[n_1+n_2+n_3-4]$, the permutations $\theta_1:[n_1]\to C_1\sqcup\{a,c\}$ (or $\theta'_1:[n_1]\to C_1\sqcup\{c\}$ in the second sum), $\theta_2:[n_2]\to C_2\sqcup\{d\}$ (or $\theta'_2:[n_2]\to C_2\sqcup\{a,d\}$ in the second sum) and $\theta_3:[n_3]\to C_3\sqcup\{b\}$ are arbitrary and $\psi$ is the monoidal symmetry 
\begin{gather*}
\oP(n_1,G_1)\ot\oEnd{A}(n_1,G_1)\ot\oP(n_2,G_2)\ot\oEnd{A}(n_2,G_2)\ot\oP(n_3,G_3)\ot\oEnd{A}(n_3,G_3) \to \\
\to \oP(n_1,G_1)\ot\oP(n_2,G_2)\ot\oP(n_3,G_3)\ot\oEnd{A}(n_1,G_1)\ot\oEnd{A}(n_2,G_2)\ot\oEnd{A}(n_3,G_3).
\end{gather*}
Denote $F_\id$ resp. $S_\id$ the terms in the first resp. second sum in \eqref{EQDoubelBracket}, i.e.
$$\{\}(\{\}\ot\id)(f\ot g\ot h) = F_\id + S_\id.$$
Let $\sigma=\perm{1 & 2 & 3}{2 & 3 & 1}\in\Sigma_3$ be the cyclic permutation and similarly let
\begin{gather*}
\{\}(\{\}\ot\id)\sigma(f\ot g\ot h) = F_\sigma + S_\sigma,\\
\{\}(\{\}\ot\id)\sigma^2(f\ot g\ot h) = F_{\sigma^2} + S_{\sigma^2}.
\end{gather*}
We want to prove $\{\}(\{\}\ot\id)(\id+\sigma+\sigma^2)(f\ot g\ot h)=0$.
For $F_\id$, let $(\theta_1\ot\theta_1)f=\sum_i f'_i\ot f''_i$ for some $f'_i\in\oP(C_1\sqcup\{a,c\},G_1), f''_i\in\oEnd{A}(C_1\sqcup\{a,c\},G_2)$ and similarly for $g$ and $h$.
By axioms of operad,
$$\ooo{a}{b}(\ooo{c}{d}\ot\id)(f'_i\ot g'_i\ot h'_i) = \ooo{b}{a}(\id\ot\ooo{c}{d})\sigma(f'_i\ot g'_i\ot h'_i) = \ooo{c}{d}(\ooo{b}{a}\ot\id)\sigma(f'_i\ot g'_i\ot h'_i).$$
Similarly
$$\ooo{a}{b}(\ooo{c}{d}\ot\id)(f''_i\ot g''_i\ot h''_i) = -\ooo{c}{d}(\ooo{b}{a}\ot\id)\sigma(f''_i\ot g''_i\ot h''_i).$$
The indices $a,b,c,d$ along which we contract can be arbitrarily renamed as in \eqref{EQBVDeltaBracketClacII}.
This shows that $F_\id=-S_\sigma$.
Similarly, $F_\sigma=-S_{\sigma^2}$ and $F_{\sigma^2}=-S_\id$.
\end{proof}

\begin{remark}
In physical literature it is customary to introduce a formal parameter $\hbar$ whose powers bookkeep the genus.
E.g. we might redefine $P:=\prod_{\substack{n\geq 0\\ G\geq 0}} \hbar^G P(n,G)$ and extend all BV operations $\fld[[\hbar]]$-linearly; this forces us to replace $\Delta$ by $\hbar\Delta$.
We won't use this convention except in the special case of Section \ref{SECMEQC}.
\end{remark}

Now we will rewrite $P$ and the BV operations into a different form.
We note that we are using the convention of Theorem \ref{THMEndIsOperad}, which means that $\oEnd{A}(n,G)$ is $\op{Hom}_{\fld}(A^{\ot n},\fld)$ rather than $\op{Hom}_{\fld}(\bigotimes_{[n]}A,\fld)$.
There is an iso
$$\op{Hom}_{\fld}(A^{\ot n},\fld) \cong {A^{\#}}^{\ot n}$$
and, since we are in characteristics $0$, there is an iso between invariants and coinvariants
$$\left(\oP(n,G)\ot\oEnd{A}(n,G)\right)^{\Sigma_n} \cong \left(\oP(n,G)\ot\oEnd{A}(n,G)\right)_{\Sigma_n} = \oP(n,G)\ot_{\Sigma_n}\oEnd{A}(n,G).$$
Composing these two, we obtain
\begin{align}
\left(\oP(n,G)\ot\oEnd{A}(n,G)\right)^{\Sigma_n} &\cong \oP(n,G)\ot_{\Sigma_n}{A^{\#}}^{\ot n}, \nonumber \\
\sum_i p_i\ot f_i &\mapsto \frac{1}{n!}\sum_{i,I} p_i \ot_{\Sigma_n} f_i(a_I)\phi^I, \label{EQIsoYetOneMoreOhh} \\
\sum_{\sigma\in\Sigma_n}\oP(\sigma)p\ot\oEnd{A}(\sigma)\phi^I &\mapsfrom p\ot_{\Sigma_n}\phi^I, \nonumber 
\end{align}
where $I:=I_1\cdots I_n:=(I_1,\ldots,I_n)$ runs over all multiindices in $[\dim A]^{\times n}$, $a_I:=a_{I_1}\ot\cdots\ot a_{I_n} \in A^{\ot n}$, $\{\phi^i\}$ is the basis dual to $\{a_i\}$ and $\phi^I:=\phi^{I_1}\ot\cdots\ot\phi^{I_n} \in {A^{\#}}^{\ot n}$.
In the bottom left, we have also denoted by $\phi^I$ the map $A^{\ot n}\to\fld$ given by
$$\phi^I(a_J)=\delta^I_J \quad \textrm{for any }a_J\in A^{\ot n}.$$
We will use this notation in the sequel; it will always be clear from the context which of the two meanings we have in mind.
Denote
\begin{gather*}
\tilde{P}(n,G) := \oP(n,G) \ot_{\Sigma_n} {A^{\#}}^{\ot n} \\
\tilde{P} := \prod_{n,G} \tilde{P}(n,G).
\end{gather*}
Then $P\cong\tilde{P}$ as vector spaces and we transfer the BV operations to $\tilde{P}$:
let $Y$ be the iso \eqref{EQIsoYetOneMoreOhh}, then $d$ and $\Delta$ are transferred along 
\begin{equation*} 
\begin{tikzpicture} [baseline=-\the\dimexpr\fontdimen22\textfont2\relax]
\matrix (m) [matrix of math nodes, row sep=3em, column sep=2.5em, text height=1.5ex, text depth=0.25ex]
{ P(n+2,G) & & \tilde{P}(n+2,G) \\
P(n,G+1) & & \tilde{P}(n,G+1) \\ };
\path[->,font=\scriptsize] (m-1-3) edge node[above]{$Y^{-1}$} (m-1-1);
\path[->,font=\scriptsize] (m-1-1) edge node[left]{$\Delta$} (m-2-1);
\path[->,dashed,font=\scriptsize] (m-1-3) edge node[right]{$\Delta$} (m-2-3);
\path[->,font=\scriptsize] (m-2-1) edge node[above]{$Y$} (m-2-3);
\end{tikzpicture}
\end{equation*}
while $\{\}$ along
\begin{equation*} 
\begin{tikzpicture} [baseline=-\the\dimexpr\fontdimen22\textfont2\relax]
\matrix (m) [matrix of math nodes, row sep=3em, column sep=2.5em, text height=1.5ex, text depth=0.25ex]
{ P(n_1+1,G_1)\ot P(n_2+1,G_2) & & \tilde{P}(n_1+1,G_1)\ot\tilde{P}(n_2+1,G_2) \\
P(n_1+n_2,G_1+G_2) & & \tilde{P}(n_1+n_2,G_1+G_2) \\ };
\path[->,font=\scriptsize] (m-1-3) edge node[above]{$Y^{-1}\ot Y^{-1}$} (m-1-1);
\path[->,font=\scriptsize] (m-1-1) edge node[left]{$\{\}$} (m-2-1);
\path[->,dashed,font=\scriptsize] (m-1-3) edge node[right]{$\{\}$} (m-2-3);
\path[->,font=\scriptsize] (m-2-1) edge node[above]{$Y\ot Y$} (m-2-3);
\end{tikzpicture}
\end{equation*}
The formulas can be conveniently expressed in terms of ``positional derivations'':
Let
$$\frac{\partial^{(i)}}{\partial \phi^j} \phi^{I_1\cdots I_m} := (-1)^{\dg{\phi^j}(\dg{\phi^{I_1}}+\cdots+\dg{\phi^{I_{i-1}}})} \delta^{I_i}_j \phi^{I_1\cdots\hat{I_i}\cdots I_m}$$
define a linear map ${A^{\#}}^{\ot m}\to{A^{\#}}^{\ot m-1}$.

\begin{lemma} \label{LEMBarannikovsIsoFormulas}
\begin{gather*}
d(p\ot_{\Sigma_n}\phi^I) = d_{\oP([n],G)}p\ot_{\Sigma_n}\phi^I + (-1)^{1+\dg{p}+\dg{\phi^I}}p\ot_{\Sigma_n}d_{{A^\#}^{\ot n}}\phi^I, \\
\Delta(p\ot_{\Sigma_{n+2}}\phi^I) = 2\sum_{\mathclap{\substack{1\leq i<j\leq n+2\\ d,e}}} \overline{\xi}_{ij}(p) \ot_{\Sigma_n} \frac{\partial^{(i)}}{\partial \phi^d} \frac{\partial^{(j)}}{\partial \phi^e} \phi^I \ (-1)^{\dg{\phi^I}+\dg{a_e}} \omega^{de}, \\
\{p\ot_{\Sigma_{n_1+1}}\phi^I,q\ot_{\Sigma_{n_2+1}}\phi^J\} = \\
= \sum_{\mathclap{\substack{1\leq i\leq n_1+1\\ 1\leq j\leq n_2+1\\ d,e}}} \ooot{i}{j}(p\ot q) \ot_{\Sigma_{n_1+n_2}} \frac{\partial^{(i)}}{\partial \phi^d} \phi^I \ot \frac{\partial^{(j)}}{\partial \phi^e} \phi^J \ \omega^{de} (-1)^{\dg{a_d}(\dg{\phi^I}+1)+\dg{\phi^I}\dg{q}+\dg{\phi^J}+\dg{\phi^I}\dg{\phi^J}+1}.
\end{gather*}
\end{lemma}


\begin{proof}
We do the computation for $\Delta$ and leave the other formulas for the reader.
\begin{gather} \label{EQTransfDeltaI}
\Delta(p\ot_{\Sigma_{n+2}}\phi^I) = \frac{1}{n!}\sum_{\substack{J\\ \sigma\in\Sigma_{n+2}}} \xi_{ab}\oP(\theta\sigma)(p) \ot_{\Sigma_n} \xi_{ab}\oEnd{A}(\theta\sigma)(\phi^I)(a_J)\phi^J,
\end{gather}
where $|I|=n+2$, $|J|=n$ and $\sigma:[n+2]\to[n]\sqcup\{a,b\}$ is an arbitrary bijection.
By definitions,
$$\xi_{ab}\oP(\theta\sigma)(p) = \oP(\theta\sigma{\rho^{ij}}^{-1}) \overline{\xi}_{ij}(p),$$
where $i:=(\theta\sigma)^{-1}(a)$, $j:=(\theta\sigma)^{-1}(b)$.
Denote $\kappa:=\theta\sigma{\rho^{ij}}^{-1}$.
Similarly
\begin{gather*}
\xi_{ab}\oEnd{A}(\theta\sigma)(\phi^I) = \oEnd{A}(\kappa)\overline{\xi}_{ij}(\phi^I) = \overline{\xi}_{ij}(\phi^I)\kappa^{-1}.
\end{gather*}
Consequently \eqref{EQTransfDeltaI} equals
\begin{gather}
\frac{1}{n!}\sum_{J,\sigma} \oP(\kappa) \overline{\xi}_{ij}(p) \ot_{\Sigma_n} \overline{\xi}_{ij}(\phi^I)(\kappa^{-1}a_J)\phi^J = \nonumber \\
= \frac{1}{n!}\sum_{J,\sigma} \overline{\xi}_{ij}(p) \ot_{\Sigma_n} \overline{\xi}_{ij}(\phi^I)(\kappa^{-1}a_J) \kappa^{-1}\phi^J = \nonumber \\
= \frac{1}{n!}\sum_{J,\sigma} \overline{\xi}_{ij}(p) \ot_{\Sigma_n} \overline{\xi}_{ij}(\phi^I)(a_J)\phi^J. \label{EQTransfDeltaII}
\end{gather}
The last equality is justified as follows: $\kappa^{-1}a_J=\pm a_{\kappa^{-1}J}$ and $\kappa^{-1}\phi^J=\pm\phi^{\kappa^{-1}J}$ with the same sign; thus we can shift the summation index $J$ to $\kappa^{-1}J$.
By definitions,
\begin{gather}
\overline{\xi}_{ij}(\phi^I)(a_J) = \nonumber \\
= \sum_{d,e} (-1)^{\dg{\phi^I}+\dg{a_e}} \omega^{de} \phi^I(\cdots\ot\underbrace{a_d}_{\mathclap{i\textrm{-th}}}\ot\cdots\ot\underbrace{a_e}_{\mathclap{j\textrm{-th}}}\ot\cdots)(a_J) = \nonumber \\
= \sum_{d,e} \omega^{de} \delta^I_{J'} (-1)^{\dg{\phi^I}+\dg{a_e}+\epsilon}, \label{EQKoszulAuxilia}
\end{gather}
where $J':=J_1 \cdots J_{i-1} d J_{i} \cdots J_{j-1} e J_{j} \cdots J_{n}$ is obtained from $J$ by inserting $d, e$ into positions $i,j$; and $(-1)^\epsilon$ is the Koszul sign of transforming $a_{deJ}$ to $a_{J'}$.
Substituting this into \eqref{EQTransfDeltaII} yields
$$\frac{1}{n!}\sum_{\sigma} \overline{\xi}_{ij}(p) \ot_{\Sigma_n} \phi^{I_1\cdots \hat{I}_{i}\cdots \hat{I}_{j}\cdots I_{n+2}} \omega^{I_{i}I_{j}} (-1)^{\dg{\phi^I}+\dg{a_{I_{j}}}+\epsilon}.$$
We can rewrite this expression in terms of the positional derivations:
$$\frac{\partial^{(i')}}{\partial \phi^{I_{i}}} \frac{\partial^{(j)}}{\partial \phi^{I_{j}}} \phi^I = (-1)^{\epsilon} \phi^{I_1\cdots\hat{I}_{i}\cdots\hat{I}_{j}\cdots I_{n+2}},$$
where we wish to derivate at positions $i$ and $j$ of $\phi^I$, so the position index of the LHS derivation is $i':=i$ if $i<j$ and $i':=i-1$ otherwise.
The sign $(-1)^\epsilon$ is indeed the one from \eqref{EQKoszulAuxilia}.
It it easy to verify, using the fact that $\omega$ has degree $1$, that the sign coincides in both cases $i<j$ and $i>j$.
We have obtained
\begin{gather*}
\frac{1}{n!}\sum_{\sigma} \overline{\xi}_{ij}(p) \ot_{\Sigma_n} \frac{\partial^{(i')}}{\partial \phi^{I_{i}}} \frac{\partial^{(j)}}{\partial \phi^{I_{j}}} \phi^I \ (-1)^{\dg{\phi^I}+\dg{a_{I_{j}}}} \omega^{I_{i}I_{j}}
\end{gather*}
In this formula, the only dependence of the summands on $\sigma$ is through $i=(\theta\sigma)^{-1}(a)$ and $j=(\theta\sigma)^{-1}(a)$.
Thus we replace $\sum_{\sigma\in\Sigma_{n+2}}$ by $n!\sum_{\substack{1\leq i,j\leq n+2\\ i\neq j}}$ to obtain:
\begin{gather*}
\sum_{\substack{1\leq i,j\leq n+2\\ i\neq j}} \overline{\xi}_{ij}(p) \ot_{\Sigma_n} \frac{\partial^{(i')}}{\partial \phi^{I_i}} \frac{\partial^{(j)}}{\partial \phi^{I_j}} \phi^I \ (-1)^{\dg{\phi^I}+\dg{a_{I_j}}} \omega^{I_iI_j} = \\
= \sum_{\substack{1\leq i,j\leq n+2\\ i\neq j\\ d,e}} \overline{\xi}_{ij}(p) \ot_{\Sigma_n} \frac{\partial^{(i')}}{\partial \phi^d} \frac{\partial^{(j)}}{\partial \phi^e} \phi^I \ (-1)^{\dg{\phi^I}+\dg{a_e}} \omega^{de}.
\end{gather*}
Finally, the summands are invariant under the exchange of $i,j$, since $\frac{\partial^{(i')}}{\partial \phi^d} \frac{\partial^{(j)}}{\partial \phi^e} = \frac{\partial^{(j')}}{\partial \phi^e} \frac{\partial^{(i)}}{\partial \phi^d}$ unless $\omega^{de}=0$; and since $(-1)^{\dg{a_e}}\omega^{de}=(-1)^{\dg{a_d}}\omega^{ed}$.
The conclusion follows.
\end{proof}

\section{\texorpdfstring{The operad $\oQC$ and related algebraic structures}{The operad QC and related algebraic structures}} \label{SECTQC}

Solutions of Zwiebach's Master Equation \cite{ZwiebachClosed} for closed string theory are equivalently described by a collection of multilinear maps satisfying certain properties.
The resulting algebraic structure is called loop homotopy Lie algebra in \cite{MarklLoop}.
The complicated axioms are succinctly described by operads.
In this section, we rephrase the results of M. Markl \cite{MarklLoop} along these lines in our formalism.

\subsection{\texorpdfstring{The modular operad $\oQC$}{The modular operad QC}}

\medskip

We define the modular operad $\oQC$, called Quantum Closed operad, to consist of homeomorphism classes of connected $2$-dimensional compact orientable surfaces with labeled boundary components.
The homeomorphism class is determined by the genus of the surface, hence, for example, $\oQC([3],1)$ is generated by
$$\PICA$$
Each boundary component is a circle and surfaces can be glued along these circles:
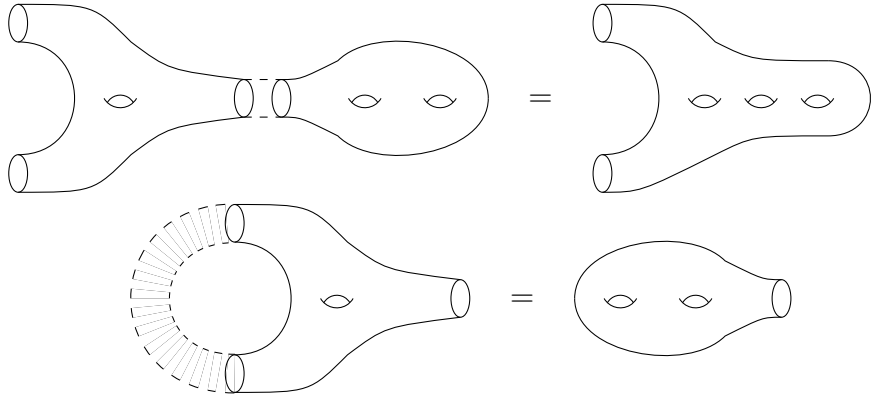
\begin{gather*}
\PICB \quad = \quad \PICC \\
\PICD \quad = \quad \PICE
\end{gather*}
Bijections acts on surfaces by relabeling the boundary components.
A formal definition follows:

\begin{definition}
For a corolla $(C,G)$, let $\oQC(C,G)$ be just the one dimensional space generated by a symbol $C^G$ of degree $0$:
$$\oQC(C,G) := \Span_{\fld}\left\{ C^G \right\}.$$
The structure operations are defined, for any bijection $\rho:C\xrightarrow{\sim}D$, as follows:
\begin{gather*}
	\oQC(\rho)(C^G) := D^{G} \\
	\ooo{a}{b} \left( (C_1\sqcup\{a\})^{G_1} \ot (C_2\sqcup\{b\})^{G_2} \right) := (C_1\sqcup C_2)^{G_1+G_2} \\
	\xi_{ab}((C\sqcup\{a,b\})^{G}) := C^{G+1}
\end{gather*}
\end{definition}

Obviously, $\oQC$ is a modular operad.

\subsection{\texorpdfstring{Loop homotopy algebras}{Loop homotopy algebras}} \label{SECTLoopHomAlg}

\medskip

\begin{theorem}[\cite{MarklLoop}, \cite{MarklOperads} Chapter III$.5.7$] \label{THMMarklLoopViaOperads}
The algebras over the Feynman transform $\feyn{\oQC}$ are loop homotopy algebras.
\end{theorem}

For the definition of loop homotopy algebras, we refer the reader to loc. cit.
To prove the theorem, one first makes axioms of algebras over $\feyn{\oQC}$ explicit in terms of operations $V^{\ot n}\to\fld$.
Second, one passes to operations $V^{\ot n-1}\to V$ and uses standard suspension isomorphisms for multilinear maps to translate the axioms.
Here we redo the first step, since it will appear later in a more complicated context.

To lighten the notation, we identify each $\oQC([n],G)^{\#}$ with $\oQC([n],G)$ by identifying $[n]^G$ with its dual.
Applying Lemma \ref{LEMMAHardToUse} and the fact that each $\oQC([n],G)$ is $1$-dimensional, we see that an algebra over $\feyn{\oQC}$ is uniquely determined by a collection 
$$\set{f_n^G:=\alpha([n],G)([n]^G)\in\oEnd{A}(n,G)}{n\geq 0,\ 2(G-1)+n>0}$$
of degree $0$ linear maps satisfying:
\begin{gather}
	\oEnd{A}(\rho)(f_n^G) = f_n^G \quad\textrm{for any }\rho\in\Sigma_n \textrm{ and} \nonumber \\
	d(f_n^G) = (\overline{\xi}_{n+1,n+2})_{\oEnd{A}}\alpha \oQC(\kappa) (\xi_{ab})^{\#}_{\oQC} ([n]^G) +{} \label{EQFeynmanZerothTime} \\
	\mathclap{ {} + \frac{1}{2} \sum_{\mathclap{\substack{{}\\ C_1\sqcup C_2=[n] \\ G_1+G_2=G}}} \oEnd{A}(\rho_1^{1}\kappa_1|_{C_1}\sqcup\rho_2^{1}\kappa_2|_{C_2})^{-1} (\ooot{n_1+1}{n_2+1})_{\oEnd{A}} (\alpha\!\ot\!\alpha) (\oQC(\kappa_1)\!\ot\!\oQC(\kappa_2)) (\oooo{a}{b}{C_1\sqcup\{a\},G_1}{C_2\sqcup\{b\},G_2})^{\#}_{\oQC} [n]^G,} \nonumber
\end{gather}
where we have chosen $\kappa$ to be increasing on $C$ and $\kappa(a)=n+1, \kappa(b)=n+2$; and we have chosen $\kappa_1$ to be increasing on $C_1$ and $\kappa_1(a)=n_1+1$ and similarly for $\kappa_2$.
Observe that $(\rho_1^{n_1+1}\kappa_1|_{C_1}\sqcup\rho_2^{n_2+1}\kappa_2|_{C_2})^{-1}$ is a $(|C_1|,|C_2|)$-shuffle, denote it $\sigma$.

Now we want to express the RHS of \eqref{EQFeynmanZerothTime} in terms of $f_n^G$'s.
We start by calculating the dual of the structure morphisms of $\oQC$:
\begin{gather*}
(\xi_{ab})^{\#}_{\oQC}[n]^G = ([n]\sqcup\{a,b\})^{G-1}, \\
(\oooo{a}{b}{C_1\sqcup\{a\}}{C_2\sqcup\{b\}})^{\#}_{\oQC} [n]^G = \sum_{\substack{G_1,G_2\\ G_1+G_2=G\\ 2(G_1-1)+|C_1|>0\\ 2(G_2-1)+|C_2|>0}} (C_1\sqcup\{a\})^{G_1} \ot (C_2\sqcup\{b\})^{G_2}.
 \end{gather*}
Thus \eqref{EQFeynmanZerothTime} becomes
\begin{gather*}
d(f_n^G) = (\overline{\xi}_{n+1,n+2})_{\oEnd{A}} f_{n+2}^{G-1} + \frac{1}{2} \sum_{\substack{n_1,n_2,\sigma \\ G_1,G_2,d}}  \oEnd{A}(\sigma) (\ooot{n_1+1}{n_2+1})_{\oEnd{A}} (f_{n_1+1}^{G_1}\ot f_{n_2+1}^{G_2}) = \\
= \sum_d f_{n+2}^{G-1}(\id^{\ot n}\ot a_d\ot b_d) + \frac{1}{2} \sum_{\substack{n_1,n_2,\sigma \\ G_1,G_2,d}} (f_{n_1+1}^{G_1}(\id^{\ot n_1}\ot a_d) \cdot f_{n_2+1}^{G_2}(\id^{\ot n_2}\ot b_d)) \circ \sigma^{-1}.
\end{gather*}

We have thus proved:
\begin{theorem}\label{THMLoop}
An algebra over $\feyn{\oQC}$ on a dg symplectic vector space $A$ is equivalently given by a collection
$$\set{f_n^G:A^{\ot n}\to\fld}{n,G\geq 0,\ 2(G-1)+n>0}$$
of degree $0$ completely symmetric linear maps satisfying the equations
\begin{gather*}
d(f_n^G)(x_1,\ldots,x_n) = \sum_{d} \pm f_{n+2}^{G-1}(x_1,\ldots,x_n,a_d,b_d) +{} \\
 {}+ \sum_{\substack{n_1+n_2=n \\ \sigma\in\Sh{n_1,n_2}}} \! \sum_{\substack{G_1,G_2\\ G_1+G_2=G\\ 2(G_1-1)+|C_1|>0\\ 2(G_2-1)+|C_2|>0}} \hspace{-1em} \sum_{d} \pm f_{n_1+1}^{G_1}(x_{\sigma(1)},\ldots,x_{\sigma(n_1)},a_d) \cdot f_{n_2+1}^{G_2}(x_{\sigma(n_1+1)},\ldots,x_{\sigma(n_1+n_2)},b_d)
\end{gather*}
for every $x_1,\ldots,x_n\in A$, where $\pm$ is the Koszul sign of the evaluation $f_{n+2}^{G-1}(\id^{\ot n}\ot a_d\ot b_d)(x_1\ot\cdots\ot x_n)$ resp. $\left( \left( f_{n_1+1}^{G_1}(\id^{\ot n_1}\ot a_d) \cdot f_{n_2+1}^{G_2}(\id^{\ot n_2}\ot b_d) \right) \circ \sigma^{-1} \right) (x_1\ot\cdots\ot x_n)$.
\qed
\end{theorem}

\subsection{\texorpdfstring{Relation between $\textbf{Mod}(\oCom)$ and $\oQC$}{Relation between Mod(\oCom) and QC}}

The operad $\oQC$ has a nice algebraic interpretation.
First observe that restricting to the genus zero part of the operad $\oQC$ yields the cyclic operad $\oCom$.
Recall that the modular envelope $\textbf{Mod}$ is the left adjoint of the forgetful functor from modular to cyclic operads.

\begin{theorem}[\cite{MarklLoop},\cite{MarklOperads}] \label{THMModCom}
The modular operad $\oQC$ is the modular envelope of the cyclic operad $\oCom$: $$\oQC \cong \oMod{\oCom}.$$
\end{theorem}

Theorem \ref{THMModCom} has a physical interpretation:
The cyclic operad $\oCom$ in fact consist of homeomorphism classes of $2$-dimensional orientable surfaces with labeled boundary components and \emph{genus zero} and the composition is gluing of the surfaces along a pair of boundary components.
Thus $\oCom$ is the loopless part of $\oQC$ and algebras over $\oCom$ are classical limits of algebras over $\oQC$.
In other words, the passage from classical to quantum case corresponds to taking modular envelope.

\begin{gather*}
\PICF \in \oCom(\{a,b,c\}) \\
\oQC(\{c\},1) \ni \PICG \quad \cong \quad \xi_{ab}\left(\PICF\right) \in \oMod{\oCom}(\{c\},1)
\end{gather*}

\begin{remark}
Restricting the Feynman transform for modular operads to cyclic cobar complex for cyclic operads and applying it to $\oCom$ yields $L_\infty$ algebras after some suspensions.
Precise relation between loop homotopy Lie and $L_\infty$ algebras is explained in detail in \cite{MarklLoop}.
We will discuss an analogue for the operad $\oAss$ later.
\end{remark}

\subsection{\texorpdfstring{Master equation}{Master equation}} \label{SECMEQC}

By applying Barannikov's theory of Section \ref{SSECBarannikov}, we can get Zwiebach's master equation \cite{ZwiebachClosed} for closed string theory directly:

Since $\oQC(C,G)$ is the trivial representation $\Span_{\fld} C^{G}$,
\begin{gather} \label{EQSecondIdentificationQC}
\oQC(n,G) \ot_{\Sigma_n} (A^{\#})^{\ot n} \cong S^n(A^{\#}),\textrm{ the }n\textrm{-th symmetric power},
\end{gather}
and $\tilde{P} \cong \widehat{S}(A^{\#}) := \prod_{n,g}S^n(A^{\#})$.
Let's write $\hbar^G \phi^I$ rather than $[n]^G \ot_{\Sigma_n} \phi^I$.
Then $\phi$'s posses a symmetry as if they were graded polynomial variables.

The BV differential is $\hbar$-linear (in the obvious sense) and is determined by
$$d(\phi^I) = (-1)^{1+\dg{\phi^I}} d_{{A^\#}^{\ot n}}\phi^I.$$
To use the formulas of Lemma \ref{LEMBarannikovsIsoFormulas} for $\Delta$ and $\{\}$, we need to make $\overline{\xi}_{ij}$ and $\ooot{i}{j}$ on $\oQC$ explicit:
\begin{gather*}
\overline{\xi}_{ij}([n+2]^G) = [n]^{G+1}, \\
\ooot{i}{j}([n_1+1]^{G_1}\ot [n_2+1]^{G_2}) = [n_1+n_2]^{G_1+G_2}.
\end{gather*}
Then $\Delta$ is $\hbar$-linear and
\begin{gather} \label{EQQCDeltaExplicit}
\Delta(\phi^{I_1\cdots I_n}) = 2\hbar\sum_{i<j} \pm \omega^{I_iI_j} \phi^{I_1\cdots\widehat{I_i}\cdots\widehat{I_j}\cdots I_n},
\end{gather}
where the sign consists of $(-1)^{\dg{\phi^I}+\dg{\phi^{I_i}}}$ and the Koszul sign of permutation taking $\phi^{I_1\cdots I_n}$ to $\phi^{I_iI_jI_1\cdots\widehat{I_i}\cdots\widehat{I_j}\cdots I_n}$;
and $\{\}$ is $\hbar$-linear and
\begin{gather} \label{EQQCBracketExplicit}
\{\phi^{I_1\cdots I_p},\phi^{J_1\cdots J_q}\} = \sum_{i,j} \pm \omega^{I_iJ_j} \phi^{I_1\cdots\widehat{I_i}\cdots I_pJ_1\cdots\widehat{J_j}\cdots J_q},
\end{gather}
where the sign consists of $(-1)^{\dg{\phi^{I_i}}(\dg{\phi^I}+1)+\dg{\phi^J}+\dg{\phi^I}\dg{\phi^J}+1}$ and the Koszul sign of permutation taking $\phi^{I_1\cdots I_pJ_1\cdots J_q}$ to $\phi^{I_iI_1\cdots\widehat{I_i}\cdots I_pJ_jJ_1\cdots\widehat{J_j}\cdots J_q}$.
The solutions $S\in\widehat{S}(A^{\#})$ of degree $0$ of the master equation 
$$d(S)+\Delta(S)+\frac{1}{2}\{S,S\}=0$$
thus correspond to algebras over $\feyn{\oQC}$ on $A$.

We wish to compare the above formulas for BV operations to those of Zwiebach in \cite{ZwiebachClosed}.
We introduce left and right derivations 
\begin{gather*}
\frac{\partial_L}{\partial \phi^j} \phi^I := \sum_{i=1}^{|I|} \frac{\partial^{(i)}}{\partial \phi^j} \phi^I, \\
\frac{\partial_R}{\partial \phi^j} \phi^I := \sum_{i=1}^{|I|} (-1)^{\dg{\phi^j}\dg{\phi^I}+\dg{\phi^j}\dg{\phi^{I_i}}} \frac{\partial^{(i)}}{\partial \phi^j} \phi^I
\end{gather*}
with respect to the standard $\hbar$-linear commutative multiplication in $\widehat{S}(A^{\#})$.
Observe that $\frac{\partial_R}{\partial \phi^j} \phi^I = (-1)^{\dg{\phi^j}(\dg{\phi^I}+1)} \frac{\partial_L}{\partial \phi^j} \phi^I$.
Thus we obtain
\begin{gather*}
\Delta(\phi^I) = \hbar \sum_{d,e} (-1)^{1+\dg{\phi^I}\dg{a_e}} \omega^{de} \frac{\partial_R}{\partial \phi^d} \frac{\partial_L}{\partial \phi^e}\phi^I, \\
\{\phi^I,\phi^J\} = \sum_{d,e} (-1)^{\dg{\phi^J}+\dg{\phi^I}\dg{\phi^J}+1} \omega^{de} \frac{\partial_R}{\partial \phi^d} \phi^I \cdot \frac{\partial_L}{\partial \phi^e} \phi^J.
\end{gather*}

It is now natural to ask what is the compatibility between the commutative multiplication $\cdot$ and $\Delta,\{\}$.
One easily verifies
$$\Delta(\phi^I\cdot\phi^J)-(-1)^{\dg{\phi^J}}\Delta(\phi^I)\cdot\phi^J-\phi^I\cdot\Delta(\phi^J)=2\{\phi^J,\phi^I\}.$$
Thus in order to get a classical BV algebra, we redefine the operations as follows:
\begin{gather*}
d'(\phi^I):=\frac{1}{2}(-1)^{\dg{\phi^I}}d(\phi^I), \\
\Delta'(\phi^I):=\frac{1}{2}(-1)^{\dg{\phi^I}}\Delta(\phi^I), \\
\{\phi^I,\phi^J\}':=(-1)^{\dg{\phi^I}+\dg{\phi^J}+\dg{\phi^I}\dg{\phi^J}}\{\phi^I,\phi^J\}.
\end{gather*}
The primed operations on $\widehat{S}(A^{\#})$ indeed form a BV algebra, for example
\begin{gather*}
\Delta'(\phi^I\phi^J)-\Delta'(\phi^I)\phi^J-(-1)^{\dg{\phi^I}}\phi^I\Delta'(\phi^J)=\{\phi^I,\phi^J\}', \\
\{\phi^I,\phi^J\phi^K\}'-\{\phi^I,\phi^J\}'\phi^K-(-1)^{(1+\dg{\phi^I})\dg{\phi^J}}\phi^J\{\phi^I,\phi^K\}'=0.
\end{gather*}
Solutions $S'$ of the corresponding master equation
\begin{gather} \label{EQMEPrimed}
d'(S')+\Delta'(S') S+\frac{1}{2}\{S',S'\}'=0
\end{gather}
are then easily seen to be in bijection $S'=\frac{1}{2}S$ with solutions $S$ of the unprimed master equation $dS+\Delta S+\frac{1}{2}\{S,S\}=0$.

Let's specialize to the $2N$-dimensional space $A$ with the standard symplectic form
$$\omega=\left(\begin{matrix} 0 & \id \\ -\id & 0 \end{matrix}\right).$$
The operations $d',\Delta',\{\}'$ differ from those considered by Zwiebach in \cite{ZwiebachClosed} by signs.
However, \eqref{EQMEPrimed} is precisely the Zwiebach's master equation
$$d'(S')+  \hbar\sum_{d=1}^N \frac{\partial_R}{\partial \phi^{d}}\frac{\partial_L}{\partial \phi^{d+N}}S' + \sum_{d=1}^N \frac{\partial_R}{\partial \phi^{d}}S'\cdot\frac{\partial_L}{\partial \phi^{d+N}}S' =0.$$
The term $d'(S')$ is customary absorbed into $S$ by declaring $S'':=S'+\frac{1}{4}\sum_{i,j}\omega(da_i,a_j)\phi^i \phi^j$.
Then the master equation reads just $\hbar \sum_{d=1}^N \frac{\partial_R}{\partial \phi^{d}}\frac{\partial_L}{\partial \phi^{d+N}}S'' + \sum_{d=1}^N \frac{\partial_R}{\partial \phi^{d}}S''\cdot\frac{\partial_L}{\partial \phi^{d+N}}S'' =0$.
Notice that $S''$ is obtained from $S'$ by allowing the term $f_2^0$ proportional to $\omega(d\ot\id)$, corresponding to an ``unstable corolla''.

\section{\texorpdfstring{The operad $\oQO$ and related algebraic structures}{QO and related algebraic structures}} \label{sec:modass}

\subsection{\texorpdfstring{The modular operad $\oQO$}{The modular operad QO}}

The modular operad $\oQO$ consists of homeomorphism classes of connected compact $2$-dimensional orientable surfaces with labeled marked points on the boundary.
To define the operadic composition, it is convenient to replace each marked point by an interval embedded in the boundary and then glue one edge of a short strip to the interval.
The edge opposite to the one glued to the interval is called open end (of an open string).
\begin{figure}[h!]
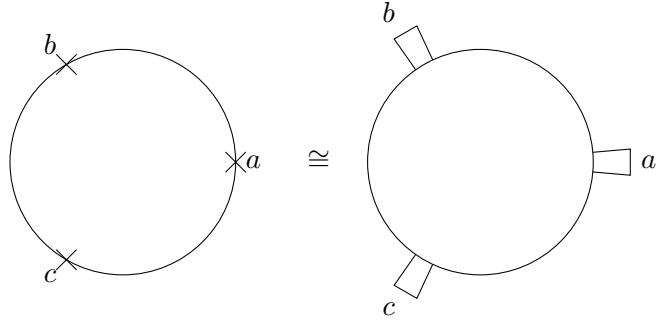
 \centering
$\PICH \quad \cong \quad \PICI$
\caption{Replacing the marked points by open ends.
Think of the interior of the circle as of hole in a sphere of large diameter.}
\end{figure}
The surface(s) can then be glued along the open ends.
Obviously, only gluing resulting in orientable surfaces are allowed.
Now we proceed formally:

\begin{definition} \label{DEFCycle}
A cycle in a set $C$ is an equivalence class $\cyc{x_1,\ldots,x_n}$ of an $n$-tuple $(x_1,\ldots,x_n)$ of several distinct elements of $C$ under the equivalence $(x_1,\ldots,x_n) \sim \tau(x_1,\ldots,x_n)$, where $\tau\in\Sigma_n$ is the cyclic permutation $\tau(i)=i+1$ for $1\leq i\leq n-1$ and $\tau(n)=1$.
In other words,
$$\cyc{x_1,\ldots,x_n} = \cdots = \cyc{x_{n-i+1},\ldots,x_n,x_1,\ldots,x_{n-i}} = \cdots = \cyc{x_2,\ldots,x_n,x_1}.$$
We call $n$ the length of the cycle.
We also admit the empty cycle $\cyc{}$, which is a cycle in any set.

For a bijection $\rho:C\xrightarrow{\sim}D$ and a cycle $\cyc{x_1,\ldots,x_n}$ in $C$, define a cycle in $D$:
$$\rho\cyc{x_1,\ldots,x_n} := \cyc{\rho(x_1),\ldots,\rho(x_n)}.$$
\end{definition}

\begin{definition}
\begin{gather*}
\oQO(C,G) := \\
\Span_{\fld}\left\{ \{\cc_1,\ldots,\cc_b\}^g \ |\ b\in\N,\ g\in\N_0,\ \cc_i\textrm{'s are cycles in }C,\ \bigsqcup_{i=1}^b\cc_i=C,\ G=2g+b-1 \right\},
\end{gather*}
where $\{\cc_1,\ldots,\cc_b\}^g$ is a symbol of degree $0$, formally being a pair consisting of $g\in\N_0$ and a set of cycles in $C$ with the above properties.
We remind the reader that $\bigsqcup$ means a \emph{disjoint} union, i.e. $i\neq j\Rightarrow \cc_i\cap\cc_j=\emptyset$!
Also recall that $\oQC(C,G)$ is defined only if the stability condition $2(G-1)+|C|>0$ is met.
Equivalently, this is
\begin{gather} \label{EQStabilityQO}
4g+2b-4+|C|>0.
\end{gather}

For a bijection $\rho:C\xrightarrow{\sim}D$, let 
$$\oQO(\rho)(\{\cc_1,\ldots,\cc_b\}^g) := \{\rho(\cc_1),\ldots,\rho(\cc_b)\}^g.$$

Next, we define $\ooo{a}{b} : \oQO(C_1\sqcup\{a\},G_1) \ot \oQO(C_2\sqcup\{b\},G_2) \to \oQO(C_1\sqcup C_2,G_1+G_2)$.
Assume $\cc_i=\cyc{a,x_1,\ldots,x_m}$ is a cycle in $C_1\sqcup\{a\}$ and let $\cd_j=\cyc{b,y_1,\ldots,y_n}$ be a cycle in $C_2\sqcup\{b\}$.
Then
\begin{gather}
\ooo{a}{b}(\{\cc_1,\ldots,\cc_{b_1}\}^{g_1} \ot \{\cd_1,\ldots,\cd_{b_2}\}^{g_2}) := \nonumber \\
\{\cyc{x_1,\ldots,x_m,y_1,\ldots,y_m},\cc_1,\ldots,\widehat{\cc_i},\ldots,\cc_{b_1},\cd_1,\ldots,\widehat{\cd_j},\ldots,\cd_{b_2}\}^{g_1+g_2}. \label{EQDefQOOO}
\end{gather}
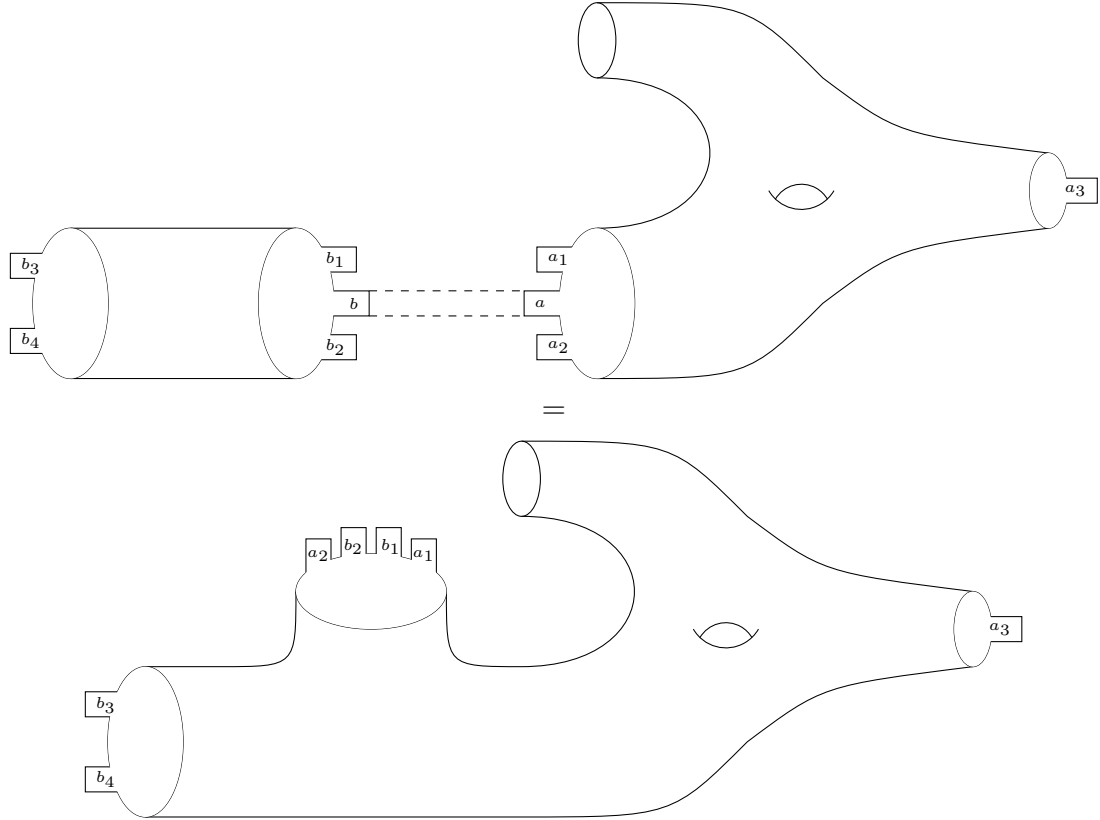
\begin{gather*}
\PICQOA \\ = \\ \PICQOB
\end{gather*}

Now we define $\xi_{ab} : \oQO(C\sqcup\{a,b\},G) \to \oQO(C,G+1)$.
Let $\{\cc_1,\ldots,\cc_{b}\}^{g} \in \oQO(C,G)$.
If there are $i<j$ such that $\cc_i=\cyc{a,x_1,\ldots,x_m}$ and $\cc_j=\cyc{b,y_1,\ldots,y_n}$, then define\footnote{In \eqref{EQXiQOOne}, the symbol $b$ in the subscript of $\xi$ and the number $b$ in the subscript of $\cc_i$ are different objects. We use this abuse of notation in the sequel.}
\begin{gather} \label{EQXiQOOne}
\xi_{ab}(\{\cc_1,\ldots,\cc_{b}\}^{g}) := \{\cyc{x_1,\ldots,x_m,y_1,\ldots,y_m},\cc_1,\ldots,\widehat{\cc_i},\ldots,\widehat{\cc_j},\ldots,\cc_{b}\}^{g+1}.
\end{gather}
$$\PICQOC=\PICQOD$$

Otherwise, there is $i$ such that $\cc_i=\cyc{a,x_1,\ldots,x_m,b,y_1,\ldots,y_m}$.
Then define
\begin{gather} \label{EQXiQOTwo}
\xi_{ab}(\{\cc_1,\ldots,\cc_{b}\}^{g}) := \{\cyc{x_1,\ldots,x_m},\cyc{y_1,\ldots,y_m},\cc_1,\ldots,\widehat{\cc_i},\ldots,\cc_{b}\}^{g}.
\end{gather}
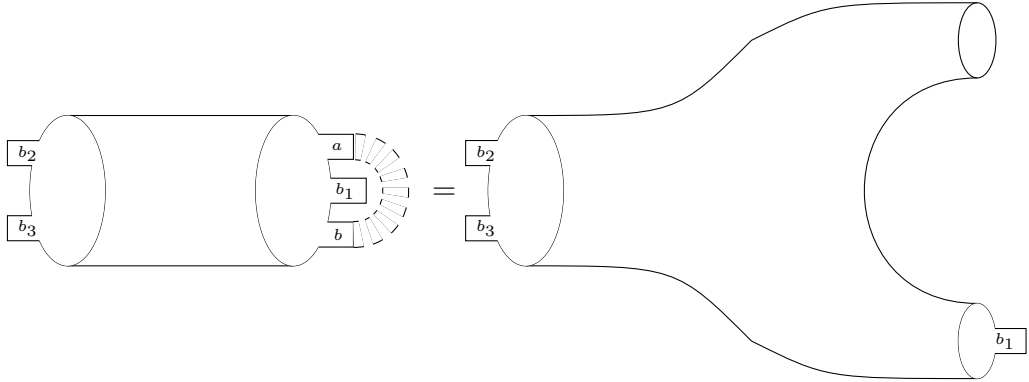
$$\PICQOE=\PICQOF$$
\end{definition}

Notice that we allow repeated empty cycles to appear in $\{\cc_1,\ldots,\cc_b\}^g$, for example $\{\cyc{},\cyc{},\cyc{3},\cyc{14},\cyc{25}\}^{2} \in \oQO([5],8)$.
Also notice that $\xi_{ab}$ can produce empty cycles: $\xi_{ab}\{\cyc{a},\cyc{b}\}^{g}=\{\cyc{}\}^{g+1}$ and $\xi_{ab}\{\cyc{ab}\}^{g}=\{\cyc{},\cyc{}\}^g$.
Observe that $$\dim_{\fld}\oQO(C,G) < \infty$$ for any $(C,G)\in\cCo$ and the whole $\oQO$ is of degree $0$.
The reader can now verify:

\begin{theorem}
$\oQO$ is a modular operad. \qed
\end{theorem}

The modular operad $\oQO$ is closely related to the modular operad $\mathbb{S}[t]$ of \cite{BarannikovModopBV}, associated to stable ribbon graphs. 

\subsection{\texorpdfstring{Cyclic $A_\infty$-algebras}{Cyclic A-infinity algebras}} \label{ssec:ainfty}

In Remark \ref{REMSModCyc}, we have already observed that restricting to the $G=0$ part of a modular operad $\oP$ and forgetting $\xi_{ab}$'s yields a cyclic operad.
Similarly, by forgetting the part of the Feynman differential dual to $\xi_{ab}$'s, the Feynman transform $\feyn{\oP}$ becomes the cyclic cobar complex $\cobar{\oP}$ (for noncyclic operads, precise analogue is defined in Chapter $3.1$ of \cite{MarklOperads}; in cyclic case, a variant appears in \cite{GetzlerCycop}).
Expressed in terms of Theorem \ref{LEMMAAlgOverFeynTrans}, algebra over $\cobar{\oP}$ is given by the collection of $\alpha(C,G)$'s satisfying the corresponding equations with terms containing $\xi_{ab}$'s omitted.

Restricting to the $G=0$ part of the modular operad $\oQO$, we obtain the cyclic operad $\oAss$: 
\begin{gather*}
\oAss(C) = \Span_{\fld}\left\{\textrm{all }|C|\textrm{-element cycles in }C\right\}
\end{gather*}
for $|C|\geq 3$ and zero otherwise.
Since $G=0$ implies $g=0$, we omit $G$ and $g$ from the notation, writing just $\cyc{a,b,\ldots}$ instead of $\{\cyc{a,b,\ldots}\}^0$.

We make explicit axioms of algebras over $\cobar{\oAss}$ to get used to our formalism.
Of course, it is well known that we will obtain cyclic $A_\infty$-algebras \cite{GetzlerCycop}.
Still, we treat this calculation in detail since standard references avoid it and since it clarifies more complicated calculations in the subsequent parts of the paper.

To lighten the notation, we identify $\oAss([n])\cong\oAss([n])^{\#}$ by sending sending the basis consisting of all cycles in $[n]$ to its dual basis.
The degree $0$ maps $\alpha(n):\oAss([n])\to\oEnd{A}(n)$ of Lemma \ref{LEMMAHardToUse} are determined by their values on representative of each $\Sigma_n$ orbit; but there is only one orbit and let it be represented by 
$$\cc_n=\cyc{1,2,\ldots,n}.$$
We see that an algebra over $\cobar{\oAss}$ is uniquely determined by a collection
$$\set{f_n:=\alpha(\cc_n)\in\oEnd{A}(n)}{n\geq 3}$$
of degree $0$ linear maps  such that
$$f_n = f_n\circ\tau$$
for all cyclic permutations $\tau\in\Sigma_n$, $n\geq 3$, and
\begin{gather}
	d(f_n) = \label{EQCycliAssFeyn} \\
	\mathclap{ \frac{1}{2} \sum_{\mathclap{\substack{{}\\ C_1\sqcup C_2=[n] \\ G_1+G_2=G}}}  \oEnd{A}(\rho_1^{\kappa_1(a)}\kappa_1|_{C_1}\sqcup\rho_2^{\kappa_2(b)}\kappa_2|_{C_2})^{-1} (\ooot{\kappa_1(a)}{\kappa_2(b)})_{\oEnd{A}} (\alpha\!\ot\!\alpha) (\oAss(\kappa_1)\!\ot\!\oAss(\kappa_2)) (\oooo{a}{b}{C_1\sqcup\{a\},G_1}{C_2\sqcup\{b\},G_2})^{\#}_{\oAss}(\cc_n). } \nonumber
\end{gather}
To choose convenient $\kappa_1$ and $\kappa_2$, we need to understand the dual of composition in $\oAss$.
We immediately see that $(\oooo{a}{b}{C_1\sqcup\{a\}}{C_2\sqcup\{b\}})^{\#}_{\oAss} (\cyc{1,\ldots,n})$ vanishes for most of the decompositions $C_1\sqcup C_2=[n]$.
Think of the cycle $\cyc{1,\ldots,n}$ as $n$ distinct points on a circle.
\begin{figure}[!h] \centering
$$\PICKorbiA \qquad \PICKorbiB$$
\caption{The cycle $\cyc{1,\ldots,n}$ and $\cyc{1,\ldots,n}$ cut into two pieces $C_1$,$C_2$} 
\label{fig:cycle}
\end{figure}
Then the only non-vanishing terms are those where we separate $[n]$ into two pieces $C_1,C_2$ by cutting the circle exactly twice (see Figure \ref{fig:cycle}).
Thus we get
\begin{gather}
\sum_{C_1\sqcup C_2=[n]} (\oooo{a}{b}{C_1\sqcup\{a\}}{C_2\sqcup\{b\}})^{\#}_{\oAss} (\cyc{1,\ldots,n}) \nonumber \\
 = \sum_{s=0}^{n-1} \sum_{l=2}^{n-2} (\hspace{-0.6ex}( a,\underbrace{s+1,\ldots,s+l}_l )\hspace{-0.6ex}) \ot (\hspace{-0.6ex}( b,\underbrace{s+l+1,\ldots,s+n}_{n-l} )\hspace{-0.6ex}), \label{EQDualOfComInAss}
\end{gather}
where the numbers $s+k$ are counted modulo $n$ so that they take values in $[n]$.
The reader may have expected $l$ to run through $0,1,\ldots,n$, but $l=0,1,n-1,n$ violates the stability condition $|C|\geq 3$ from the definition of $\oAss$.
Notice that consequently $(\ooo{a}{b})^{\#}_{\oAss} (\cyc{1,\ldots,n}) =0$ for $n<4$.

So it suffices to restrict to $C_1=\{s+1,\ldots,s+l\}$ and $C_2=\{s+l+1,\ldots,s+n\}$ (counted modulo $n$) as in \eqref{EQDualOfComInAss}.
We choose $\kappa_1:C_1\sqcup\{a\}\to[|C_1|+1]$ so that 
$$\oAss(\kappa_1)(\cyc{a,s+1,\ldots,s+l})=\cyc{1,\ldots,l+1},$$
for example let $\kappa_1$ be given by $\kappa_1(s+k)=k$ ($s+k$ counted modulo $n$) for $1\leq k\leq l$ and $\kappa_1(a)=|C_1|+1$.
Similarly, choose $\kappa_2:C_2\sqcup\{b\}\to[|C_2|+1]$ so that
$$\oAss(\kappa_2)(\cyc{b,s+l+1,\ldots,s+n})=\cyc{1,\ldots,n-l+1},$$
for example $\kappa_2(s+k)=k$ for $1\leq k\leq n-l$ ($s+k$ counted modulo $n$) and $\kappa_2(b)=|C_2|+1$.
This implies
$$(\alpha\ot\alpha) (\oAss(\kappa_1)\ot\oAss(\kappa_2)) (\oooo{a}{b}{C_1\sqcup\{a\},G_1}{C_2\sqcup\{b\},G_2})^{\#}_{\oAss}(\cc_n) = f_{|C_1|+1}\ot f_{|C_2|+1}.$$
Let $\psi$ denote the permutation $\rho_1^{\kappa_1(a)}\kappa_1|_{C_1}\sqcup\rho_2^{\kappa_2(b)}\kappa_2|_{C_2}$ in \eqref{EQCycliAssFeyn}.
We immediately see that it is the cyclic permutation
\begin{gather} \label{EQAooCycDefPsi}
\psi(s+k)=k \textrm{ for } 1\leq k\leq n \quad  (s+k\textrm{ counted modulo }n).
\end{gather}
Thus:
\begin{theorem} \label{THMAoo}
An algebra over $\cobar{\oAss}$ on a dg symplectic vector space $A$ is equivalently given by a collection
$$\{f_n:A^{\ot n}\to\fld\ |\ n\geq 3\}$$
of degree $0$ cyclically symmetric linear maps, i.e.
$$f_n = f_n \comp \tau,$$
where $\tau$ is the cyclic permutation as in Definition \ref{DEFCycle}.
These maps are required to satisfy
\begin{gather}
d (f_3) = 0  \label{EQAooInDisguise}, \\
d(f_n)(x_1,\ldots,x_n) = \frac{1}{2} \sum_{s=0}^{n-1} \sum_{l=2}^{n-2} \sum_{d=1}^{\dim A} \pm f_{l+1}(x_{s+1},\ldots,x_{s+l},a_d) \cdot f_{n-l+1}(x_{s+l+1},\ldots,x_{s+n},b_d) \nonumber
\end{gather}
for every $x_1,\ldots,x_n\in A$ and $n\geq 4$, where $\pm$ is the Koszul sign of evaluating $\left( \left( f_{l+1}(\id^{\ot l}\ot a_d) \cdot f_{n-l+1}(\id^{\ot n-l+1}\ot b_d) \right) \circ \psi \right) (x_1\ot\ldots\ot x_n)$ (the $\psi$ is defined in \eqref{EQAooCycDefPsi}).
\end{theorem}

Notice that the cyclic symmetry of $f_{l+1}$ allows us to move $a_d$ to an arbitrary input.
Similarly for $b_d$.
This is reflected by the ambiguity of the choice of $\kappa_1,\kappa_2$.

Next, we sketch two equivalent description of the cyclic $A_\infty$ algebras.

First, denote $TA:=\bigoplus_{n\geq 0}A^{\ot n}$ the tensor algebra on $A$.
There are isos
\begin{gather*}
\begin{matrix}
\mathop{Hom}(A^{\ot n},\fld) & \cong & \mathop{Hom}(A^{\ot n},A) & \cong & \mathop{Hom}((\uparrow\!\!A)^{\ot n},\uparrow\!\!A) \\
\omega(1\ot m_n) & \leftrightarrow & m_n & \leftrightarrow & m_n':=\uparrow m_n \downarrow^{\ot n}
\end{matrix}
\end{gather*}
Let $m_1':=\uparrow d\downarrow$ be the differential on $\uparrow\!\!A$.
In terms of these maps $m_n'$ of degree $2-n$, \eqref{EQAooInDisguise} becomes
$$\sum_{i_1+i_2+i_3=n} (-1)^{i_1i_2+i_3} m'_{i_1+i_3+1} \left(\id^{\ot i_1}\ot m'_{i_2}\ot\id^{\ot i_3}\right)=0,$$
which is the usual cyclic $A_\infty$-algebra relation.

Next, we apply the Barannikov's theory of Section \ref{SSECBarannikov} to get a description of cyclic $A_\infty$-algebras in terms of solutions of the master equation.
We proceed as in Section \ref{SECMEQC}:
Recall that the element $\cc_n=\cyc{1,2,\ldots,n}\in\oAss(n)$ is a representative of the single $\Sigma_n$ orbit.
Its stabilizer obviously consists of the cyclic permutations, hence
\begin{gather*}
\oAss(n) \ot_{\Sigma_n} (A^{\#})^{\ot n} \cong (A^{\#})^{\ot n} / \textrm{cyclic permutations} =: \widehat{C}^n(A^{\#}).
\end{gather*}
and $\tilde{P}=\widehat{C}(A^{\#}):=\prod_{n\geq 3} \widehat{C}^n(A^{\#})$.
The BV differential is
\begin{gather} \label{EQQCBVDiff}
d(\cc_n\ot_{\Sigma_n}\phi^I) = \sum_{i=1}^n (-1)^{1+\dg{\phi^I}} \cc_n\ot_{\Sigma_n}d_{{A^{\#}}^{\ot n}}\phi^I.
\end{gather}
To get a formula for the BV bracket, one first easily verifies
\begin{gather*}
\ooot{i}{j}(\cc_{n_1+1}\ot\cc_{n_2+1}) = \\
= \cyc{1,\ldots,i-1,j+n_1,\ldots,n_2+n_1,1+n_1,\ldots,j-1+n_1,i,\ldots,n_1} = \pi_{ij}\cc_{n_1+n_2},
\end{gather*}
where $\pi_{ij}\in\Sigma_{n_1+n_2}$ is a permutation of blocks of lengths $i-1, n_1+1-i, j-1, n_2+1-j$ swapping the second and fourth block.
Symbolically,
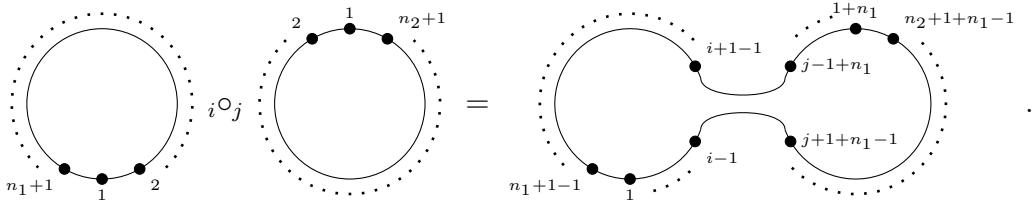
\begin{gather*}
\PICABGax \ooo{i}{j} \PICABGbx = \PICABGx.
\end{gather*}
Then
\begin{gather}
\{ \cc_{n_1+1} \ot_{\Sigma_{n_1+1}} \phi^I , \cc_{n_2+1} \ot_{\Sigma_{n_2+1}} \phi^J \} = \nonumber \\
= \sum_{i,j,d,e} \pi_{ij}\cc_{n_1+n_2} \ot_{\Sigma_{n_1+n_2}} (-1)^{\dg{a_d}(\dg{\phi^I}+1)+\dg{\phi^J}+\dg{\phi^I}\dg{\phi^J}+1} \ \omega^{de} \frac{\partial^{(i)}}{\partial \phi^d} \phi^I \ot \frac{\partial^{(j)}}{\partial \phi^e} \phi^J =  \label{EQBracketAss} \\
= \sum_{i,j} (-1)^{\dg{a_d}(\dg{\phi^I}+1)+\dg{\phi^J}+\dg{\phi^I}\dg{\phi^J}+1} \ \omega^{I_iJ_j} \cc_{n_1+n_2} \ot_{\Sigma_{n_1+n_2}} \pi_{ij}^{-1} \left( \frac{\partial^{(i)}}{\partial \phi^{I_i}} \phi^I \ot \frac{\partial^{(j)}}{\partial \phi^{J_j}} \phi^J \right) \nonumber
\end{gather}
If one thinks of each $\phi^{I_i}$ in the tensors $\phi^I$ as sitting on a circle at position $i$, the summands above have a direct relation to the combinatorics of $\oAss$:
For brevity of notation, let's write simply $\phi^I$ rather than $\cc_n\ot_{\Sigma_n}\phi^I$.
Then
\begin{gather}
\{\phi^{I_1\cdots I_{n_1+1}},\phi^{J_1\cdots J_{n_2+1}}\} = \sum_{i,j} \pm \ \omega^{I_iJ_j} \phi^{I_1\cdots I_{i-1}J_{j+1}\cdots J_{n_2+1}J_1\cdots J_{j-1}I_{i+1}\cdots I_{n_1+1}} \label{EQQCBVBrac} \\
\Big\{ \hspace{-1em} \PICABGa , \PICABGb \hspace{-1em} \Big\} = \sum_{i,j} \pm \ \omega^{I_iJ_j} \PICABG \nonumber
\end{gather}
The sign $\pm$ consist of the sign in \eqref{EQBracketAss}, the sign of positional derivatives and the Koszul sign of $\pi_{ij}$.
For example, if $\dg{\phi^I}=\dg{\phi^J}=0$ and $I=(I_1,\ldots,I_5), J=(J_1,\ldots,J_4)$, then 
$$\pm = (-1)^{1+\dg{\phi^{I_3}}(\dg{\phi^{I_1}}+\dg{\phi^{I_2}})+\dg{\phi^{J_2}}\dg{\phi^{J_1}}+(\dg{I_4}+\dg{I_5})(\dg{J_1}+\dg{J_3}+\dg{J_4})+\dg{J_1}(\dg{J_3}+\dg{J_4})}.$$

To conclude, solutions $S$ of the master equation
$$d(S)+\frac{1}{2}\{S,S\}=0$$
in $\widehat{C}(A^{\#})$ with operations \eqref{EQQCBVDiff} and \eqref{EQQCBVBrac} are in bijection with cyclic $A_\infty$-algebras on $A$.

\subsection{\texorpdfstring{Quantum $A_\infty$-algebras}{Quantum A-infinity algebras}} \label{SECTQuanAooAlg}

\begin{definition}
A quantum $A_\infty$ algebra is an algebra over $\feyn{\oQO}$.
\end{definition}


\subsubsection{\texorpdfstring{Axioms for algebras over $\feyn{\oQO}$}{Axioms for algebras over F(QO)}} \label{SECAxiomsForQuantumAoo}

In this section, we make the axioms for these algebras explicit.

We identify $\oQO([n],G)\cong\oQO([n],G)^{\#}$ in the standard way.
To use Lemma \ref{LEMMAHardToUse} efficiently, we first discuss the orbits of $\oQO$:
Recall that a basis of $\oQO([n],G)$ consists of elements
$$\{\cc_1,\ldots,\cc_b\}^g,$$
where $\cc_i$'s are cycles in $[n]$ such that $\bigsqcup_{i=1}^b\cc_i=[n]$ and $G=2g+b-1$.
Denote 
\begin{gather} \label{EQSequenceofBs}
(b_k) := (b_0,b_1,\ldots),
\end{gather}
where $b_k$ is the number of cycles of length $k$.
$(b_k)$ is called the b-sequence of $\cc_1,\ldots,\cc_b$.
Hence 
\begin{gather} \label{EQOrbits}
\sum_{k=0}^\infty b_k=b,\quad \sum_{k=0}^\infty kb_k=n.
\end{gather}
Of course, $(b_k)$ is eventually zero, thus the last two sums contain only finitely many nonzero terms.
Obviously, two elements of the above form belong to the same $\Sigma_n$-orbit in $\oQO([n],G)$ iff their b-sequences coincide.
Conversely, any sequence satisfying \eqref{EQOrbits} and $G=2g+b-1$ determines an orbit\footnote{In fact, if all $\cc_i$'s are nonempty, $\{\cc_1,\ldots,\cc_b\}$ can be seen as a decomposition into independent cycles of a permutation on $[n]$. Then $\Sigma_n$ acts by conjugation and the sequence of lengths of cycles is a familiar invariant of the orbits.} in $\oQO([n],G)$.
Hence given such a $(b_k)$, we choose a representative of the corresponding orbit as follows:
\begin{align}
\overline{(b_k)}^g &:= \{\underbrace{\cc_{1},\ldots,\cc_{b_0}}_{\textrm{length }0\textrm{ cycles}},\underbrace{\cc_{b_0+1},\ldots,\cc_{b_0+b_1}}_{\textrm{length }1\textrm{ cycles}},\ldots,\cc_{b}\}^g = \label{EQOrbitRepresentative} \\
&= \{\underbrace{\emptyset,\emptyset,\ldots,\emptyset}_{b_0},\underbrace{\cyc{1},\cyc{2},\ldots,\cyc{b_1}}_{b_1},\cyc{b_1+1,b_1+2},\cyc{b_1+3,b_1+4},\ldots\}^{g}, \nonumber
\end{align}
where each $\cc_i=\cycb{c_i^1,\ldots,c_i^{|\cc_i|}}$ and these satisfy $\bigsqcup_{i=1}^{b}\cc_{i}=[n]$ and $c_i^{k}<c_j^{l}$ whenever $i<j$ or $i=j$ and $k<l$.
By abuse of notation, we will often write $\cc_i$ meaning a \emph{tuple} $(c_i^1,\ldots,c_i^{|\cc_i|})$ rather than a cycle.

Let $\alpha([n],G):\oQO([n],G)\to\oEnd{A}(n,G)$ be the morphisms determining an algebra over $\feyn(\oQO)$ on $A$.
Lemma \ref{LEMMAHardToUse} implies that this algebra is determined by the maps
$$f^{(b_k),g} := \alpha([n],G)(\overline{(b_k)}^g) \in \op{Hom}(A^{\ot n},\fld)$$
ranging over all $b$-sequences and integers $g\geq 0$.
Recall that $G=2g+b-1$, thus given $b$, $g$ determines $G$ and vice versa.
Lemma \ref{LEMMAHardToUse} lists the axioms which these maps are required to satisfy.
We make these axioms explicit in the sequel.

We first dualize the operad structure maps on $\oQO$.
The empty cycles appearing in the elements of $\oQO$ make formulas look complicated.
The reason is that while you can always distinguish cycles $\cc_i,\cc_j$ of which at least one has nonzero length, you can't distinguish them if both are empty.
This forces us to treat the empty cycles $\cc_1=\cdots=\cc_{b_0}=\emptyset$ separately.
Also, the stability condition \eqref{EQStabilityQO} affects some terms below.

Let's dualize $(\xi_{ab})_{\oQO}$ (Equations \eqref{EQXiQOTwo} and \eqref{EQXiQOOne}) and evaluate on $\overline{(b_k)}^g = \{\cc_1,\ldots,\cc_b\}^g$:
\begin{gather}
(\xi_{ab})^{\#}_{\oQO}(\overline{(b_k)}^g) = (\xi_{ab})^{\#}_{\oQO}(\{\cc_1,\ldots,\cc_b\}^g) = \label{EQDualoFXi} \\
= \sum_{\mathclap{\substack{i,j\\ i\neq j\\ |\cc_i|\neq 0\neq|\cc_j|}}} \sum_{p=0}^{|\cc_i|-1} \sum_{q=0}^{|\cc_j|-1} \{\cycb{ac_i^{p+1}\cdots c_i^{p+|\cc_i|}bc_j^{q+1}\cdots c_j^{q+|\cc_j|}},\cc_1,\ldots\widehat{\cc_i}\ldots\widehat{\cc_j}\ldots,\cc_b\}^{g} +{} \nonumber \\[-20pt]
+ \mathop{\delta}_{b_0>0} \sum_{\substack{j\\ |\cc_j|\neq 0}} \sum_{q=0}^{|\cc_j|-1} \{\cycb{abc_j^{q+1}\cdots c_j^{q+|\cc_j|}},\widehat{\cc_1},\cc_2,\ldots\widehat{\cc_j}\ldots,\cc_b\}^{g} +{} \nonumber \\[-5pt]
+ \mathop{\delta}_{b_0>0} \sum_{\substack{i\\ |\cc_i|\neq 0}} \sum_{p=0}^{|\cc_i|-1} \{\cycb{ac_i^{p+1}\cdots c_i^{p+|\cc_i|}b},\widehat{\cc_1},\cc_2,\ldots\widehat{\cc_i}\ldots,\cc_b\}^{g} +{} \nonumber \\[-5pt]
+ \mathop{\delta}_{\substack{b_0>1\textrm{ and}\\ (b>2\textrm{ or }g>0)}} \{\cyc{ab},\widehat{\cc_1},\widehat{\cc_2},\cc_3,\ldots,\cc_b\}^{g} + \cdots \nonumber
\end{gather}
\begin{gather}
\cdots + \sum_{\substack{m\\ |\cc_m|\neq 0}} \sum_{s=0}^{|\cc_m|-1} \sum_{l=0}^{|\cc_m|} \{\cycm{ac_m^{s+1}\cdots c_m^{s+l}},\cycm{bc_m^{s+l+1}\cdots c_m^{s+|\cc_m|}},\cc_1,\ldots\widehat{\cc_m}\ldots,\cc_b\}^{g-1} +{} \nonumber \\
+ \mathop{\delta}_{b_0>0} \{\cyc{a},\cyc{b},\widehat{\cc_1},\cc_2,\ldots,\cc_b\}^{g-1}, \nonumber
\end{gather}
where the upper indices of $c_i$ are counted modulo $|\cc_i|$ so that their values belong to $[|\cc_i|]$ and analogously for the upper indices of $c_j$ resp. $c_m$;
and $\mathop{\delta}_{X}$ equals $1$ if the lower index condition $X$ is met, otherwise it equals $0$.

Now we dualize $(\ooo{a}{b})_{\oQO}$ (Equation \eqref{EQDefQOOO}).
We are in fact interested in the sum over $C_1,C_2,G_1,G_2$ of the duals of the composition.
A moment's reflection convinces us that
\begin{gather}
\sum_{\substack{G_1,G_2\\ G_1+G_2=G}} \sum_{\substack{C_1,C_2\\C_1\sqcup C_2=[n]}} (\oooo{a}{b}{C_1\sqcup\{a\},G_1}{C_2\sqcup\{b\},G_2})^{\#}_{\oQO} (\{\cc_1,\ldots,\cc_b\}^g) = \label{EQDualOfOO} \\
= \sum_{m=b_0+1}^b \sum_{\substack{I,J\subset [b]-[b_0]\\ I\sqcup J=\\ [b]-[b_0]-\{m\}}} \sum_{e=0}^{b_0} \sum_{\substack{g_1,g_2\\ g_1+g_2=g}} \sum_{s=0}^{|\cc_m|-1} \sum_{l=0}^{|\cc_m|} \nonumber \\
\delta^1 \!\!\cdot\! \{\!\cycm{ac_m^{s+1} \hspace{-.8em}\cdots c_m^{s+l}}\!\!,\cc_1,\ldots,\cc_e,\cc_{i_1},\ldots,\cc_{i_{|I|}}\}^{g_1} \!\ot\! \{\!\cycm{bc_m^{s+l+1}\hspace{-1.5em}\cdots c_m^{s+|\cc_m|}}\!\!,\cc_{e+1},\ldots,\cc_{b_0},\cc_{j_1},\ldots,\cc_{j_{|J|}}\}^{g_2} + \nonumber \\
+ \sum_{\substack{I,J\subset[b]-[b_0]\\ I\sqcup J=\\ [b]-[b_0]}} \! \sum_{e=0}^{b_0-1} \!\! \sum_{\substack{g_1,g_2\\ g_1+g_2=g}} \hspace{-.8em} \delta^2 \!\!\cdot\! \{\!\cyc{a}\!\!,\cc_1,\ldots,\cc_{e},\cc_{i_1},\ldots,\cc_{i_{|I|}}\}^{g_1} \!\ot\! \{\!\cyc{b}\!\!,\cc_{e+1},\ldots,\cc_{b_0-1},\cc_{j_1},\ldots,\cc_{j_{|J|}}\}^{g_2}, \nonumber
\end{gather}
where $I=\{i_1,\ldots,i_{|I|}\}$ and $J=\{j_1,\ldots,j_{|J|}\}$, the upper indices are counted mod $|\cc_m|$ as explained above.
$G_i$'s and $g_i$'s in the first term are related by $G_1=2g_1+(1+e+|I|)-1$ and $G_2=2g_2+(1+b_0-e+|J|)-1$.
Notice that $C_1 = \{c_m^{s+1},\cdots,c_m^{s+l}\} \sqcup \bigcup_{k=1}^{e}\cc_k \sqcup \bigcup_{k=1}^{|I|}\cc_{i_k} = \{c_m^{s+1},\cdots,c_m^{s+l}\} \sqcup \bigcup_{k=1}^{|I|}\cc_{i_k}$ and analogously for $C_2$.
Then
\begin{gather} \label{EQdelta1}
\delta^1 := \mathop{\delta}_{\substack{g_1>0\textrm{ or}\\ e+|I|>0\textrm{ or}\\ l\geq 2}} \cdot \mathop{\delta}_{\substack{g_2>0\textrm{ or}\\ b_0-e+|J|>0\textrm{ or}\\ |\cc_m|-l\geq 2}}
\end{gather}
is forced by the stability condition \eqref{EQStabilityQO}.
Similarly,
\begin{gather}\label{EQdelta2}
\delta^2 := \mathop{\delta}_{\substack{g_1>0\textrm{ or}\\ e+|I|>0}} \cdot \mathop{\delta}_{\substack{g_2>0\textrm{ or}\\ b_0-1-e+|J|>0}}.
\end{gather}

Now we evaluate the equation of Lemma \ref{LEMMAHardToUse} on $\overline{(b_k)}^g$.
There are three contributions: two coming from the first four resp. last two summands in \eqref{EQDualoFXi} and one coming from the right-hand side of \eqref{EQDualOfOO}.


\textbf{First contribution.}
First four summands of \eqref{EQDualoFXi} contribute to
$$\oEnd{A}((\rho^{\kappa(a)\kappa(b)}\kappa|_{[n]})^{-1})(\overline{\xi}_{\kappa(a)\kappa(b)})_{\oEnd{A}}\alpha \oQO(\kappa^{-1})^{\#} (\xi_{ab})^{\#}_{\oQO} (\overline{(b_k)}^g)$$
of Lemma \ref{LEMMAHardToUse} by
\begin{gather*}
\sum_{i,j} \sum_{p=0}^{|\cc_i|-1} \sum_{q=0}^{|\cc_j|-1} \\
\mathclap{ \oEnd{A}(\rho^{\kappa(a)\kappa(b)}\kappa|_{[n]})^{-1}(\overline{\xi}_{\kappa(a)\kappa(b)})_{\oEnd{A}}\alpha \oQO(\kappa) \{\cycb{ac_i^{p+1}\cdots c_i^{p+|\cc_i|}bc_j^{q+1}\cdots c_j^{q+|\cc_j|}},\cc_1,\ldots\widehat{\cc_i}\ldots\widehat{\cc_j}\ldots,\cc_b\}^{g}. }
\end{gather*}
By abuse of notation, we have written all four terms as one.
Choose arbitrary $\kappa\in\Sigma_{n+2}$ so that $\oQO(\kappa) (\{\cycb{ac_i^{p+1}\cdots c_i^{p+|\cc_i|}bc_j^{q+1}\cdots c_j^{q+|\cc_j|}},\cc_1,\ldots\widehat{\cc_i}\ldots\widehat{\cc_j}\ldots,\cc_b\}^{g}) = \overline{(b'_k)}^g$ is the orbit representative in $\oQO([n+2],G-1)$.
Let $\psi:=\rho^{\kappa(a)\kappa(b)}\kappa|_{[n]}$.
Then
$$\sum_{i,j} \sum_{p=0}^{|\cc_i|-1} \sum_{q=0}^{|\cc_j|-1} \sum_{d=1}^{\dim A} f^{(b'_k),g} (\id^{\ot\cdots}\ot\underbrace{a_d}_{\mathclap{\substack{\kappa(a)\textrm{-th}\\ \textrm{from the left}}}}\ot\id^{\ot\cdots}\ot\underbrace{b_d}_{\mathclap{\kappa(b)\mathrm{-th}}}\ot\id^{\ot\cdots}) \circ \psi.$$
%
Of course, the summands above \emph{directly} reflect the combinatorics of $\oQO$, although this is obscured by the complicated definition of $\kappa$ and $\psi$.
An example makes that clearer:

\begin{example}
Denote by
$$\perm{x_1 & x_2 & \ldots & x_n}{f(x_1) & f(x_2) & \ldots & f(x_n)}$$
a map $f:\{x_1,\ldots,x_n\}\to Y$ mapping each $x_i$ to $f(x_i)$.

Let $\overline{(b_k)}^g=\overline{(0,1,2,0,\ldots)}^g=\{\cyc{1},\cyc{23},\cyc{45}\}^g$ and let's investigate the term with $i=1, j=2$.
The relevant terms of $(\xi_{ab})^{\#}_{\oQO}(\{\cyc{1},\cyc{23},\cyc{45}\}^g)$ are
$$\{\cyc{45},\cyc{a1b23}\}^g \quad\textrm{and}\quad \{\cyc{45},\cyc{a1b32}\}^g.$$
Let's focus, for example, on the first term.
Choose $\kappa:=\perm{1&2&3&4&5&a&b}{4&6&7&1&2&3&5}$ (hence $\kappa^{-1}=\perm{1&2&3&4&5&6&7}{4&5&a&1&b&2&3}$).
Then $\oQO(\kappa)(\{\cyc{45},\cyc{a1b23}\}^g)=\{\cyc{12},\cyc{34567}\}^g=\overline{(0,0,1,0,0,1,0,\ldots)}^g$ and $\psi=\perm{1&2&3&4&5}{3&4&5&1&2}$ (hence $\psi^{-1}=\perm{1&2&3&4&5}{4&5&1&2&3}$).
Thus we get
$$\sum_d \pm f^{(0,0,1,0,0,1,0,\ldots),g}(x_4,x_5,a_d,x_1,b_d,x_2,x_3),$$
where $\pm$ are the Koszul signs of evaluating $f^{(0,0,1,0,0,1,0,\ldots),g}(\id\ot\id\ot a_d\ot\id\ot b_d\ot\id\ot\id)\circ\psi$ at $x_1\ot\cdots\ot x_5 \in A^{\ot 5}$.
Notice that you can directly read off the order of arguments from the values of $\kappa^{-1}$.
Different choices of $\kappa$ lead to orders: $(x_4,x_5,x_3,a_d,x_1,b_d,x_2)$, $(x_4,x_5,x_2,x_3,a_d,x_1,b_d)$, $(x_4,x_5,b_d,x_2,x_3,a_d,x_1)$ or $(x_4,x_5,x_1,b_d,x_2,x_3,a_d)$.
Of course, this choice is irrelevant due to the symmetry of $f^{(0,0,1,0,0,1,0,\ldots),g}$.
\end{example}
Now we rewrite the first four terms of \eqref{EQDualoFXi} more carefully:
\begin{gather*}
\sum_{\mathclap{\substack{i,j\\ i\neq j\\ |\cc_i|\neq 0\neq|\cc_j|}}} \sum_{p=0}^{|\cc_i|-1} \sum_{q=0}^{|\cc_j|-1} \sum_{d=1}^{\dim A}  f^{(b'_k),g} (\cdots\ot a_d\ot\cdots\ot b_d\ot\cdots) \comp \psi \ + \\
+ \mathop{\delta}_{b_0>0} \sum_{\substack{j\\ |\cc_j|\neq 0}} \sum_{q=0}^{|\cc_j|-1} \sum_{d=1}^{\dim A} f^{(b'_k),g} (\cdots\ot a_d\ot\cdots\ot b_d\ot\cdots) \comp \psi \ + \\
+ \mathop{\delta}_{b_0>0} \sum_{\substack{i\\ |\cc_i|\neq 0}} \sum_{p=0}^{|\cc_i|-1} \sum_{d=1}^{\dim A} f^{(b'_k),g} (\cdots\ot a_d\ot\cdots\ot b_d\ot\cdots) \comp \psi \ + \\
+ \mathop{\delta}_{\substack{b_0>1\textrm{ and}\\ (b>2\textrm{ or }g>0)}} \sum_{d=1}^{\dim A} f^{(b'_k),g} (\cdots\ot a_d\ot\cdots\ot b_d\ot\cdots) \comp \psi.
\end{gather*}
Here $b'_k$ in the first sum is $(b_0,b_1,\ldots,b_{|\cc_i|}-1,\ldots,b_{|\cc_j|}-1,\ldots,b_{|\cc_i|+|\cc_j|+2}+1,\ldots)$ if $|\cc_i|\neq|\cc_j|$, and $(b_0,b_1,\ldots,b_{|\cc_i|}-2,\ldots,b_{2|\cc_i|+2}+1,\ldots)$ if $|\cc_i|=|\cc_j|$.
In the second (resp. third) sum, $(b'_k)=(b_0-1,\ldots,b_{|\cc_i|}-1,\ldots,b_{|\cc_i|+2}+1,\ldots)$ (resp. $(b_0-1,\ldots,b_{|\cc_j|}-1,\ldots,b_{|\cc_j|+2}+1,\ldots)$).
In the fourth sum, $(b'_k)=(b_0-2,b_1,b_2+1,\ldots)$.
The positions of $a_d$ and $b_d$ and $\psi$ (depending on $i,j,p,q$) are described above and the choices made in determining them don't affect the expression $f^{(b'_k),g} (\cdots\ot a_d\ot\cdots\ot b_d\ot\cdots) \comp \psi$ because of the symmetries of $f^{(b'_k),g}$.

Now fix $i_0$ and $j_0$ and observe that the terms with $i=i_0<j=j_0$ and $i=j_0>j=i_0$ coincide.
Hence we get
\begin{gather}
2  \sum_{\mathclap{\substack{i,j\\ i<j\\ |\cc_i|\neq 0\neq|\cc_j|}}} \sum_{p=0}^{|\cc_i|-1} \sum_{q=0}^{|\cc_j|-1} \sum_{d=1}^{\dim A} f^{(b'_k),g} (\cdots\ot a_d\ot\cdots\ot b_d\ot\cdots) \comp \psi \ + \nonumber \\
+ \ 2 \mathop{\delta}_{b_0>0} \sum_{\substack{i\\ |\cc_i|\neq 0}} \sum_{p=0}^{|\cc_i|-1} \sum_{d=1}^{\dim A} f^{(b'_k),g} (\cdots\ot a_d\ot\cdots\ot b_d\ot\cdots) \comp \psi \ + \nonumber \\
+ \mathop{\delta}_{\substack{b_0>1\textrm{ and}\\ (b>2\textrm{ or }g>0)}} \sum_{d=1}^{\dim A} f^{(b'_k),g} (\cdots\ot a_d\ot\cdots\ot b_d\ot\cdots) \comp \psi. \label{EQFirstContr}
\end{gather}
Notice that the coefficient $2$ doesn't appear at the last summand.

\textbf{Second contribution.}
The last two terms of \eqref{EQDualoFXi} contribute to
$$\oEnd{A}((\rho^{\kappa(a)\kappa(b)}\kappa|_{[n]})^{-1})(\overline{\xi}_{\kappa(a)\kappa(b)})_{\oEnd{A}}\alpha \oQO(\kappa^{-1})^{\#} (\xi_{ab})^{\#}_{\oQO} (\overline{(b_k)}^g)$$
of Lemma \ref{LEMMAHardToUse} by
\begin{gather*}
\sum_{\substack{m\\ |\cc_m|\neq 0}} \sum_{s=0}^{|\cc_m|-1} \sum_{l=0}^{|\cc_m|} \sum_{d=1}^{\dim A} f^{(b'_k),g-1} (\cdots\ot a_d\ot\cdots\ot b_d\ot\cdots) \comp \psi \ + \\
+ \mathop{\delta}_{b_0>0} \sum_{d=1}^{\dim A} f^{(b'_k),g-1} (\cdots\ot a_d\ot\cdots\ot b_d\ot\cdots), 
\end{gather*}
where:
in the first summand, $a_d,b_d$ sit at positions $\kappa(a),\kappa(b)$, where $\kappa:[n]\sqcup\{a,b\}\to[n+2]$ is arbitrary such that $\oQO(\kappa)(\{\cycm{ac_m^{s+1}\cdots c_m^{s+l}},\cycm{bc_m^{s+l+1}\cdots c_m^{s+|\cc_m|}},\cc_1,\ldots\widehat{\cc_m}\ldots,\cc_b\}^{g-1}) = \overline{(b'_k)}^{g-1}$; and $\psi:=\rho^{\kappa(a)\kappa(b)}\kappa|_{[n]}$.
An example illustrates the meaning of $\kappa$ and $\psi$:

\begin{example}
Let $\overline{(b_k)}^g=\{\cyc{1234}\}^g$.
$(\xi_{ab})^{\#}_{\oQO}$ applied to this contains a summand
$$\{\cyc{a23},\cyc{b41}\}^{g-1}.$$
Its corresponding term is
$$\sum_d \pm f^{(0,0,0,2,0,\ldots),g-1}(a_d,x_2,x_3,b_d,x_4,x_1).$$
There are $|\Stab{\{\cyc{a23},\cyc{b41}\}^{g-1}}|=18$ other possible orders of the arguments, for example: $(x_3,a_d,x_2,x_4,x_1,b_d)$and $(b_d,x_4,x_1,a_d,x_2,x_3)$.
\end{example}

Now fix $s_0$ and $l_0$ and observe that the terms with $s=s_0,\ l=l_0$ and $s=s_0+l_0,\ l=|\cc_m|-l_0$ coincide.
Hence we pair the coincident terms,
thus obtaining
\begin{gather}
2 \sum_{\substack{m\\ |\cc_m|\neq 0}} \sum_{s=0}^{|\cc_m|-1} \sum_{l=|\cc_m|-s}^{|\cc_m|} \sum_{d=1}^{\dim A} f^{(b'_k),g-1} (\cdots\ot a_d\ot\cdots\ot b_d\ot\cdots) \comp \psi + \nonumber \\
+ \mathop{\delta}_{b_0>0} \sum_{d=1}^{\dim A} f^{(b'_k),g-1} (\cdots\ot a_d\ot\cdots\ot b_d\ot\cdots) \comp \psi. \label{EQSecondContr}
\end{gather}

\textbf{Third contribution.}
\eqref{EQDualOfOO} contributes to sum of the terms of the form
$$\oEnd{A}((\rho_1^{\kappa_1(a)}\kappa_1|_{C_1}\sqcup\rho_2^{\kappa_2(b)}\kappa_2|_{C_2})^{-1}) (\ooot{\kappa_1(a)}{\kappa_2(b)})_{\oEnd{A}} (\alpha\ot\alpha) (\oQO(\kappa_1^{-1})^{\#}\ot\oQO(\kappa_2^{-1})^{\#}) (\oooo{a}{b}{C_1\sqcup\{a\},G_1}{C_2\sqcup\{b\},G_2})^{\#}_{\oP}$$
of Lemma \ref{LEMMAHardToUse} by
\begin{gather}
\hspace{-.5em}\sum_{m=b_0+1}^b \hspace{-1.3em} \sum_{\substack{I,J\\ I\sqcup J=\\ [b]-[b_0]-\{m\}}} \hspace{-1.1em} \sum_{e=0}^{b_0} \!\! \sum_{\substack{g_1,g_2\\ g_1+g_2=g}} \hspace{-.8em} \sum_{s=0}^{|\cc_m|-1} \sum_{l=0}^{|\cc_m|} \sum_{d=1}^{\dim A} \!\! \delta^1 \!\cdot\! \left( f^{(b^1_k),g_1}(\cdots\ot a_d\ot\cdots)\!\cdot\! f^{(b^2_k),g_2}(\cdots\ot b_d\ot\cdots) \right) \!\comp\! \psi + \nonumber \\
+ \sum_{\substack{I,J\\ I\sqcup J=\\ [b]-[b_0]}} \sum_{e=0}^{b_0-1} \sum_{\substack{g_1,g_2\\ g_1+g_2=g}} \sum_{d=1}^{\dim A} \delta^2 \cdot \left( f^{(b^1_k),g_1}(\cdots\ot a_d\ot\cdots) \!\cdot\! f^{(b^2_k),g_2}(\cdots\ot b_d\ot\cdots) \right) \!\comp\! \psi, \label{EQThirdContr}
\end{gather}
where, in the first sum, $a_k$ sits at position $\kappa^1(a)$, where $\kappa^1:C_1\sqcup\{a\}\to[|C_1|+1]$ is such that $\oQO(\kappa^1)(\{\cyc{ac_m^{s+1}\cdots c_m^{s+l}},\cc_1,\ldots,\cc_e,\cc_{i_1},\ldots,\cc_{i_{|I|}}\}^g) = \overline{(b^1_k)}^{g_1}$;
similarly for the second factor;
and $\psi=\rho_1^{\kappa_1(a)}\kappa_1|_{C_1}\sqcup\rho_2^{\kappa_2(b)}\kappa_2|_{C_2}\in\Sigma_n$.
As usually, this permutation is best understood on an example:

\begin{example}
Let $\overline{(b_k)}^g=\{\cyc{},\cyc{12}\}^g$.
$(\ooo{a}{b})^{\#}_{\oQO}$ applied to this contains a summand
$$\{\cyc{a},\cyc{12}\}^{g_1} \ot \{\cyc{b}\}^{g_2},$$
for some $g_1+g_2=g$.
Its corresponding term is
$$\sum_d \pm f^{(0,1,1,0,\ldots),g_1}(a_d,x_1,x_2) \cdot f^{(0,1,0,\ldots),g_2}(b_d).$$
The only other possible order of the arguments of the left factor is $(a_d,x_2,x_1)$.
\end{example}

To summarize the application of Lemma \ref{LEMMAHardToUse} for $\oP=\oQO$, notice that its main equation takes the following form:
On the left-hand side, we have $d(f^{(b_k),g})(x_1,x_2,\ldots)$ for some fixed $q=\overline{(b_k)}^g$ and $x_1,\ldots,x_n\in A$.
To get the right-hand side, we collect all possible basis elements $y\in\oQO$ such that $\oxi{ab}(y)=q$ and all possible pairs $z_1\ot z_2$ of basis elements such that $\ooo{a}{b}(z_1\ot z_2)=q$.
To each such $y$, there is a term on the right-hand side of the form $\sum_d \pm f^{(b_k'),g'}\tilde{\kappa}(a_d,b_d,x_1,x_2,\ldots)$.
The calculations above clarify how to get $(b_k')$,$g'$ and $\tilde{\kappa}\in\Sigma_{n+2}$ from $y$ in a very easy way.
Similarly, each $z_1\ot z_2$ contributes the term $\sum_d \pm (f^{(b_k^1),g_1}\cdot f^{(b_k^2),g_2})\tilde{\kappa}(a_d,b_d,x_1,x_2,\ldots)$.

\begin{theorem}\label{THMQA_infty}
An algebra over $\feyn{\oQO}$ on a dg symplectic vector space $A$ is equivalently given by a collection of degree $0$ linear maps
$$f^{(b_k),g}:A^{\ot n}\to \fld$$
indexed by all eventually zero sequences $(b_k)_{k=0}^{\infty}$ of nonnegative integers satisfying the stability condition
$$4g+2b-4+n>0,$$
where $n:=\sum_{k=0}^\infty kb_k$ and $b:=\sum_{k=0}^\infty b_k$.
These maps are required to satisfy:
\begin{enumerate}
\item The $\Sigma_n$-stabilizer of $f^{(b_k),g}$ in $\mathop{Hom}(A^{\ot n},\fld)$ contains the $\Sigma_n$-stabilizer of $\overline{(b_k)}^g$ in $\oQO([n],2g+b-1)$, where $b=\sum_{k=0}^{\infty}b_k$,
\item The equation
\begin{gather} \label{EQMasterQuantumAoo}
d\left( f^{(b_k),g} \right) = \eqref{EQFirstContr} + \eqref{EQSecondContr} + \frac{1}{2}\eqref{EQThirdContr}
\end{gather}
holds.
\end{enumerate}
\end{theorem}

\begin{example}
$F:=f^{(0,0,0,2,0,\ldots),g}:A^{\ot 6}\to\fld$ has symmetries generated by the following three permutations:
\begin{gather*}
F(x_1,x_2,x_3,x_4,x_5,x_6) = (-1)^{\dg{x_3}(\dg{x_1}+\dg{x_2})} F(x_3,x_1,x_2,x_4,x_5,x_6), \\
F(x_1,x_2,x_3,x_4,x_5,x_6) = (-1)^{\dg{x_6}(\dg{x_4}+\dg{x_5})} F(x_1,x_2,x_3,x_6,x_4,x_5), \\
F(x_1,x_2,x_3,x_4,x_5,x_6) = (-1)^{(\dg{x_1}+\dg{x_2}+\dg{x_3})(\dg{x_4}+\dg{x_5}+\dg{x_6})} F(x_4,x_5,x_6,x_1,x_2,x_3).
\end{gather*}
\end{example}

\begin{remark}
If $b=0$ is also included in the definition of $\oQO$, the only new thing we get is $\{\}^g \in \oQO(\emptyset,2g-1)$ for $g\geq 2$.
These elements can't be composed using $\ooo{a}{b}$ nor $\xi_{ab}$ and even don't affect the dual of the structure maps.
Call $\widetilde{\oQO}$ this extension of $\oQO$.
Let $\alpha:\feyn{\widetilde{\oQO}}\to\oEnd{A}$ be an operad morphism.
Then $\alpha(\{\}^g):\fld\to\fld$ and so $d_{\oEnd{A}}\alpha(\{\}^g)=0$.
We easily see $\partial_{\feyn{\oQO}}(\{\}^g)=0$, hence $d_{\oEnd{A}}\alpha(\{\}^g)=\alpha\partial_{\feyn{\oQO}}(\{\}^g)$ is tautological.
Hence, for algebras over $\feyn{\oQO}$ extended in this way, we are only getting a collection of scalars which don't interact in any way with the rest of the operations $f^{(b_k),g}$.
\end{remark}


\subsubsection{\texorpdfstring{Relation to Herbst's quantum $A_\infty$ algebras}{Relation to Herbst's quantum A\_oo algebras}} \label{SECHerbstOne}

In this Section, we recover the results of Herbst in \cite{Herbst} concerning the quantum $A_\infty$ algebras satisfying $f^{(b_k),g}=0 \textrm{ whenever }b_0>0$.

In physics, so called string vertices are used rather than the collection $\{f^{(b_k),g}\}$.
Roughly speaking, the string vertex is an $f^{(b_k),g}$ evaluated at some fixed vectors of $A$.

Define the reduced string vertex
\begin{gather}
\mathcal{F}^{g,b}_{v_1,\ldots,v_{l_1}|v_{l_1+1},\ldots,v_{l_1+l_2}|\cdots|v_{l_1+\cdots+l_{\ob-1}+1},\ldots,v_{l_{l_1+\cdots+l_{\ob-1}+l_{\ob}}}} :=  \label{EQBigF} \\
= -\frac{1}{2} (f^{(b_k),g}\comp\beta_{(l_1,\ldots,l_\ob)})(v_1\ot\cdots\ot v_n) = \nonumber \\
= \pm -\frac{1}{2} f^{(b_k),g} ( v_{l_{1}+\cdots+l_{\beta^{-1}(1)-1}+1}\ot\ldots\ot v_{l_1+\cdots+l_{\beta^{-1}(1)-1}+l_{\beta^{-1}(1)}}\ot\cdots \nonumber \\
\phantom{= \pm -10 f^{(b_k),g} (} \cdots \ot v_{l_{1}+\cdots+l_{\beta^{-1}(\ob)-1}+1}\ot\ldots\ot v_{l_1+\cdots+l_{\beta^{-1}(\ob)-1}+l_{\beta^{-1}(\ob)}} ), \nonumber
\end{gather}
where the b-sequence $(b_k)$ and $\beta\in\Sigma_\ob$ on the RHS are determined as follows:
for $k\geq 1$, $b_k$ is the number of $l_i$'s equal to $k$; then $b_0:=b-\ob$.
The permutation $\beta\in\Sigma_\ob$ is arbitrary such that $\beta(l_1,\ldots,l_\ob)$ is nondecreasing, i.e.
$$l_{\beta^{-1}(1)} \leq \cdots \leq l_{\beta^{-1}(\ob)}$$
and $\beta_{(l_1,\ldots,l_b)}\in\Sigma_n$ is the block permutation permuting blocks $1,\ldots,l_1$ and $l_1+1,\ldots,l_1+l_2$ and so on, according to $\beta$.
The sign $\pm$ is the Koszul sign of the action of $\beta_{(l_1,\ldots,l_b)}$ on $v$'s.
The coefficient $-1/2$ is purely conventional.

The intuition behind the formula \eqref{EQBigF} is roughly this: the blocks in the subscript of $\mathcal{F}$, separated by ``$|$'', correspond to nonempty cycles of $\{\cc_1,\ldots,\cc_b\}^g$.
The subscript also determines an order of the cycles.
Before evaluating, we reorder the cycles using $\beta$ so that the lengths of the cycles form a nondecreasing sequence.

Notice that the above requirement doesn't determine $\beta$ uniquely: if $\beta'\in\Sigma_\ob$ is another permutation such that $l_{\beta'^{-1}(1)} \leq \cdots \leq l_{\beta'^{-1}(\ob)}$, then it is easy to verify that $\beta'_{(l_1,\ldots,l_\ob)}\beta^{-1}_{(l_1,\ldots,l_\ob)} \in \mathrm{Stab}(\overline{(b_k)}^g)$, hence $f^{(b_k),g}\comp\beta_{(l_1,\ldots,l_\ob)} = f^{(b_k),g}\comp\beta'_{(l_1,\ldots,l_\ob)}$ and the reduced string vertex is independent of this choice.

\begin{example}
\begin{gather*}
\mathcal{F}^{g,4}_{v_{1}|v_{2}v_{3}} = -\frac{1}{2}f^{(2,1,1,0,\ldots),g}(v_1\ot v_2\ot v_3) \\
\mathcal{F}^{g,2}_{v_{1}v_{2}|v_{3}} = -(-1)^{\dg{v_3}(\dg{v_1}+\dg{v_2})}\frac{1}{2}f^{(0,1,1,0,\ldots),g}(v_3\ot v_1\ot v_2)
\end{gather*}
\end{example}

It is now easy to express the equation \eqref{EQMasterQuantumAoo} in terms of the reduced string vertices.
With the forthcoming Theorem \ref{THMHerbstMain} in mind, we assume $b_0=0$ and thus we will be interested only in the first terms of the three contributions.

\begin{example} \label{EXTheOnlYOneQuestMark}
We express the first term in \eqref{EQFirstContr} evaluated on $v_1\ot\cdots\ot v_n \in A^{\ot n}$.
To lighten the notation, we will write $k$ instead of $v_k$ in the subscript of $\mathcal{F}$.
Similarly, we write $a$ instead of $a_d$ and $b$ instead of $b_d$.
Further, given $\overline{(b_k)}^g = \{\cc_1,\ldots,\cc_b\}^g$ as in \eqref{EQOrbitRepresentative}, we write $\cc_k$ instead of $(c_k^{1},\ldots,c_k^{|\cc_k|})$ in the subscript of $\mathcal{F}$.
Finally, if an empty cycle appears in the subscript of $\mathcal{F}$, then we omit it.

We easily see that the first term of \eqref{EQFirstContr} yields (omitting the summations)
\begin{gather*}
2f^{(b'_k),g}(\cdots\ot a_d\ot\cdots\ot b_d\ot\cdots)\psi(v_1\ot\cdots\ot v_n) = \\
= -\pm 4 \mathcal{F}^{g,b-1}_{ac_i^{p+1}\cdots c_i^{p+|\cc_i|}bc_j^{q+1}\cdots c_j^{q+|\cc_j|}|\cc_1|\cdots\widehat{\cc_i}\cdots\widehat{\cc_j}\cdots|\cc_b},
\end{gather*}
where $\pm$ is the Koszul sign of permuting $v_{ab1\cdots n}$ to $v_{ac_i^{p+1}\cdots c_i^{p+|\cc_i|}bc_j^{q+1}\cdots c_j^{q+|\cc_j|},\cc_1,\cdots\widehat{\cc_i}\cdots\widehat{\cc_j}\cdots,\cc_b}$.

The first term in \eqref{EQSecondContr} is handled analogously.
\end{example}

\begin{example} \label{EXTheOnlYOneQuestMarkII}
Using the abbreviations as in the previous example, the first term in \eqref{EQThirdContr} can be expressed, upon evaluation on $v_1\ot\cdots\ot v_n$, by
\begin{gather*}
\left( f^{(b^1_k),g_1}(\cdots\ot a_d\ot\cdots) \cdot f^{(b^2_k),g_2}(\cdots\ot b_d\ot\cdots) \right) \psi (v_1\ot\cdots\ot v_n) = \\
= \pm 4 \mathcal{F}^{g_1,|I|+1}_{ac_m^{s+1}\cdots c_m^{s+l}|\cc_{i_1}|\cdots|\cc_{i_{|I|}}} \cdot \mathcal{F}^{g_2,|J|+1}_{bc_m^{s+l+1}\cdots c_m^{s+|\cc_m|}|\cc_{j_1}|\ldots|\cc_{j_{|J|}}},
\end{gather*}
where $\pm$ is the Koszul sign of permuting $v_{ab1\cdots n}$ to $v_{ac_m^{s+1}\cdots c_m^{s+l},\cc_{i_1},\cdots,\cc_{i_{|I|}},bc_m^{s+l+1}\cdots c_m^{s+|\cc_m|},\cc_{j_1},\cdots,\cc_{j_{|J|}}}$.
\end{example}

Finally, let's observe that the reduced string vertices have the expected symmetries:
\begin{gather}
\mathcal{F}^{g,b}_{\cdots|v_1\cdots v_i|w_1\cdots w_j|\cdots} = (-1)^{(\sum_{k=1}^i\dg{v_k})(\sum_{k=1}^j\dg{w_k})} \mathcal{F}^{g,b}_{\cdots|w_1\cdots w_j|v_1\cdots v_i|\cdots} \nonumber \\
\mathcal{F}^{g,b}_{\cdots|v_1v_2\cdots v_{i-1}v_i|\cdots} = (-1)^{\dg{v_i}(\sum_{k=1}^{i-1}\dg{v_k})} \mathcal{F}^{g,b}_{\cdots|v_iv_1v_2\cdots v_{i-1}|\cdots} \label{EQSignsInHerbst}
\end{gather}

Now we can state the precise comparison to Herbst's work:

\begin{theorem} \label{THMHerbstMain}
Let $(A,d=0,\omega)$ be a symplectic dg vector space with zero differential.
If $A$ carries a structure of algebra over $\feyn{\oQO}$ satisfying
$$f^{(b_k),g}=0 \textrm{ whenever }b_0>0,$$
then for any $\overline{(b_k)}^g=\{\cc_1,\ldots,\cc_b\}^g$ with $b_0=0$ and any $v_1,\ldots,v_n\in A$ we have
\begin{gather}
\sum_{\substack{i,j\\ i<j}} \sum_{p=0}^{|\cc_i|-1} \sum_{q=0}^{|\cc_j|-1} \sum_{d=1}^{\dim A} 
\pm \mathcal{F}^{g,b-1}_{ac_i^{p+1}\cdots c_i^{p+|\cc_i|}bc_j^{q+1}\cdots c_j^{q+|\cc_j|}|\cc_1|\cdots|\widehat{\cc_i}|\cdots|\widehat{\cc_j}|\cdots|\cc_b} + \nonumber \\
+ \sum_{m=1}^b \sum_{s=0}^{|\cc_m|-1} \sum_{l=|\cc_m|-s}^{|\cc_m|} \sum_{d=1}^{\dim A} 
\pm \mathcal{F}^{g-1,b+1}_{ac_m^{s+1}\cdots c_m^{s+l}|bc_m^{s+l+1}\cdots c_m^{s+|\cc_m|}|\cc_1|\cdots|\widehat{\cc_m}|\cdots|\cc_b} = \label{EQMEHErbst} \\
= \frac{1}{2} \sum_{m=1}^b \sum_{\substack{I,J\\ I\sqcup J=\\ [b]-\{m\}}} \sum_{\substack{g_1,g_2\\ g_1+g_2=g}} \sum_{s=0}^{|\cc_m|-1} \sum_{l=0}^{|\cc_m|} \sum_{d=1}^{\dim A} \pm \delta^1 \cdot \mathcal{F}^{g_1,|I|+1}_{ac_m^{s+1}\cdots c_m^{s+l}|\cc_{i_1}|\cdots|\cc_{i_{|I|}}} \cdot \mathcal{F}^{g_2,|J|+1}_{bc_m^{s+l+1}\cdots c_m^{s+|\cc_m|}|\cc_{j_1}|\cdots|\cc_{j_{|J|}}}, \nonumber
\end{gather}
where $\delta^1$ is as in \eqref{EQdelta1} with $e=b_0=0$ and the signs are Koszul signs produced by permuting $a_d\ot b_d\ot v_{1\cdots n}$ into the orders indicated by the subscripts of $\mathcal{F}$'s.
Notice that we are using the abbreviations introduced in Example \ref{EXTheOnlYOneQuestMark}.

The above equation is precisely the Herbst's \emph{minimal quantum $A_\infty$ relation} of Theorem $1$ of \cite{Herbst}.
\end{theorem}

\begin{proof}
Assume the algebra over $\feyn{\oQO}$ is given and let's rewrite Equation \eqref{EQMasterQuantumAoo} in terms of the reduced string vertices.
The first sum was worked out in detail including the sign in Example \ref{EXTheOnlYOneQuestMark}.
The second sum is completely analogous.
The third sum was worked out in Example \ref{EXTheOnlYOneQuestMarkII}.
This yields \eqref{EQMEHErbst}.

At this point, our notation is almost the same as Herbst's in Equation $(24)$ of \cite{Herbst}.
We just need to adjust the summation in the last term:

First, we consider a pairing similar to that used to get \eqref{EQSecondContr}:
Let the numbers $m$ and those of the tuple $(I,J,g_1,g_2,s,l)$ have the meaning as in the summations \eqref{EQThirdContr} or the RHS of \eqref{EQMEHErbst}.
Given $m$, pair $$(I^0,J^0,g_1^0,g_2^0,s^0,l^0) \leftrightarrow (J^0,I^0,g_2^0,g_1^0,s^0+l^0,|\cc_m|-l^0).$$
The corresponding $\pm\mathcal{F}'s$ coincide.
Thus the RHS of \eqref{EQMEHErbst} becomes
\begin{gather} \label{EQHerbstTwo}
\sum_{m=1}^b \sum_{\substack{I,J\\ I\sqcup J=\\ [b]-\{m\}}} \sum_{\substack{g_1,g_2\\ g_1+g_2=g}} \sum_{s=0}^{|\cc_m|-1} \sum_{l=|\cc_m|-s}^{|\cc_m|} \sum_{d=1}^{\dim A} \pm \delta^1 \cdot \mathcal{F}^{g_1,|I|+1}_{ac_m^{s+1}\cdots c_m^{s+l}|\cc_{i_1}|\cdots|\cc_{i_{|I|}}} \cdot \mathcal{F}^{g_2,|J|+1}_{bc_m^{s+l+1}\cdots c_m^{s+|\cc_m|}|\cc_{j_1}|\cdots|\cc_{j_{|J|}}}.
\end{gather}

Next step is to replace the summation 
$$\sum_{m=1}^{b} \sum_{\substack{I,J\\ I\sqcup J=\\ [b]-\{m\}}} \quad\textrm{by}\quad \sum_{\substack{|I|,|J|\geq 0\\ |I|+|J|=b-1}} \sum_{\sigma\in\Sigma_{b}}.$$
The correspondence is
$$(|I|,|J|,\sigma) \mapsto (\sigma(1), \{\sigma(2),\ldots,\sigma(|I|+1)\}, \{\sigma(|I|+2),\ldots,\sigma(\underbrace{|I|+|J|+1}_{b})\}).$$
This is surjective but not injective.
Fortunately, it is easy to see that each preimage has $|I|!|J|!$ elements.
Using \eqref{EQSignsInHerbst}, we see that \eqref{EQHerbstTwo} equals
\begin{gather*}
\sum_{\substack{|I|,|J|\geq 0\\ |I|+|J|=b-1}} \sum_{\sigma\in\Sigma_{b}} \sum_{\substack{g_1,g_2\\ g_1+g_2=g}} \sum_{s=0}^{|\cc_m|-1} \sum_{l=|\cc_m|-s}^{|\cc_m|} \sum_{d=1}^{\dim A} \\
\pm \delta^1 \ \frac{1}{|I|!|J|!} \mathcal{F}^{g_1,|I|+1}_{ac_{\sigma(1)}^{s+1}\cdots c_{\sigma(1)}^{s+l}|\cc_{\sigma(2)}|\cdots|\cc_{\sigma(|I|+1)}} \cdot \mathcal{F}^{g_2,|J|+1}_{bc_{\sigma(1)}^{s+l+1}\cdots c_{\sigma(1)}^{s+|\cc_{\sigma(1)}|}|\cc_{{\sigma(|I|+2)}}|\cdots|\cc_{\sigma(b)}}.
\end{gather*}
Finally, the stability condition, represented by the $\delta^1$ above, corresponds to the Herbst's notion of minimality.
\end{proof}

\begin{remark}
A priori, the obvious converse of Theorem \ref{THMHerbstMain} doesn't hold.
In fact, algebras over $\feyn{\oQO}$ with $f^{(b_k),g}=0$ for $b_0>0$ satisfy not only equations \eqref{EQMEHErbst}, but also equations \eqref{EQMasterQuantumAoo} with $b_0=0$, whose LHS vanishes but there can be nonzero terms on the RHS.
\end{remark}


\subsubsection{\texorpdfstring{Master equation}{Master equation}} \label{SECBarQO}

In this section, we apply the Barannikov's theory of Section \ref{SSECBarannikov} to get a master equation describing quantum $A_\infty$-algebras.
We obtain results dual to those of the preceding Section \ref{SECHerbstOne}, except that we allow for empty boundaries.

The quantum $A_\infty$-algebras are degree $0$ solutions $S$ of the master equation 
\begin{gather} \label{EQMEQAoo}
dS+\Delta S+\frac{1}{2}\{S,S\} = 0
\end{gather}
solved in the space 
$$\tilde{P} = \prod_{n,G} \oQO(n,G) \ot_{\Sigma_n} {A^\#}^{\ot n}.$$
We introduce a notation similar to that for the reduced string vertices \eqref{EQBigF}:
For each $1\leq k\leq\overline{b}$, let $I^k=(I^k_1\cdots I^k_{l_k})$ be an $l_k$-tuple in $[\dim A]$.
Denote
\begin{gather} \label{EQHerbstPhi}
\phi^{I^1|\cdots|I^{\overline{b}};g,b} := \overline{(b_k)}^g \ot_{\Sigma_n} \beta_{(l_1,\ldots,l_{\overline{b}})}\phi^{I^1\cdots I^{\overline{b}}} \in \oQO([n],G) \ot_{\Sigma_n} {A^\#}^{\ot n},
\end{gather}
where $(b_k)$ is the $b$-sequence corresponding to $\overline{b}$ cycles of lengths $l_1,\ldots,l_{\overline{b}}$ and $b-\overline{b}$ empty cycles, $n=l_1+\cdots l_{\overline{b}}$ and $\beta_{(l_1,\ldots,l_{\overline{b}})}\in\Sigma_n$ is the block permutation described below \eqref{EQBigF} and $I^k$ is an $l_k$-tuple $(I^k_1\cdots I^k_{l_k})$ for each $k$; and $G=2g+b-1$.

Notice that every $\phi^{I^1|\cdots|I^{\overline{b}};g,b} \in \tilde{P}$ is restricted by the stability condition $2(2g+b-2)+\sum_{k=1}^{\overline{b}}l_k > 0$.
Obviously, every element of $\tilde{P}$ can be written in the form
$$S = \sum_{\substack{g,\overline{b}\\ I^1,\ldots,I^{\overline{b}}}} C_{I^1|\cdots|I^{\overline{b}}}^{g,b}\phi^{I^1|\cdots|I^{\overline{b}};g,b}$$
for some coefficients $C_{I^1|\cdots|I^{\overline{b}}}^{g,b}\in\fld$.
Notice that this expression is not unique (see also Section \ref{SECHerbstII} below).


Using Lemma \ref{LEMBarannikovsIsoFormulas}, we easily obtain the following formula for the BV operator $\Delta$:
\begin{gather}
\Delta \phi^{I^1|\cdots|I^{\overline{b}};g,b} = \nonumber \\
= 2\sum_{\substack{p<q\\ i,j}} \pm \omega^{I^p_iI^q_j} \phi^{ I^p_{i+1}\cdots I^p_{l_p} I^p_{1}\cdots I^p_{i-1} I^q_{j+1}\cdots I^q_{l_q} I^q_1\cdots I^q_{j-1} | I^1|\cdots|\widehat{I^p}|\cdots|\widehat{I^q}|\cdots|I^{\overline{b}} ;g+1,b-1 } + \label{EQQOCDeltaExplicit} \\
+ 2\sum_{\substack{p\\ i<j}} \pm \omega^{I^p_iI^q_j} \phi^{I^p_{i+1}\cdots I^p_{j-1} | I^p_{j+1}\cdots I^p_{l_p}I^p_1\cdots I^p_{i-1} | I^1|\cdots|\widehat{I^p}|\cdots|I^{\overline{b}};g,b+1}. \nonumber
\end{gather}
In the first term, the sign consists of $(-1)^{\dg{\phi^{I^1\cdots I^{\overline{b}}}}+\dg{\phi^{I^q_j}}}$ and of the Koszul sign of permuting $\phi^{I^1\cdots I^{\overline{b}}}$ to $\phi^{I^p_iI^q_jI^p_{i+1}\cdots I^p_{l_p}I^q_1\cdots I^q_{j-1}I^q_{j+1}\cdots I^q_{l_q}I^p_1\cdots I^p_{i-1}I^1\cdots\widehat{I^p}\cdots\widehat{I^q}\cdots I^{\overline{b}}}$ (this Koszul sign already includes the signs coming from the positional derivatives of Lemma \ref{LEMBarannikovsIsoFormulas}).
Notice that this permutation brings the order of the superscript indices $I^1\cdots I^{\overline{b}}$ appearing in the argument of $\Delta$ on the first line to their order in the second line.
The sign in the second term is similar.

For the BV bracket, we obtain
\begin{gather}
\left\{ \phi^{I^1|\cdots|I^{\overline{b}_1};b_1,g_1} , \phi^{J^1|\cdots|J^{\overline{b}_2};b_2,g_2} \right\} = \nonumber \\
= \sum_{p,q,i,j} \pm \omega^{I^p_iJ^q_j} \phi^{ I^p_{i+1}\cdots I^p_{l_p}I^p_1\cdots I^p_{i-1}J^q_{j+1}\cdots J^q_{m_q}J^q_1\cdots J^q_{j-1}|I^1|\cdots|\widehat{I^p}|\cdots|I^{\overline{b}_1}|J^1|\cdots|\widehat{J^q}|\cdots|J^{\overline{b}_2} }, \label{EQQOCBrackerExplicit}
\end{gather}
where, for each $k$, $l_k$ is the length of the tuple $I^k$; and similarly $m_k$'s are lengths of $J^k$'s.
The sign consists of the factor $(-1)^{ \dg{\phi^{I^p_i}}(\dg{\phi^{I^1\cdots I^{\overline{b}_1}}}+1) + \dg{\phi^{J^1\cdots J^{\overline{b}_2}}} + \dg{\phi^{I^1\cdots I^{\overline{b}_1}}}\dg{\phi^{J^1\cdots J^{\overline{b}_2}}} + 1 }$ of Lemma \ref{LEMBarannikovsIsoFormulas}) and of the Koszul sign of permutation taking $\phi^{I^1\cdots I^{\overline{b}_1}J^1\cdots J^{\overline{b}_2}}$ to $\phi^{ J^q_i I^p_i I^p_{i+1}\cdots I^p_{l_p}I^p_1\cdots I^p_{i-1}J^q_{j+1}\cdots J^q_{m_q}J^q_1\cdots J^q_{j-1} I^1 \cdots \widehat{I^p} \cdots I^{\overline{b}_1} J^1 \cdots \widehat{J^q} \cdots J^{\overline{b}_2} }$.
Again, notice that this permutation corresponds to the change of order of indices from the first to the second line except for the switch of $I^p_i$ and $J^q_i$.

Notice that the restrictions due to stability are included already in the allowed ``monomials'' $\phi^{I^1|\cdots|I^{\overline{b}};g,b}$ and, contrary to Sections \ref{SECAxiomsForQuantumAoo} and \ref{SECHerbstOne}, don't complicate the master equation.


\subsubsection{\texorpdfstring{Comparison to Herbst's generating function}{Comparison to Herbst's generating function}} \label{SECHerbstII}

\begin{lemma} \label{LEMAnotherFormOfS}
Every element in $\tilde{P}(n,G)$ can be uniquely expressed as
\begin{gather}
S_{n,G}=\sum_{\substack{g,b\\ I^1,\ldots,I^{\overline{b}}}} \frac{1}{\overline{b}!l_1\cdots l_{\overline{b}}} C^{g,b}_{I^{1}|\cdots|I^{\overline{b}}} \ \phi^{I^{1}|\cdots|I^{\overline{b}};g,b}, \label{EQHerbstFormOfS}
\end{gather}
where the summation runs through $g\geq 0$, $b\geq 1$, $\overline{b}\geq 0$ and $I^k\in[\dim A]^{\times l_k}$ for each $1\leq k\leq\overline{b}$, such that $\overline{b}\leq b$ and $l_1+\cdots+l_{\overline{b}}=n$ and $G=2g+b-1$;
and the coefficients $C^{g,b}_{I^{1}|\cdots|I^{\overline{b}}} \in \fld$ are required to satisfy
\begin{gather} \label{EQSymmetriesOfCS}
\begin{split}
C^{g,b}_{\cdots|I^{k}|I^{k+1}|\cdots} = (-1)^{\dg{\phi^{I^k}}\dg{\phi^{I^{k+1}}}} C^{g,b}_{\cdots|I^{k+1}|I^{k}|\cdots}, \\
C^{g,b}_{\cdots|I^k_1\cdots I^k_{l_k}|\cdots} = (-1)^{\dg{\phi^{I^k_1}} (\dg{\phi^{I^k_2}}+\cdots+\dg{\phi^{I^k_{l_k}}})} C^{g,b}_{\cdots|I^k_2\cdots I^k_{l_k}I^k_1|\cdots}
\end{split}
\end{gather}
for each $k$, where $I^k=(I^k_1\cdots I^k_{l_k})$.

If $l_1\leq\cdots\leq l_{\overline{b}}$ and $(b_k)$ is the corresponding $b$-sequence, let 
$$f^{(b_k),g}(a_{I^{1}\cdots I^{\overline{b}}})=C^{g,b}_{I^{1}|\cdots|I^{\overline{b}}}.$$
Then $\sum_{n,G}S_{n,G}$ satisfies the master equation \eqref{EQMEQAoo} iff $f^{(b_k),g}$'s satisfy the conditions of Theorem \ref{THMQA_infty}.
\end{lemma}

\begin{proof}
Let's describe a basis of $\tilde{P}(n,G)$.
Let $\Stab\overline{(b_k)}^g$ be the stabilizer of $\overline{(b_k)}^g\in\oQO([n],G)$ inside $\Sigma_n$.
Observe that
$$|\Stab\overline{(b_k)}^g|=\prod_{k\geq 1}b_k!k^{b_k}.$$
$\Stab\overline{(b_k)}^g$ acts on $[\dim A]^{\times n}$.
For every orbit of this action, choose a representative.
Let $B^{(b_k),g}$ be the set of all these chosen representatives.
Then it is easy to see that the basis of $\tilde{P}(n,G)$ consists of all elements of the form
$$\overline{(b_k)}^g\ot_{\Sigma_n}\phi^I,$$
where $\overline{(b_k)}^g$ runs through all the canonical $\Sigma_n$ orbit representatives \eqref{EQOrbitRepresentative} in $\oQO([n],G)$ and $I$ runs through $B^{(b_k),g}$.
Thus arbitrary element of $\tilde{P}(n,G)$ can be written uniquely in the form
\begin{gather} \label{EQQOTildePBasis}
\sum_{\substack{\overline{(b_k)}^g\\ I\in B^{(b_k),g}}} C^{(b_k),g}_I \ \overline{(b_k)}^g\ot_{\Sigma_n}\phi^I
\end{gather}
for some coefficients $C^{(b_k),g}_I\in\fld$.

Now we wish to sum over all $I\in[\dim A]^{\times n}$.
For a fixed $I\in B^{(b_k),g}$ and $\sigma\in\Stab\overline{(b_k)}^g$, observe that $\overline{(b_k)}^g\ot_{\Sigma_n}\phi^{\sigma I} = \pm \overline{(b_k)}^g\ot_{\Sigma_n}\phi^{I}$.
With the same sign, we define
$$C^{(b_k),g}_{\sigma I} := \pm C^{(b_k),g}_I.$$
Replacing $B^{(b_k),g}$ by $[\dim A]^{\times n}$ in the summation above overcounts by the factor $|\Stab\overline{(b_k)}^g|$.
Thus
\begin{gather}
\sum_{\substack{\overline{(b_k)}^g\\ I\in B^{(b_k),g}}} C^{(b_k),g}_I \overline{(b_k)}^g\ot_{\Sigma_n}\phi^I = \sum_{\substack{\overline{(b_k)}^g\\ I\in[\dim A]^{\times n}}} \frac{1}{|\Stab\overline{(b_k)}^g|} C^{(b_k),g}_I \ \overline{(b_k)}^g\ot_{\Sigma_n}\phi^I. \label{EQSecondUniqueExpression}
\end{gather}

Next, observe that the choice of $\overline{(b_k)}^g$ is equivalent to choosing $g,b,\overline{b}$ and $l_1,\ldots,l_{\overline{b}}$ such that $G=2g+b-1$, $\overline{b}\leq b$, $l_i\geq 1$ for each $i$ and $l_1+\cdots+l_{\overline{b}}=n$.
Notice that the order of $l_i$'s is irrelevant.
In other words, we are choosing an unordered partition of $n$, which we will encode by a multiset\footnote{That is, we allow for repeated elements.} $\{l_1,\ldots,l_{\overline{b}}\}$.
So we see that we are in fact summing over all $g,b$ and unordered partitions $\{l_1,\ldots,l_{\overline{b}}\}$ subject to the conditions above.
It is more natural to sum simply over positive numbers $l_1,\ldots,l_{\overline{b}}$ (that is, over ordered partitions of $n$).
To change the summation, consider a decomposition $I^1,\ldots,I^{\overline{b}}$ of $[n]$ into disjoint subsets consisting of $l_1,\ldots,l_{\overline{b}}$ elements.
If $l_1\leq\cdots\leq l_{\overline{b}}$, define
$$C^{g,b}_{I^{1}|\cdots|I^{\overline{b}}} := C^{(b_k),g}_{I^{1}\cdots I^{\overline{b}}},$$
where the multiindices in the subscript of the RHS are concatenated.
For $\sigma\in\Sigma_{\overline{b}}$, observe that
$$\phi^{I^{\sigma^{-1}(1)}|\cdots|I^{\sigma^{-1}(\overline{b})};g,b} = \pm \phi^{I^{1}|\cdots|I^{\overline{b}};g,b}$$
(recall \eqref{EQHerbstPhi}).
With the same sign, we define
$$C^{g,b}_{I^{\sigma^{-1}(1)}|\cdots|I^{\sigma^{-1}(\overline{b})}} := \pm C^{g,b}_{I^{1}|\cdots|I^{\overline{b}}}.$$
Notice that this is analogous to the definition of $\mathcal{F}^{g,b}_{I^1|\cdots|I^n}$ in \eqref{EQBigF}.
Let $UP(n)$ be the set of all unordered partitions $\{l_1,\ldots,l_{\overline{b}}\}$ of $n$.
Let $OP(n)$ be the set of all ordered partitions $(l_1,\ldots,l_{\overline{b}})$ of $n$.
By replacing $UP(n)$ in the summation \eqref{EQSecondUniqueExpression} by $OP(n)$, we overcount again.
To see how many times, look at the map $OP(n)\to UP(n)$ given by $(l_1,\ldots,l_{\overline{b}}) \mapsto \{l_1,\ldots,l_{\overline{b}}\}$.
Let $(b_k)$ be the $b$-sequence associated to $\{l_1,\ldots,l_{\overline{b}}\}$.
Then we immediately see that the inverse image of $\{l_1,\ldots,l_{\overline{b}}\}$ contains exactly $\frac{\overline{b}!}{\prod_{k\geq 1}b_k!}$ elements.
Thus
\begin{gather*}
\sum_{\substack{\overline{(b_k)}^g\\ I\in[\dim A]^{\times n}}} \frac{1}{|\Stab\overline{(b_k)}^g|} C^{(b_k),g}_I \ \overline{(b_k)}^g\ot_{\Sigma_n}\phi^I = \sum_{\substack{g,b\\ I^1,\ldots,I^{\overline{b}}}} \frac{1}{|\Stab\overline{(b_k)}^g|} \frac{\prod_{k\geq 1}b_k!}{\overline{b}!} C^{g,b}_{I^{1}|\cdots|I^{\overline{b}}} \ \phi^{I^{1}|\cdots|I^{\overline{b}};g,b} = \\
= \sum_{\substack{g,b\\ I^1,\ldots,I^{\overline{b}}}} \frac{1}{\overline{b}!l_1\cdots l_{\overline{b}}} C^{g,b}_{I^{1}|\cdots|I^{\overline{b}}} \ \phi^{I^{1}|\cdots|I^{\overline{b}};g,b},
\end{gather*}
where each $I^k$ runs over all elements of $[\dim A]^{\times l_k}$.

The last claim of the lemma follows by comparing \eqref{EQIsoYetOneMoreOhh} to \eqref{EQSecondUniqueExpression}:
So we need to write \eqref{EQIsoYetOneMoreOhh} in the basis as in \eqref{EQQOTildePBasis}.
Recall that $f^{(b_k),g}=\alpha(\overline{(b_k)}^g)$ (the operad morphism $\alpha:\feyn{\oQO}\to\oEnd{A}$ determines the $\feyn{\oQO}$ algebra structure) and similarly set $f^{q}=\alpha(q)$ for arbitrary $q\in\oQO$.
Let $q$ run through the basis of $\oQO([n],G)$.
\begin{gather*}
\frac{1}{n!} \sum_{\substack{q\\ I\in[\dim A]^{\times n}}} q \ot_{\Sigma_n} \alpha(q)(a_I) \phi^I = \\
= \frac{1}{n!} \sum_{\substack{\overline{(b_k)}^g\\ \sigma\in\Sigma_n \\ I\in[\dim A]^{\times n}}} \frac{1}{|\Stab{\overline{(b_k)}^g}|} \ \sigma \overline{(b_k)}^g \ot_{\Sigma_n} \alpha(\oQO(\sigma)\overline{(b_k)}^g)(a_I) \phi^I = \\
= \frac{1}{n!} \sum_{\substack{\overline{(b_k)}^g\\ \sigma\in\Sigma_n \\ I\in[\dim A]^{\times n}}} \frac{1}{|\Stab{\overline{(b_k)}^g}|} \ \overline{(b_k)}^g \ot_{\Sigma_n} \alpha(\overline{(b_k)}^g)(a_I) \phi^I = \\
= \sum_{\substack{\overline{(b_k)}^g\\ I\in[\dim A]^{\times n}}} \frac{1}{|\Stab{\overline{(b_k)}^g}|} \ \overline{(b_k)}^g \ot_{\Sigma_n} \alpha(\overline{(b_k)}^g)(a_I) \phi^I = \\
= \sum_{\substack{\overline{(b_k)}^g\\ I\in B^{(b_k),g}}} f^{(b_k),g}(a_I) \ \overline{(b_k)}^g \ot_{\Sigma_n} \phi^I,
\end{gather*}
where the final step is carried out as in \eqref{EQSecondUniqueExpression}.
We have proved $f^{(b_k),g}(a_I) = C^{(b_k),g}_I$ and this is easily seen to extend to $f^{(b_k),g}(a_I) =  C^{g,b}_{I^1|\cdots|I^{\overline{b}}}$.
\end{proof}

Lemma \ref{LEMAnotherFormOfS} implies that
$$C^{g,b}_{I^{1}|\cdots|I^{\overline{b}}} = -2 \mathcal{F}^{g,b}_{I^{1}|\cdots|I^{\overline{b}}}.$$
This substitution leads to the Herbst's formula (42) of \cite{Herbst} for the generating function:
$$S = \sum_{n,G}S_{n,G} = -2 \sum_{n,G} \sum_{\substack{g,b\\ I_1,\ldots,I_{\overline{b}}}} \frac{1}{\overline{b}!l_1\cdots l_{\overline{b}}} \mathcal{F}^{g,b}_{I^{1}|\cdots|I^{\overline{b}}} \ \phi^{I^{1}|\cdots|I^{\overline{b}};g,b}.$$

Let us note that the reduced string vertices are graded symmetric only with respect to permutations of non-empty cycles. In order to be later, when discussing the quantum open-closed string field theory, compatible with the 
physics notation, we introduce the string vertices as follows:
\begin{gather} \label{EQBigF1}
\mathcal{V}^{g,b}_{I'^1|\cdots|I'^b} = b_0!\mathcal{F}^{g,b}_{I^1|\cdots|I^\ob},
\end{gather}
where on the LHS, multiindices $I'^i$ of zero length are allowed and $I^i$ on the RHS are obtained by omitting these zero length multiindices while keeping the order of the others.
In terms of so defined string vertices we have
$$S = -2 \sum_{n,b,g} \sum_{I'^1,\ldots,I'^b} \frac{1}{b!\prod'_k l'_{k}} \mathcal{V}^{g,b}_{I'^1|\cdots|I'^b} \ \phi^{I'^1|\cdots|I'^b},$$
where $\prod'$ is the product of nonzero $l'_k$'s.


\subsection{\texorpdfstring{Relation between $\oMod{\oAss}$ and $\oQO$}{Relation between Mod(Ass) and QO}}

Motivated by the relationship between $\oQC$ and its genus zero part, cyclic $\oCom$, in closed string case (Theorem \ref{THMModCom}), one can ask the same question for open strings: Is $\oQO$ the modular envelope of its genus zero part, cyclic $\oAss$?
This is obvious neither from the topological viewpoint (in terms of $2$-dimensional surfaces), nor from the algebraic viewpoint (in terms of adding the results of $\xi_{ab}$ compositions to $\oAss$ as freely as possible).
An affirmative answer is given in \cite{MDModAss}:

\begin{theorem} \label{CONModAss}
There is a modular operad isomorphism $\oQO \cong \oMod{\oAss}$ compatible with the obvious morphisms from $\oAss$ to $\oQO$ resp. to $\oMod{\oAss}$.
\end{theorem}

\section{\texorpdfstring{Quantum open-closed Operad and related algebraic structures}{Quantum open-closed Operad and related algebraic structures}} \label{sec:qoc}

\subsection{\texorpdfstring{$2$-coloured modular operad}{2-coloured modular operad}}

Here we define a $2$-coloured modular operad with half-integral genus and related notions.
This is clearly only a provisional definition - it is coined to express the master equation for open-closed string theory.
We expect it to fit into a more general framework for ``operads of modular type''.

\begin{definition}
Let $\cCo_{2}$ be the category of stable $2$-coloured corollas: the objects are pairs $(O,C,G)$ with $O,C$ finite sets\footnote{O stands for Open, C for Closed.} and $G$ a nonnegative half-integer\footnote{That is, $G$ is of the form $a/2$ for some $a\in\N_0$.} such that the stability condition is satisfied:
\begin{gather} \label{EQTwoColorStab}
2(G-1)+|O|+|C|>0.
\end{gather}
Elements of $O$ are called open, elements of $C$ are closed.

A morphism $(O,C,G)\to(O',C',G')$ is defined only if $G=G'$ and it is a pair of bijections $O\xrightarrow{\sim}O'$ and $C\xrightarrow{\sim}C'$.
\end{definition}

\begin{definition}
To define a $2$-coloured modular operad, replace $\cCo$ by $\cCo_2$ everywhere in Definition \ref{DEFModOp} and also require $\ooo{a}{b}$ and $\xi_{ab}$ to be defined only if both $a,b$ are open or both are closed.

In the same way, we obtain the definition of twisted $2$-coloured modular operad.

Now we define the $2$-coloured twisted modular endomorphism operad:
Let $A_{\uo}\oplus A_{\uc}$ be an abbreviation for the direct sum of dg symplectic vector spaces $(A_{\uo},d_{\uo},\omega_{\uo})$ and $(A_{\uc},d_{\uc},\omega_{\uc})$.
Let 
$$\oEnd{A_{\uo}\oplus A_{\uc}}(O,C,G) := \op{Hom}_{\fld}(\bigotimes_O A_{\uo} \ot \bigotimes_C A_{\uc},\fld).$$
There are homogeneous bases $\{a_i^{\uo}\},\{b_i^{\uo}\}$ on $A_{\uo}$ related by the obvious analogue of \eqref{EQDefOmegaInverse} and similarly $\{a_i^{\uc}\},\{b_i^{\uc}\}$ on $A_{\uc}$.
Then $\ooo{a}{b}$ and $\xi_{ab}$ are defined, analogously to \eqref{EQoo} and \eqref{EQxi}, using the ${\uo}$-indexed (resp. ${\uc}$-indexed) bases if $a,b$ are open (resp. closed).

Algebra over a $2$-coloured twisted modular operad is again defined by replacing $\cCo$ by $\cCo_2$ in Definition \ref{DEFAlgOverTwisted}.

The notion of Feynman transform of a $2$-coloured modular operad is defined using a suitable definition of $2$-coloured graphs.
We leave it to the reader to fill in the details.
\end{definition}

The analogue of Theorem \ref{LEMMAAlgOverFeynTrans} is:

\begin{theorem} \label{THMFTBiColour}
Algebra over the Feynman transform $\feyn{\oP}$ of a $2$-coloured modular operad $\oP$ on a dg symplectic vector space $A:=A_{\uo}\oplus A_{\uc}$ is uniquely determined by a collection
$$\set{\alpha(O,C,G):\oP(O,C,G)^{\#}\to\oEnd{A}(O,C,G)}{(O,C,G)\in\cCo_2}$$
of degree $0$ linear maps (no compatibility with differential on $\oP(O,C,G)^{\#}$!) such that 
\begin{gather*}
\oEnd{A}(\rho) \comp \alpha(O,C,G) = \alpha(O',C',,G) \comp \oP(\rho^{-1})^{\#} \textrm{ for any }\rho:(O,C,G)\to(O',C',G) \textrm{ in }\cCo_2, \\
	d \comp \alpha(O,C,G) = \alpha(O,C,G) \comp \partial_{\oP}^{\#} + \\
	+ (\xi_{ab})_{\oEnd{A}} \comp \alpha(O\sqcup\{a,b\},C,G-1) \comp (\xi_{ab})^{\#}_{\oP} \ + \\
	+ (\xi_{ab})_{\oEnd{A}} \comp \alpha(O,C\sqcup\{a,b\},G-1) \comp (\xi_{ab})^{\#}_{\oP} \ + \\
	+ \frac{1}{2} \sum_{\substack{O_1\sqcup O_2 = O \\ C_1\sqcup C_2=C \\ G_1+G_2=G}} \hspace{-1em} (\ooo{a}{b})_{\oEnd{A}} \!\comp\! \left( \alpha(O_1\sqcup\{a\},C_1,G_1)\ot\alpha(O_2\sqcup\{b\},C_2,G_2) \right) \!\comp\! (\oooo{a}{b}{O_1\sqcup\{a\},C_1,G_1}{O_2\sqcup\{b\},C_2,G_2})^{\#}_{\oP} + \\
	+ \frac{1}{2} \sum_{\substack{O_1\sqcup O_2 = O \\ C_1\sqcup C_2=C \\ G_1+G_2=G}} \hspace{-1em} (\ooo{a}{b})_{\oEnd{A}} \!\comp\! \left( \alpha(O_1,C_1\sqcup\{a\},G_1)\ot\alpha(O_2,C_2\sqcup\{b\},G_2) \right) \!\comp\! (\oooo{a}{b}{O_1,C_1\sqcup\{a\},G_1}{O_2,C_2\sqcup\{b\},G_2})^{\#}_{\oP}.
\end{gather*}
\end{theorem}


\subsection{\texorpdfstring{The modular operad $\oQOC$}{The modular operad QOC}}

The idea is that the $2$-coloured modular operad $\oQOC$ consists of labeled $2$-dimensional orientable surfaces with both open and closed ends, as explained for $\oQC$ and $\oQO$.
We are only allowed to glue two open ends together or two closed ends together.

\begin{definition}
\begin{multline}
\oQOC(O,C,G) := \Span_{\fld} \big\{ \{\co_1,\ldots,\co_b\}^g_C \ |\ b\in\N_0,\ g\in\N_0,\ \co_i\textrm{'s are cycles in }O,\ \bigsqcup_{i=1}^b\co_i=O,\\
G=2g+b+\frac{|C|}{2}-1 \big\}
\end{multline}
where $\{\co_1,\ldots,\co_b\}^g_C$ is a symbol of degree $0$.
The geometric interpretation of the element $\{\co_1,\ldots,\co_b\}^g_C$ is the following:
It is the homeomorphism class of a surface with boundary such that each boundary component is either a closed end or corresponds to one of the cycles $\co_1,\ldots,\co_b$.
The closed ends are labeled by the set $C$.
The boundary component corresponding to a cycle $\co_k$ contains exactly $|\co_k|$ open ends labeled by the elements of $\co_k$.
The pictures are combinations of those seen for operads $\oQC$ and $\oQO$.

Recall that $\oQOC(O,C,G)$ is defined only if the stability condition $2(G-1)+|O|+|C|>0$ is met.
Equivalently, this is
\begin{gather*}
4g+2b+2|C|+|O|-4>0.
\end{gather*}
For bijections $\rho_{\uo}:O\xrightarrow{\sim}O'$ and $\rho_{\uc}:C\xrightarrow{\sim}C'$, let 
$$\oQO(\rho_{\uo},\rho_{\uc})(\{\cc_1,\ldots,\cc_b\}^g_C) := \{\rho_{\uo}(\cc_1),\ldots,\rho_{\uo}(\cc_b)\}^g_{\rho_{\uc}(C)}.$$

Assume $\co_i=\cyc{a,x_1,\ldots,x_m}$ is a cycle in $O_1\sqcup\{a\}$ and let $\co'_j=\cyc{b,y_1,\ldots,y_n}$ be a cycle in $O_2\sqcup\{b\}$.
Then the operadic composition along open ends is defined by
\begin{gather*}
\ooo{a}{b}(\{\co_1,\ldots,\co_{b_1}\}^{g_1}_{C_1} \ot \{\co'_1,\ldots,\co'_{b_2}\}^{g_2}_{C_2}) := \nonumber \\
\{\cyc{x_1,\ldots,x_m,y_1,\ldots,y_m},\co_1,\ldots,\widehat{\co_i},\ldots,\co_{b_1},\co'_1,\ldots,\widehat{\co'_j},\ldots,\co'_{b_2}\}^{g_1+g_2}_{C_1\sqcup C_2}.
\end{gather*}
The operadic composition along closed ends is defined by
\begin{gather}
\ooo{a}{b}(\{\co_1,\ldots,\co_{b_1}\}^{g_1}_{C_1\sqcup\{a\}} \ot \{\co'_1,\ldots,\co'_{b_2}\}^{g_2}_{C_2\sqcup\{b\}}) := \nonumber \\
\{\co_1,\ldots,\co_{b_1},\co'_1,\ldots,\co'_{b_2}\}^{g_1+g_2}_{C_1\sqcup C_2}.
\end{gather}

The operadic contraction along open ends is defined as follows:
If there are $i<j$ such that $\co_i=\cyc{a,x_1,\ldots,x_m}$ and $\co_j=\cyc{b,y_1,\ldots,y_n}$, then define
\begin{gather*}
\xi_{ab}(\{\co_1,\ldots,\co_{b}\}^{g}_C) := \{\cyc{x_1,\ldots,x_m,y_1,\ldots,y_m},\co_1,\ldots,\widehat{\co_i},\ldots,\widehat{\co_j},\ldots,\co_{b}\}^{g+1}_C.
\end{gather*}
Otherwise, there is $i$ such that $\co_i=\cyc{a,x_1,\ldots,x_m,b,y_1,\ldots,y_m}$ and then
\begin{gather*}
\xi_{ab}(\{\co_1,\ldots,\co_{b}\}^{g}_C) := \{\cyc{x_1,\ldots,x_m},\cyc{y_1,\ldots,y_m},\co_1,\ldots,\widehat{\co_i},\ldots,\co_{b}\}^{g}_C.
\end{gather*}
The operadic contraction along closed ends is defined by
$$\xi_{ab}(\{\co_1,\ldots,\co_{b}\}^{g}_{C\sqcup\{a,b\}}) := \{\co_1,\ldots,\co_{b}\}^{g+1}_{C}.$$
\end{definition}
 

The reader can now verify:

\begin{theorem}
$\oQOC$ is a $2$-coloured modular operad. \qed
\end{theorem}


\subsection{\texorpdfstring{Quantum open-closed homotopy algebra}{Quantum open-closed homotopy algebra}}

\begin{definition}
A quantum open-closed homotopy algebra is an algebra over $\feyn{\oQOC}$.
\end{definition}

The quantum open-closed homotopy algebras play the role in open-closed string theory similar to that of loop homotopy algebras (algebras over $\feyn{\oQC}$) and quantum $A_\infty$ algebras (algebras over $\feyn{\oQO}$) in closed and open string theory respectively.
They correspond to the solutions of the consistency conditions on string vertices.
The methods explained in this article allow making these consistency condition explicit in two ways:

First, we can make the axioms of algebras over $\feyn{\oQOC}$ explicit in the spirit of Sections \ref{SECTLoopHomAlg} and \ref{SECAxiomsForQuantumAoo} using Theorem \ref{THMFTBiColour}.
In this case, the algebra is given in terms of generating operations and relations, which directly reflect the combinatorics of the \emph{dual} of $\oQOC$.
A sufficiently patient reader can do this on his own.

Second, we can characterize the algebras over $\feyn{\oQOC}$ as solutions of certain master equation using an analogue of Barannikov's theory of Section \ref{SSECBarannikov}.
The master equation directly reflects the combinatorics of $\oQOC$, hence the results are somewhat more intuitive then in the first case.
Adapting the theory to the $2$-coloured case is easy and we just state the result, leaving details to the reader.

\medskip

The master equation $dS+\Delta S+\frac{1}{2}\{S,S\} = 0$ is solved in the space
$$\tilde{P} = \prod_{n,m,G} \oQOC([n],[m],G) \ot_{\Sigma_n\times\Sigma_m} \left( {A_{\uo}^\#}^{\ot n}\ot{A_{\uc}^\#}^{\ot m} \right).$$
Elements of $\tilde{P}$ are conveniently denoted as in \eqref{EQHerbstPhi}:
We fix a representative
$$\overline{(b_k)}^g_m := \{\underbrace{\emptyset,\emptyset,\ldots,\emptyset}_{b_0},\underbrace{\cyc{1},\cyc{2},\ldots,\cyc{b_1}}_{b_1},\cyc{b_1+1,b_1+2},\cyc{b_1+3,b_1+4},\ldots\}^{g}_{[m]}$$
of each $\Sigma_n\times\Sigma_m$ orbit in $\oQOC([n],[m],G)$.
Let $l_1,\ldots,l_{\ob}$ be the lengths of nonempty cycles in the order above.
For $1\leq k\leq\ob$, let $I^k\in[\dim A_{\uo}]^{\times l_k}$ be multiindices.
Let $J\in[\dim A_{\uc}]^{\times m}$ be a multiindex.
Then define
\begin{gather}
\phi^{I^1|\cdots|I^{\ob};J;g,b} := \overline{(b_k)}^g_m \ot_{\Sigma_n\times\Sigma_m} \beta_{(l_1,\ldots,l_{\ob})}\phi^{I^1\cdots I^{\ob}}\ot\phi^J, 
\end{gather}
where $\beta_{(l_1,\ldots,l_{\ob})}\in\Sigma_n$ has the same meaning as in \eqref{EQHerbstPhi}.
The stability condition further requires $2(2g+b-2+\frac{m}{2})+m+\sum_{k=1}^{\overline{b}}l_k > 0$.
Every element of $\tilde{P}$ can be written in the form
\begin{gather} \label{EQGenFunctQOC}
\sum_{\substack{g,\ob\\ I^1,\ldots,I^\ob,J}} \frac{1}{m!\ob!l_1\cdots l_\ob} C_{I^1|\cdots|I^\ob;J}^{g,b}\phi^{I^1|\cdots|I^\ob;J;g,b}
\end{gather}
for some coefficients $C_{I^1|\cdots|I^\ob;J}^{g,b}\in\fld$.
Moreover, if we require these coefficients to be graded symmetric in $J$'s and posses symmetries \eqref{EQSymmetriesOfCS} in $I$'s, then this expression is unique.

Let $J=(J_1,\ldots,J_m)$.
Then the BV operations are as follows:
\begin{gather*}
\Delta \phi^{I^1|\cdots|I^{\overline{b}};J;g,b} = \\
= 2\sum_{\substack{p<q\\ i,j}} \pm \omega^{I^p_iI^q_j} \phi^{ I^p_{i+1}\cdots I^p_{l_p} I^p_{1}\cdots I^p_{i-1} I^q_{j+1}\cdots I^q_{l_q} I^q_1\cdots I^q_{j-1} | I^1|\cdots|\widehat{I^p}|\cdots|\widehat{I^q}|\cdots|I^{\overline{b}} ;J ;g+1,b-1 } +{} \\
{}+ 2\sum_{\substack{p\\ i<j}} \pm \omega^{I^p_iI^q_j} \phi^{I^p_{i+1}\cdots I^p_{j-1} | I^p_{j+1}\cdots I^p_{l_p}I^p_1\cdots I^p_{i-1} | I^1|\cdots|\widehat{I^p}|\cdots|I^{\overline{b}} ;J ;g,b+1} +{} \\
{}+ 2\sum_{i<j} \pm \omega^{J_iJ_j} \phi^{I^1|\cdots|I^\ob;J_1\cdots\widehat{J_i}\cdots\widehat{J_j}\cdots J_m ;g+1,b},
\end{gather*}
where the signs are determined as in \eqref{EQQOCDeltaExplicit} and \eqref{EQQCDeltaExplicit}.
\begin{gather*}
\{ \phi^{I^1|\cdots|I^{\ob_1};J;g_1,b_1} , \phi^{K^1|\cdots|K^{\ob_2};L;g_2,b_2} \} = \\
= \sum_{p,q,i,j} \pm \omega^{I^p_iK^q_j} \phi^{ I^p_{i+1}\cdots I^p_{l_p}I^p_1\cdots I^p_{i-1}K^q_{j+1}\cdots K^q_{m_q}K^q_1\cdots K^q_{j-1}|I^1|\cdots|\widehat{I^p}|\cdots|I^{\ob_1}|K^1|\cdots|\widehat{K^q}|\cdots|K^{\ob_2};J,L;g_1+g_2;b_1+b_2-1 } +{} \\
{}+ \sum_{i,j} \pm \omega^{J_iL_j} \phi^{ I^1|\cdots|I^{\ob_1}|K^1|\cdots|K^{\ob_2};J_1\cdots\widehat{J_i}\cdots J_{m_1}L_1\cdots\widehat{L_j}\cdots L_{m_2};g_1+g_2;b_1+b_2 }
\end{gather*}
where the signs are determined as in \eqref{EQQOCBrackerExplicit} and \eqref{EQQCBracketExplicit}.

The generating function \eqref{EQGenFunctQOC} can be expressed in terms of string vertices, where we allow also for empty multiindices $I'^i$ as in \eqref{EQBigF1}:
By defining
$$\mathcal{V}^{g,b}_{I'^1|\cdots|I'^b;J} = -\frac{b_0!}{2} C^{g,b}_{I^1|\cdots|I^\ob;J},$$
we obtain
$$S = -2 \sum_{n,m,b,g} \sum_{I'^1,\ldots,I'^b,J} \frac{1}{m!b!\prod'_k l'_{k}} \mathcal{V}^{g,b}_{I'^1|\cdots|I'^b;J} \ \phi^{I'^1|\cdots|I'^b;J;g,b}.$$
This is precisely the physics expression for the open-closed string field theory action, cf. \cite{ZwiebachOpen-closed}, \cite{KajiuraOpen-closed2} or \cite{qocha}.


\subsection{\texorpdfstring{Stability}{Stability}}

In this section we discuss the stability condition \eqref{EQTwoColorStab}.

Let's list the $2$-dimensional surfaces with open and closed ends which are unstable according to the above definition:
\begin{enumerate}
	\item $g=1,\ b=0,\ |C|=0,\ |O|=0$ ($G=1$) : $\{\}^1_\emptyset$, torus.
	\item $g=0,\ b=2,\ |C|=0,\ |O|=0$ ($G=1$) : $\{\emptyset,\emptyset\}^0_\emptyset$, sphere with $2$ empty boundaries.
	\item $g=0,\ b=1,\ |C|=1,\ |O|=0$ ($G=1/2$) : $\{\emptyset\}^0_{\{1\}}$, sphere with $1$ closed end and $1$ empty boundary.
	\item $g=0,\ b=1,\ |C|=0,\ |O|=0,1,2$ ($G=0$) : $\{\emptyset\}^0_{\emptyset}, \{\cyc{1}\}^0_{\emptyset}, \{\cyc{1,2}\}^0_{\emptyset}$, sphere with $1$ boundary with $0,1$ or $2$ open ends.
	\item $g=0,\ b=0,\ |C|=2,\ |O|=0$ ($G=0$) : $\{\}^0_{\{1,2\}}$, sphere with $2$ closed ends.
\end{enumerate}
Notice that the sphere with $0$ or $1$ closed ends are excluded since for these $G<0$.

In physics, it is desirable to include the surface $3.$ in $\oQOC(\emptyset,[1],1/2)$.
Call it $\overline{c}$.
But $\ooot{1}{1}(\overline{c}\ot\overline{c})$ is the surface $2.$, hence that closedness under the operadic compositions implies that $2.$ has to be included too.
It is easy to check that no more unstables are produced by gluing $2., 3.$ and stable surfaces.
Thus we obtain an extension $\oQOC'$ of $\oQOC$ such that $\oQOC'(\emptyset,\{c\},1/2)$ and $\oQOC'(\emptyset,\emptyset,1)$ are $1$-dimensional.
In an algebra $\alpha:\feyn{\oQOC'}\to\oEnd{A}$, $\alpha(\overline{c}):A_c\to\fld$ is a cohomology class: $0 = d(\alpha(\overline{c})) = \alpha(\overline{c}) \comp d_A$, since one easily verifies $\partial_{\feyn{\oQOC'}}(\overline{c}) = 0$.

The corollas $(\emptyset,\{c\},1/2)$ and $(\emptyset,\emptyset,1)$ are not stable, thus formally $\oQOC'$ is not a modular operad.
However, there is no problem:
The purpose of stability is to guarantee the finiteness of number of iso classes of graphs appearing in the Feynman transform (see Lemma $2.16$ in \cite{GetzlerModop}), i.e. finiteness of dimension of $\feyn{\oP}(n,G)$ for each $n,G$.
Allowing the corollas $(\emptyset,\{c\},1/2)$ and $(\emptyset,\emptyset,1)$, this property is easily seen to be still true.

This suggests that our notion of stability is unnecessarily strict and should be refined in any more systematic approach to modular operads.


\subsection{\texorpdfstring{Further questions}{Further questions}} \label{SECFQ}

Inspired by Theorems \ref{THMModCom} and \ref{CONModAss}, one asks whether $\oQOC$ is the modular envelope of its genus zero part.
Recall that the operadic genus $G$ of a surface $\{\cc_1,\cc_2,\ldots\}_C^g\in\oQOC(O,C,G)$ is given by the formula $G=2g+b-1+|C|/2$.
Hence
$$\oQOC(O,C,0)=\begin{cases} \Span_{\fld}\{\{\co\}^0_\emptyset\ |\ \co \textrm{ is a cycle in }O\textrm{ of length }|O|\} & \textrm{if } C=\emptyset \\ 0 & \textrm{otherwise} \end{cases}$$
Notice that the nontrivial part is a suboperad in open inputs isomorphic to $\oAss$.
Also notice that the surfaces $\emptyset_C^0$ for any $C$ (these would correspond to a suboperad in closed inputs isomorphic to $\oCom$) are not present (the stability condition removes $\emptyset_C^0$ for any $|C|\leq 2$).
Consequently, the modular envelope of $\oQOC(-,-,0)$ is trivial whenever $C\neq\emptyset$ and thus is not $\oQOC$.

However, the genus zero part of $\oQOC$ is not the right operad to consider.
For this result to have an interesting physical interpretation, we would like $\oQOC$ to be a modular envelope of an operad describing vertices of "classical limit" of Feynman diagrams, that is diagrams with no circles.
In the classical limit, the vertices are genus zero surfaces with any number of open and closed ends, but with \emph{at most one open boundary component}\footnote{Recall boundary components of surfaces are either closed ends or contain open ends (possibly no open end). The latter components are called open boundary components.}.
Thus we are led to consider the Open Closed operad $\oOC$
\begin{gather*}
\oOC(O,C) := \begin{cases} \Span_{\fld}\{\emptyset_C^0,\ \{\emptyset\}_C^0\} & \textrm{if } O=\emptyset \\ \Span_{\fld}\{\{\co\}_C^0\ |\ \co \textrm{ is a cycle in }O\textrm{ of length }|O|\} & \textrm{otherwise} \end{cases}
\end{gather*}
The operadic composition should be inherited from $\oQOC$ via the obvious map $\oOC\to\oQOC$.
The $\xi_{ab}$ composition is, of course, omitted.
The $\ooo{a}{b}$ composition is however not well defined, since $\ooo{a}{b}(\co_C \ot \co'_{C'}) \not\in \oOC$ in general.
Hence $\oOC$ is not a cyclic operad.
This problem was partially overcome in the work of Kajiura and Stasheff, e.g. \cite{KajiuraOpen-closed1}.
We briefly explain it here.
Observe that one can compose open ends arbitrarily in $\oOC$, but composition of closed ends is possible only if one of the composed surfaces has no open boundary.
One easily sees that this restriction is satisfied if we choose a distinguished end of each surface so that closed end is distinguished only if there is no open boundary component on the surface.
This way, $\oOC$ becomes a $2$-coloured \emph{non}cyclic operad.
In Appendix  of Kajiura and Stasheff's \cite{KajiuraOpen-closed1}, this operad is seen to be the Koszul dual of the $2$-coloured operad for a Leibniz pair.
This answer is unsatisfying, since in the absence of cyclicity, there seems to be no way to relate $\oOC$ with the \emph{modular} operad $\oQOC$ in terms of some "free" construction, e.g. modular envelope.

However, there are more ways to view $\oOC$.
E.g. it is an operadic module over the cyclic operad $\oCom$ and one can consider a variants of modular envelopes for modules.
We plan to address these questions in future.

\vspace{2cm}\noindent {\bf Acknowledgements}:
We would like to thank M. Markl, J. Pulmann and I. Sachs for many helpful discussions.  The research of M.D. was supported by GA\v CR P201/13/27340P, B.J. was supported by grant GA\v CR P201/12/G028, whereas K.M. was supported in parts by the DFG Transregional Collaborative Research Centre TRR 33 and the DFG cluster of excellence ``Origin and Structure of the Universe''. We also thank to DAAD (PPP) and ASCR \& MEYS (Mobility) for supporting our collaboration. Finally, B.J, wants to thank CERN for hospitality.


\end{document}